\documentclass[12pt]{amsart}
\usepackage{subfiles}
\usepackage{fullpage} 
\usepackage{graphicx} 
\usepackage{prefix}
\usepackage{amsmath}
\usepackage{amsfonts}
\usepackage{tikz-cd}
\usepackage{tfrupee}
\usepackage{braket} 

\title{Functoriality of the KSGNS Construction for  Intertwiners of Strict Positive $C^*$-Correspondences}
\author{Lucus Brady}

\address{Department of Mathematical Sciences\\Montana State University\\Bozeman, MT 59717}
\email{lucusbrady@montana.edu}

\author{Ryan Grady}

\address{Department of Mathematical Sciences\\Montana State University\\Bozeman, MT 59717}
\email{ryan.grady1@montana.edu}

\date{\today}

\begin{document}
	
	\maketitle
	\begin{abstract}
		
		\noindent We prove that the KSGNS construction can be viewed as an endofunctor on a category whose objects are strict positive $C^*$-correspondences from a fixed $C^*$-algebra and morphisms are given by intertwiners which account for automorphisms of the fixed $C^*$-algebra. Using this perspective, we provide a functorial perspective for strict positive equivariant $C^*$-correspondences of $C^*$-dynamical systems and show every strict positive equivariant $C^*$-correspondence of $C^*$-dynamical systems unitarily uniquely dilates under the KSGNS construction to an equivariant $C^*$-correspondence of the dynamical systems.
		
	\end{abstract}
	
	\tableofcontents
	
	\section{Introduction}
	
	The KSGNS construction\footnote{Named after Kasparov, Stinespring, Gelfand, Naimark, and Segal} provides a generalization of both the GNS construction and Stinespring Dilation Theorem to the setting of Hilbert modules, and enables one to form $C^*$-correspondences from suitable completely positive maps. Abstractly, the KSGNS construction provides an assignment of strict positive $C^*$-correspondences to $C^*$-correspondences. Due to this transformation of data, we seek to investigate in what sense this assignment is functorial.
	
	Before we can start investigating the functorial nature of the KSGNS construction, we first need a category encoding strict positive $C^*$-correspondence. Typically, correspondences for $C^*$-algebras are treated as morphisms rather than objects. For example, the Enchilada category \cite{echterhoff2005categoricalapproachimprimitivitytheorems, eryuzlu2023exactsequencesenchiladacategory} has objects as $C^*$-algebras, morphisms as unitary isomorphism classes of $C^*$-correspondences, and composition given by the balance/interior tensor product. The Enchilada category can be generalized to the Quesadilla semi-category by replacing morphisms with unitary isomorphism classes of strict positive $C^*$-correspondences \cite{brix2024morphismscuntzpimsneralgebrascompletely}. The Enchilada category is contained as a subcategory of Quesadilla, and the KSGNS construction can be viewed as retraction of Quesadilla to Enchilada \cite{brix2024morphismscuntzpimsneralgebrascompletely}.
	
	While the Enchilada category and Quesadilla semi-category provide a natural environment for studying various constructions/notions related to $C^*$-algebras such as Morita equivalence, cross products, and the Primsner-Cuntz construction, these categories do not provide a means for studying correspondences and morphisms between them. The closest one can get with either of these two categories is by keeping track of the intertwining data as 2-morphisms similar to what was done in \cite{Buss_2012}; however, this approach limits the type of intertwiners one can consider. For example, given two correspondences $(E_1, \phi_1)$ and $(E_2, \phi_2)$ from $A$ to $B$, we are unable to consider intertwiners from $(E_1, \phi_1) $ to $(E_2, \phi_2)$ which account for automorphisms of $A$ due to an ill-defined composition rule for automorphism on different $C^*$-algebras. Furthermore, this approach limits considering morphisms between correspondences with different domains and codomains.
	
	Traditionally, the varying structure for intertwiners we are attempting to consider such as automorphisms is not taken into account (e.g. representations of groups). However, the varying structure we are considering for $C^*$-algebras play an important role when viewing the algebras through the lens of quantum mechanics. To emphasize this point, let us suppose we have a $C^*$-algebra $A$ representing the algebra of observables for a quantum system. As states are examples of non-degenerate completely positive maps, then every state $\phi\in S(A)$ defines a strict positive $C^*$-correspondence from $A$ to $\mathbb{C}$. Applying the GNS construction yields a triple $(H_\phi, \pi_\phi, x_\phi)$ where $H_\phi$ is a complex Hilbert space and $\pi_\phi:A\to \mathcal{B}(H_\phi)$ is a cyclic representation of $A$ with unit cyclic vector $x_\phi$. Trivially, the identity map on $\mathbb{C}$ intertwines $\phi$ with itself. From the GNS construction, the identity map on $\mathbb{C}$ lifts to the identity map on $H_\phi$ which intertwines $\pi_\phi$ with itself.
	
	Now suppose we have additional knowledge of a $*$-automorphism $\alpha:A\to A$ for which $\phi(\alpha(a))=\phi(a)$ for all $a\in A$. If we take into account both $\alpha$ and $1_\mathbb{C}$, we can construct a non-trivial unitary $U\in \mathcal{B}(H_\phi)$ for which $U\circ\pi_\phi(a)=\pi_\phi(\alpha(a))\circ U$ for all $a\in A$. As the GNS construction provides a localization of the $C^*$-algebra framework of quantum mechanics to the Hilbert space framework, the dependence of the intertwiner for the localized system on automorphisms of $A$ suggest that we account for this information when discussing intertwiners for states. As the KSGNS construction is a generalization of the GNS construction, we find it appropriate to consider intertwiners for strict positive $C^*$-correspondences which account for automorphisms of the domain $C^*$-algebra.
	
	To overcome the structural issues and obtain a category which encodes intertwiners of correspondences over different domains and codomains, we introduce a new category, called $\textnormal{PosCor}$, which behaves as the vertical hom-category of a semi-double category formed on Quesadilla. In particular, one can define the semi-double category where objects are given by $C^*$-algebras, horizontal $1$-morphisms are given by non-degenerate $*$-algebra homomorphisms, vertical $1$-morphisms are given by strict positive $C^*$-correspondences, and the 2-morphism relating the horizontal and vertical morphisms is given by an adjointable operator satisfying some conditions. Note, this would form a \textit{semi}-double category as the Quesadilla category is a semi-category due to the lack of an identity morphism. By taking the vertical hom-category of this semi-double category, we obtain  $\textnormal{PosCor}$.

	As we are interested in $\textnormal{PosCor}$ rather than the entire semi-double category, we avoid constructing the semi-double category directly and instead construct $\textnormal{PosCor}$ directly in two stages. In the first stage, we define a category for fixed $C^*$-algebras $B$ called $\textnormal{PosCor}(-, B)$  where objects encode vertical morphisms of the double category for a fixed codomain $B$. In the second stage, we combine together the categories $\textnormal{PosCor}(-, B)$ as $B$ varies. To describe morphisms between objects in different categories, we utilize the interior tensor product. As the interior tensor product defines a functor on the categories $\textnormal{PosCor}(-, B)$ (Theorem \ref{Theorem: Tensor Functor}), we are able to utilize the functorial properties to form a well-defined composition rule leading to the construction of our desired category $\textnormal{PosCor}$.
	
	Using the category $\textnormal{PosCor}$, we show the KSGNS construction defines as an idempotent endofunctor on a subcategory $\textnormal{PosCor}_{*\textnormal{-iso}}$ where the horizontal morphisms for the domain $C^*$-algebra of the vertical morphisms are isomorphisms (Theorem \ref{Theorem: KSGNS Functor 2} and Theorem \ref{Theorem: Idempotency 2}). Due to how $\textnormal{PosCor}$ is constructed, we arrive at this assertion in two intermediate steps. In the first step, we show the KSGNS construction defines an idempotent endofunctor on the category $\textnormal{PosCor}_{*\textnormal{-iso}}(-, B)$ (Theorem \ref{Theorem: KSGNS Endofucntor} and Proposition \ref{Prop: Idempotency of KSGSN}). In the second step, we show the KSGNS endofunctors on the categories of $\textnormal{PosCor}_{*\textnormal{-iso}}(-, B)$ commute with the functors coming from the interior tensor product (Theorem \ref{Theorem: KSNGS commutes with ITP}). Combining these results together, we obtain the perspective of the KSGNS construction as a idempotent endofunctor on $\textnormal{PosCor}_{*\textnormal{-iso}}$.
	
	Due to the data of the morphisms in $\textnormal{PosCor}(-, B)$ and $\textnormal{PosCor}$, we are able to define natural topologies on the hom-sets of  both $\textnormal{PosCor}(-, B)$ and the subcategory  $\textnormal{PosCor}_{*\textnormal{-alg}}$ of $\textnormal{PosCor}$. Using these topologies, we discuss in what sense the maps on the hom-sets obtained from composition, KSGNS endofunctors, and the functors coming from the interior tensor product are continuous. Additionally, we show as an application that strict positive equivariant $C^*$-correspondences of $C^*$-dynamical systems can be viewed as topological unitary valued functors from $BG$ to $\textnormal{PosCor}_{*\textnormal{-alg}}$ (Proposition \ref{Theorem: ESPC as Functors}) as well as every strict positive equivariant $C^*$-correspondence between $C^*$-dynamical systems unitarily uniquely dilates under the KSGNS construction to an equivariant $C^*$-correspondence of the $C^*$-dynamical systems (Theorem \ref{Theorem: Dilation Result}). In the case the group for the $C^*$-dynamical systems is trivial, Theorem \ref{Theorem: Dilation Result} is simply the KSGNS construction. Therefore Theorem \ref{Theorem: Dilation Result} can viewed as a generalization of the KSGNS construction.
	
	Let us outline the layout of this paper. In section 2, we review the basics of completely positive maps and Hilbert modules, as well as the KSGNS construction and interior tensor product. We also provide a comment in the last subsection regarding notation for equivalence classes of simple tensors in both the interior tensor product and KSGNS construction. In section 3, we build the categories $\textnormal{PosCor}(-, B)$, $\textnormal{PosCor}$, and $\textnormal{PosCor}_{*\textnormal{-alg}}$. In section 4, we prove the KSGNS construction defines an idempotent endofunctors on $\textnormal{PosCor}_{*\textnormal{-iso}}(-, B)$ and $\textnormal{PosCor}_{*\textnormal{-iso}}$.
	
	In section 5, we define the topologies on the hom-sets of $\textnormal{PosCor}(-, B)$ and $\textnormal{PosCor}_{*\textnormal{-alg}}$. In this section, we also consider the extent for which composition, the KSGNS endofunctors, and the functors from the interior tensor product define continuous maps on the hom-sets. Finally, in section 6, we apply the framework and results to strict positive equivariant $C^*$-correspondences of $C^*$-dynamical systems.
	
	\subsection*{Acknowledgments}
	LB was sponsored in part by the Air Force Research Laboratory under Agreement Number FA8750-24-1-1019. The U.S.Government is authorized to reproduce and distribute reprints for government purposes not withstanding any copyright notation thereon. Any opinions, findings, and conclusions or recommendations expressed in this material are those of the authors and do not necessarily reflect the view of the funder. RG is supported by the Simons Foundation under Travel Support/Collaboration 9966728.
	
	\section{Preliminary}
	
	In this section we provide some background information on completely positive maps and Hilbert modules. For more information on completely positive maps see \cite{paulsen2002completely} or \cite{blackadar2006operator}. For more information on Hilbert modules see \cite{lance1995hilbert}, \cite{raeburn1998morita}, or \cite{blackadar2006operator}.
	
	\subsection{Completely positive maps}
	
	For a $C^*$-algebra $A$, we have for each $n\in\mathbb{N}$ the $*$-algebra of $n$-by-$n$ matrices with entries in $A$, denote by $M_n(A)$. Using any faithful representation of $A$,  the $*$-algebra $M_n(A)$ can be equipped with a norm making it into a $C^*$-algebra. Given a linear map $\phi:A\to B$ between $C^*$-algebras, we obtain an induced linear map $\phi_n:M_n(A)\to M_n(B)$ for each $n\in\mathbb{N}$; in particular, we have
	$$\phi_n:M_n(A)\to M_n(B)\quad\quad \phi_n([a^i_j])=[\phi(a^i_j)].$$
	If $\phi$ is a $*$-algebra homomorphism, then $\phi_n$ is as well.
	
	For any two $C^*$-algebras $A$ and $B$, a \textit{positive map} from $A$ to $B$ is a linear function $f:A\to B$ such that for all $a\in A$, $f(a^*a)\in B^+$. As every element of $A^+$ is of the form $a^*a$ for some $a\in A$, then we equivalently have that $f$ is positive if $f(A^+)\subset B^+$. While continuity of $f$ is not assumed, every positive map is bounded; in fact, if $A$ and $B$ are unital $C^*$-algebras, then $||f||=||f(1)||$ \cite{paulsen2002completely}. Furthermore, as positive maps preserve positive elements, it follows that every positive map preserves the $*$-operation.
	
	For any two $C^*$-algebras $A$ and $B$, a \textit{completely positive map} from $A$ to $B$ is a linear function $f:A\to B$ such that $f_n:M_n(A)\to M_n(B)$ is positive for every $n\in\mathbb{N}$. Examples of completely positive maps include $*$-algebra homomorphisms and positive maps whenever the domain or codomain is a unital commutative $C^*$-algebra \cite{paulsen2002completely}. Note, not every positive map is completely positive; for example, the transpose map on $M_2(\mathbb{C})$ defines a positive map which is not completely positive.
	
	\subsection{Hilbert modules}
	
	For a $C^*$-algebra $A$, an \textit{inner product $A$-module} is a complex vector space $E$ that is a right $A$-module together with a map $\langle\cdot, \cdot\rangle:E\times E\to A$ satisfying the following properties:
	\begin{enumerate}
		\item  $\langle x, x\rangle\geq 0$ for all $x\in E$, and $\langle x, x\rangle=0$ if and only if $x=0$. 
		\item for all $x, y, z\in E$ and $\lambda\in \mathbb{C}$, $\langle x, y+\lambda z\rangle=\langle x, y\rangle+\lambda\langle x, z\rangle$.
		\item for all $x, y\in E$, $\langle x, y\rangle=\langle y, x\rangle^*$
		\item for all $x, y\in E$ and $a\in A$, $\langle x, ya\rangle=\langle x, y\rangle a$.
		
	\end{enumerate}
	Using the norm on $A$, we can equip $E$ with a norm: $||x||=\sqrt{||\langle x, x\rangle||_A}$. If $E$ is a Banach space with respect to this norm, then $E$ is called a \textit{Hilbert $A$-module}.
	
	\begin{example}\

		\noindent A Hilbert $\mathbb{C}$-module is precisely a Hilbert space. 
		
	\end{example}
	
	\begin{example}\

		\noindent Let $A$ be a $C^*$-algebra. Viewing $A$ as a right $A$-module, we define the pairing
		$$\langle\cdot, \cdot\rangle:A\times A\to A\quad\quad \langle a, b\rangle=a^*b$$
		The map $\langle\cdot, \cdot\rangle$ makes $A$ into an inner product $A$-module whose induced norm is exactly the norm on $A$. As $A$ is Banach space with respect to its norm, then $A$ is a Hilbert $A$-module.
		
	\end{example}
	
	\begin{example}\

		\noindent Let $X$ be a compact Hausdorff topological space, and let $\pi:E\to X$ be a finite rank complex vector bundle with Hermitian metric $g$. The space of continuous sections $\Gamma(X, E)$ forms a $C(X)$-module under pointwise multiplication. Using the metric $g$, we define
		$$\tilde{g}:\Gamma(X, E)\times \Gamma(X, E)\to C(X)\quad\quad \tilde{g}(s_1, s_2)(x)=g_x(s_1(x), s_2(x))$$
		Using properties of $g$, it follows that $\Gamma(X, E)$ is an inner product $C(X)$-module with respect to $\tilde{g}$. Using local trivializations of the vector bundle and that each fiber is a Hilbert space,  $\Gamma(X, E)$ is complete with respect to the induced norm from $\tilde{g}$ making it into a Hilbert $C(X)$-module.
		
	\end{example}
	
	\begin{example}\

		\noindent Let $E$ be a Hilbert $A$-module, then $E^n$ is a Hilbert $A$-module where $A$ acts on each summand and the $A$-valued pairing is given by
		$$\left\langle \begin{bmatrix} x_1\\  \vdots\\ x_n \end{bmatrix}, \begin{bmatrix} y_1\\  \vdots\\ y_n \end{bmatrix}\right\rangle=\sum\limits_{k=1}^n \langle x_k, y_k\rangle_{E}.$$
		Alternatively, we can view $E^n$ as a Hilbert $M_n(A)$-module where the right action of $M_n(A)$ is given by
		$$\begin{bmatrix}
			x_1 & \ldots & x_n
		\end{bmatrix}[a^i_j]=\begin{bmatrix}
			\sum\limits_{i=1}^n x_ia^i_1 & \ldots &\sum\limits_{i=1}^n x_ia^i_n
		\end{bmatrix}$$
		and the $M_n(A)$-valued inner product on $E^n$ is given by
		$$\left\langle \begin{bmatrix}
			x_1 & \ldots & x_n
		\end{bmatrix}, \begin{bmatrix}
			y_1 & \ldots & y_n
		\end{bmatrix}\right\rangle=[\langle x_i, y_j\rangle_E]$$
		
	\end{example}
	
	Let $E$ and $F$ be Hilbert $A$-modules. A function $f:E\to F$ is adjointable if there exists a function $f^*:F\to E$, called the adjoint of $f$, such that for all $x\in E$ and $y\in F$, $\langle f(x), y\rangle_F=\langle x, f^*(y)\rangle_E$. It follows that every adjointable map is $A$-linear, has a unique adjoint, and is bounded. We denote the set of all adjointable $A$-linear maps from $E$ to $F$ as $\mathcal{L}(E, F)$. In the case $E=F$,  $\mathcal{L}(E):=\mathcal{L}(E, E)$ forms a unital $C^*$-algebra with respect to the usual operator norm.  
	
	Given a Hilbert $A$-module $E$, we can identify a subspace of $\mathcal{L}(E)$ which is similar to compact operators on a Hilbert space. For each $x, y\in E$, let
	$$\theta_{x, y}:E\to E\quad\quad \theta_{x, y}z=x\langle y, z\rangle$$
	The map $\theta_{x, y}$ is $A$-linear and adjointable with $\theta^*_{x, y}=\theta_{y, x}$. Let $\mathcal{K}(E)$ be the closed linear span of $\{\theta_{x, y}: x, y\in E\}$ in $\mathcal{L}(E)$, then $\mathcal{K}(E)$ is a closed essential ideal in $\mathcal{L}(E)$. In the case $E$ is an actual Hilbert space, $\mathcal{K}(E)$ is precisely the algebra of compact operators as an application of the Spectral Theorem for Compact Operators.  Furthermore, in the case $E=A$, then $\mathcal{K}(E)$ can be identified with $A$ and $\mathcal{L}(E)$ with the multiplier algebra of $A$.
	
	Similar to bounded operators on a Hilbert space, there is an alternative topology on $\mathcal{L}(E)$ we will work with called the strict topology. The \textit{strict topology}\footnote{This topology is more commonly called the $*$-strong topology; however, we follow the terminology convention of this topology given in \cite{lance1995hilbert}.} is generated by the semi-norms
	$$\{p_x:\mathcal{L}(E)\xrightarrow{T\mapsto ||Tx||} [0, \infty)\}_{x\in E}\cup \{p_x^*:\mathcal{L}(E)\xrightarrow{T\mapsto ||T^*x||} [0, \infty)\}_{x\in E}$$
	and makes $\mathcal{L}(E)$ into  a locally convex Hausdorff topological vector space. Therefore, a net $(T_\lambda)_{\lambda\in \Lambda}$ in $\mathcal{L}(E)$ converges to $T$ in the strict topology if and only if for all $x\in E$ both $T_\lambda x\to T x$ and $T_\lambda^* x\to T^*x$ with respect to the norm on $E$. From this characterization, we see the strong operator topology is weaker than the strict operator topology; however, both topologies agree on norm bounded subsets of $\mathcal{L}(E)$ \cite{raeburn1998morita}.
	
	\subsection{Interior tensor product of Hilbert modules}
	
	Let $E$ be a Hilbert $A$-module, let $F$ be a Hilbert $B$-module, and let $\pi:A\to \mathcal{L}(F)$ be a $*$-algebra homomorphism. Using $\pi$, we view $F$ as a left $A$-module where given $y\in F$ and $a\in A$, $a\cdot y:=\pi(a)y$. Using this module structure, we construct the tensor product of $E$ and $F$ over $A$: $E\otimes_AF$. Recall, $E\otimes_A F$ is constructed by starting with the algebraic tensor product of the underlying complex vector spaces for $E$ and $F$, denote by $E\otimes_{\textnormal{alg}}F$, and quotienting out by the submodule $N$ generated by the set $\{xa\otimes y-x\otimes \pi(a)y: x\in E, y\in F, a\in A\}$.
	
	On $E\otimes_A F$ we define the pairing
	$$\langle\cdot, \cdot\rangle:(E\otimes_A F)^2\to B\quad\quad \langle x_1\otimes y_1+N, x_2\otimes y_2+N\rangle=\langle y_1, \pi(\langle x_1, x_2\rangle_E)y_2\rangle_F$$
	which yields a $B$-valued inner product. We point out that a non-trivial argument is required to conclude  $\langle\cdot, \cdot\rangle$ is non-degenerate (see chapter 4 of \cite{lance1995hilbert} for details). With respect to the induced norm, we take the completion of $E\otimes_A F$, denote by $E\otimes_\pi F$, to obtain a Hilbert module called the interior tensor product of $E$ and $F$. 
	
	Given an adjointable operator $T\in \mathcal{L}(E)$, we define 
	$$T\otimes I:E\otimes_{A}F\to E\otimes_{\pi}F\quad\quad (T\otimes I)(x\otimes y+N)=Tx\otimes y+N$$
	which is a well-defined, bounded, $B$-linear map. Therefore we can uniquely extend $T\otimes I$ to a bounded $B$-linear map on $E\otimes_\pi A$ which we denote by $T\otimes_\pi I$. Observe $T\otimes_\pi I$ is adjointable with adjoint $T^*\otimes_\pi I$. As we vary over $\mathcal{L}(E)$ we obtain a unital $*$-algebra homomorphim from $\mathcal{L}(E)$ to $\mathcal{L}(E\otimes_\pi F)$ which is continuous on the unit ball in $\mathcal{L}(E)$ with respect to the strict topology on both $\mathcal{L}(E)$ and $\mathcal{L}(E\otimes_\pi F)$.
	
	\subsection{KSGNS construction}
	
	Let $E$ be a Hilbert $A$-module, and let $F$ be a Hilbert $B$-module. If one replaces the $*$-algebra homomorphism $\pi:A\to \mathcal{L}(F)$ in the interior tensor product construction with a completely positive map $\phi:A\to \mathcal{L}(F)$, the construction can still be completed with some modifications. Let us briefly review this construction, details can be found in chapter 5 of \cite{lance1995hilbert}. 
	
	As before, we begin with the algebraic tensor product $E\otimes_{\textnormal{alg}}F$. We make $E\otimes_{\textnormal{alg}}F$ into a right $B$-module where $B$ acts on $F$. On this $B$-module, we define the $B$-valued pairing
	$$\langle\cdot, \cdot\rangle_\phi:(E\otimes_{\textnormal{alg}}F)^2\to B\quad\quad\langle x_1\otimes y_1, x_2\otimes y_2\rangle_\phi=\langle y_1, \phi(\langle x_1, x_2\rangle_E)y_2\rangle_F$$
	which satisfies all the conditions for making $E\otimes_{\textnormal{alg}}F$ into an inner product $B$-module except potentially non-degeneracy; that is, there may exist $z\in E\otimes_{\textnormal{alg}}F$ with $\langle z, z\rangle_\phi=0$ though $z\neq 0$. Let $N_\phi$ be the submodule of $E\otimes_{\textnormal{alg}}F$ for which $\langle z, z\rangle_\phi=0$.
	
	Pushing $\langle\cdot, \cdot\rangle_\phi$ onto the quotient $(E\otimes_{\textnormal{alg}}F)/N_\phi$ makes the quotient $B$-module into an inner product $B$-module. Completing the inner product $B$-module with respect to the induced norm yields a Hilbert $B$-module which we denote by $E\otimes_\phi F$. Similar to the interior tensor product, we obtain a unital $*$-algebra homomorphism
	$$\tilde{\pi}_\phi:\mathcal{L}(E)\to \mathcal{L}(E\otimes_\phi F)\quad\quad T\mapsto T\otimes_\phi I$$
	which is continuous on the unit ball in $\mathcal{L}(E)$ with respect to the strict topology on both $\mathcal{L}(E)$ and $\mathcal{L}(E\otimes_\phi F)$. 
	
	
	In the case $E=A$,  the completely positive map $\phi$ being strict yields a uniqueness for both $A\otimes_\phi F$ and $M(A)\to \mathcal{L}(A\otimes_\phi E)$ which is similar to the GNS construction and Stinespring Dilation Theorem. A completely positive map $\phi:A\to \mathcal{L}(F)$ is \textit{strict} if there exists an approximate unit $(a_\lambda)_{\lambda\in\Lambda}$ in $A$ for which $(\phi(a_\lambda))_{\lambda\in \Lambda}$ is a Cauchy net with respect to the strict topology on $\mathcal{L}(F)$. Equivalently, $\phi:A\to \mathcal{L}(F)$ is strict if there exists a completely positive map $\tilde{\phi}:M(A)\to \mathcal{L}(F)$ which is continuous on the unit ball with respect to the strict topology on both $M(A)$ and $\mathcal{L}(F)$ as well as restricts to $\phi$ on $A\subset M(A)$ \cite{lance1995hilbert}. If $A$ is unital, we can always take the constant net at $\phi(1)$ making the condition of being strict always satisfied. 
	
	If $\phi:A\to \mathcal{L}(F)$ is a strict completely positive map, then we use the net $(a_\lambda)_{\lambda\in\Lambda}$ to define for each $\lambda\in \Lambda$
	$$V_\lambda:F\to A\otimes_\phi F\quad\quad V_\lambda y=a_\lambda\otimes y+N_\phi$$
	which is a well-defined, $B$-linear map. As $(\phi(a_\lambda))_{\lambda\in\Lambda}$ is a Cauchy net with respect to the strict topology, then $(V_\lambda y)_{\lambda\in \Lambda}$ is a Cauchy net in $A\otimes_\phi F$ with respect to the norm on $A\otimes_{\phi} F$. As $A\otimes_{\phi}F$ is complete, the net converges to a point $V_\phi y$. Varying over $F$ yields a map $V_\phi \in \mathcal{L}(F, A\otimes_\phi F)$ for which $V_\phi y=\lim\limits_{\lambda\to \infty }V_\lambda y$ for all $y\in F$. Observe 
	$$V_\phi^*:A\otimes_\phi F\to F\quad\quad V_\phi^*(a\otimes y+N_\phi)=\phi(a)y$$
	
	While the net $(V_\lambda)_{\lambda\in\Lambda}$ is dependent on the approximate unit $(a_\lambda)_{\lambda\in\Lambda}$, the limit $V_\phi$ is independent of the approximate unit. Indeed, if $(a_\mu)_{\mu\in M}$ is another approximate unit, then using $\tilde{\phi}:M(A)\to \mathcal{L}(F)$ shows $(\phi(a_\mu))_{\mu\in M}$ converges to $\tilde{\phi}(1)$ in the strict topology (thus is Cauchy). Therefore, we can construct a net $(V_\mu)_{\mu\in M}$ and function $V'\in \mathcal{L}(F, A\otimes_\phi F)$ which is the strong limit of $(V_\mu)_{\mu\in M}$. From the observation at the end of the last paragraph, we have $V'^*=V_\phi^*$. Thus, by uniqueness of adjoints, $V'=V_\phi$.
	
	Since $\tilde{\pi}_\phi:M(A)\to \mathcal{L}(A\otimes_\phi F)$ is a unital $*$-algebra homomorphism which is strictly continuous on the unit ball in $M(A)$, then the restriction of $\tilde{\pi}_\phi$ to $A\subset M(A)$ determines a non-degenerate $*$-algebra homomorphism, denote by $\pi_\phi$. With respect to $V_\phi$ and $\pi_\phi$, we have $V_\phi^*\pi_\phi(a)V_\phi=\phi(a)$ for all $a\in A$ and the submodule $\pi_\phi(A)V_\phi F\subset A\otimes_\phi F$ defined as
	$$\pi_\phi(A)VF=\textnormal{Span}(\{\pi_\phi(a)V_\phi y: a\in A, y\in F\})$$
	is dense in $A\otimes_\phi F$. As in the GNS construction and Stinespring Dilation Theorem, the triple $(A\otimes_\phi F, \pi_\phi, V_\phi)$ is unitarily uniquely characterized by these last two properties. 
	
	\begin{theorem}[KSGNS Construction \cite{lance1995hilbert}]\
		
		\noindent Let $A$ and $B$ be $C^*$-algebras. Let $F$ be a Hilbert $B$-module, and let $\phi:A\to \mathcal{L}(F)$ be a strict completely positive map. Then there exists a triple $(F_\phi, \pi_\phi, V_\phi)$ where $F_\phi$ is a Hilbert $B$-module, $\pi_\phi\in \textnormal{Hom}_{*\textnormal{-alg}}(A, \mathcal{L}(F_\phi))$ is non-degenerate, and $V_\phi\in \mathcal{L}(F, F_\phi)$, and the triple satisfies the following conditions:
		\begin{enumerate}
			\item for all $a\in A$, $\phi(a)=V_\phi^*\pi_\phi(a)V_\phi$.
			\item the submodule $\pi_\phi(A)V_\phi F$ is dense in $F_\phi$.
		\end{enumerate}
		If $(F', \pi', V')$ is another triple satisfying the two conditions, then there exists a $B$-linear unitary $U\in \mathcal{L}(F_\phi, F')$ such that for all $a\in A$, $\pi'(a)=U\pi_\phi(a)U^*$ and $V'=UV_\phi$.
		
	\end{theorem}

	\subsection{Notation}
	
	For much of the paper, we will utilize the density of $E\otimes_AF\subset E\otimes_\pi F$ and $\frac{A\otimes_{\textnormal{alg}}F}{N_\phi}\subset A\otimes_\phi F$ in conjunction with linearity and continuity to reduce many computations to elements of the form $x\otimes y+N\in E\otimes_AF$ and $a\otimes x+N_\phi \in \frac{A\otimes_{\textnormal{alg}}F}{N_\phi}$. As a result, carrying the "$+N$" and "$+N_\phi$" becomes quite cumbersome as well as unwieldly when we start to consider multiple tensor products. Therefore, we will follow the notation convention in \cite{lance1995hilbert} by placing a dot over the tensor symbol to denote the equivalence class of a simple tensor. For example, we write $a\otimes x+N_\phi$ as $a\dot{\otimes}x$ and $x\otimes y+N$ as $x\dot{\otimes}y$. As the notation for the two different cases is the same, we will be diligent and explicit throughout the paper as to where elements like $a\dot{\otimes}x$ and $x\dot{\otimes}y$ live. Additionally, we will denote $A\otimes_\phi F$ simply as $F_\phi$.
	
	\section{Constructing the Categories $\textnormal{PosCor}(-, B)$, $\textnormal{PosCor}$, and $\textnormal{PosCor}_{*\textnormal{-alg}}$}
	
	Our goal in this section is to construct the category $\textnormal{PosCor}$. This category encodes strict positive $C^*$-correspondences from any $C^*$-algebra to any other $C^*$-algebra as objects and intertwiners as morphisms. As the domain and codomain of the correspondences can be any $C^*$-algebra, trying to immediately construct the category becomes rather complex and unwieldy due to the data needed to define the morphisms and composition rule. 
	
	To rein in the complexity, we proceed in constructing $\textnormal{PosCor}$ in two steps. First, we fix the codomain $C^*$-algebra of the correspondences, and we build a category, denoted by $\textnormal{PosCor}(-, B)$, which encodes our desired intertwining data as the morphisms. The next step is to combine the categories $\textnormal{PosCor}(-, B)$ as $B$ varies. Combining the objects from the various categories is straightforward; the difficulty of this process appears when we go to describe morphisms from objects in $\textnormal{PosCor}(-, B)$ to objects in $\textnormal{PosCor}(-, C)$. The mechanism we will use to help in describing morphisms between objects in the different categories is the interior tensor product. By realizing the interior tensor product as a functor, we can easily navigate describing morphisms and a composition rule to arrive at our desired category $\textnormal{PosCor}$. In this section, we also introduce the category $\textnormal{PosCor}_{*\textnormal{-alg}}$ which is a subcategory of $\textnormal{PosCor}$ that we will utilize in Section 5 and Section 6.
	
	\subsection{The category $\textnormal{PosCor}(-, B)$}
	
	\begin{definition}\

		\noindent Let $B$ be a $C^*$-algebra. Define $\textnormal{PosCor}(-, B)$ as the category where
		\begin{itemize}
			\item objects are triples $(A, B, (E, \phi))$ where $A$ is a $C^*$-algebra and $(E, \phi)$ is a strict positive $C^*$-correspondences from $A$ to $B$.
			\item a morphism from $(A, B, (E_1, \phi_1))$ to $(C, B, (E_2, \phi_2))$ is given by the data $(\rho, \eta)$ where $\rho:A\to C$ is a non-degenerate $*$-algebra homomorphism and
			$\eta\in \mathcal{L}(E_1, E_2)$, and the data satisfies the following condition: for all $a\in A$
			$$\phi_2(\rho(a))\circ\eta=\eta\circ \phi_1(a).$$ 
			\item given morphisms $$(\rho, \eta):(A, B, (E_1, \phi_1))\to (C, B, (E_2 ,\phi_2))\quad\textnormal{and}\quad(\chi, \beta):(C, B, (E_2, \phi_2)\to (D, B, (E_3, \phi_3)),$$
			define the composition $(\chi, \xi)\circ (\rho, \eta)$ as the morphism $(\chi\circ \rho, \xi\circ\eta)$.
		\end{itemize}
		
	\end{definition}
	
	As function composition is associative and as composition of non-degenerate $*$-algebra homomorphisms is again non-degenerate, it follows the composition rule for $\textnormal{PosCor}(-, B)$ is well-defined and associative. Furthermore, it is clear that for an object $(A, B, (E, \phi))$ in $\textnormal{PosCor}(-, B)$, the identity morphism is given by $(1_A, 1_E)$. Thus $\textnormal{PosCor}(-, B)$ does indeed form a category.  
	
	\subsection{Functoriality of the interior tensor product}
	
	Now that we have defined the categories $\textnormal{PosCor}(-, B)$, we look to combine them together to construct the category $\textnormal{PosCor}$. To see how we combine the categories, let us fix objects $(A, B, (E_B, \phi))\in\textnormal{PosCor}(-, B)$ and $(C, D, (E_D, \psi))\in\textnormal{PosCor}(-, D)$  . If $B=D$, then morphisms in $\textnormal{PosCor}(-, B)$ are sufficient for describing a morphism from $(A, B, (E_B, \phi))$ to $(C, D, (E_D, \psi))$. If $B\neq D$, then we need additional information to describe a morphism from $(A, B, (E_B ,\phi))$ to $(C, D, (E_D, \psi))$.
	
	A piece of additional information that we can utilize is a way to convert $E_B$ to a Hilbert $D$-module. In particular, if we have the data of a non-degenerate $*$-algebra homomorphism $\pi:B\to \mathcal{L}(F)$ for $F$ a Hilbert $D$-module, we can utilize the interior tensor product to form the Hilbert $D$-module $E_B\otimes_\pi F$ and strict completely positive map $\tilde{\phi}:=\phi\otimes_\pi I$ given by the composition
	$$A\xrightarrow{\phi}\mathcal{L}(E_B)\xrightarrow{T\mapsto T\otimes_\pi I}\mathcal{L}(E_B\otimes_\pi F).$$
	As $(A, D, (E_B\otimes_\pi F,\tilde{\phi}))$ is an object in  $\textnormal{PosCor}(-, D)$, we can utilize morphisms in $\textnormal{PosCor}(-, D)$ to describe a morphism from $(A, D, (E_B\otimes_\pi F,\tilde{\phi}))$ to $(C, D, (E_D, \psi))$. Thus the data of a morphism from $(A, B, (E_B, \phi))$ to $(C, D, (E_D, \psi))$ can be taken as a pair $((F, \pi), (\rho, \eta))$ where $(\rho, \eta)$ is a morphism in $\textnormal{PosCor}(-, D)$ from $(A, D, (E_B\otimes_\pi F, \tilde{\phi}))$ to $(C, D, (E_D, \psi))$.
	
	Due to the interior tensor product playing a role in the data of morphisms, we first determine in what sense the construction is functorial. Given a $C^*$-correspondence $(F, \pi)$ from $B$ to $D$, the interior tensor product provides an assignment on objects from $\textnormal{PosCor}(-, B)$ to $\textnormal{PosCor}(-, D)$. Thus, to realize the interior tensor product as a functor, we just need to determine an assignment at the level of morphisms and verify such an assignment is functorial.
	
	To determine the assignment on the morphisms, let $(\rho, \eta):(A, B,  (E_1, \phi_1))\to(C, B, (E_2, \phi_2))$ be a morphism in $\textnormal{PosCor}(-, B)$; then, we can look to define the map
	$$\hat{\eta}:E_1\otimes_BF\to E_2\otimes_\pi F\quad\quad \hat{\eta}(x\dot{\otimes}f)=\eta(x)\dot{\otimes}f$$
	Provided the map is well-defined, it is $D$-linear. Furthermore, if the map is bounded, we can extend it to all of $E_1\otimes_\pi F$. Defining a similar map for $\eta^*$ will show $\hat{\eta}$ is adjointable. From how $\hat{\eta}$, $\tilde{\phi}_1$, and $\tilde{\phi}_2$ are defined, it follows $(\rho, \hat{\eta})$ is a morphism from $(A, D, (E_1\otimes_\pi F, \tilde{\phi}_1))$ to $(C, D, (E_2\otimes_\pi F, \tilde{\phi}_2))$ in $\textnormal{PosCor}(-, D)$. Thus, to obtain our assignment at the level of morphisms, we need to verify $\hat{\eta}$ is well-defined and bounded. The next lemma will help in resolving these two items. Note, the following lemma is similar to Lemma 4.2 in \cite{lance1995hilbert}, though we include a proof here for completeness.
	
	\begin{lemma}\label{Lemma: Bound 1}\

		\noindent Let $B$ and $C$ be $C^*$-algebras, and let $(F, \pi)$ be a $C^*$-correspondence from $B$ to $C$. For $i=1, 2$, let $(A_i, B, (E_i, \phi_i))\in \textnormal{PosCor}(-, B)$. If $(\rho, \eta)$ is a morphism from $(A_1, B, (E_1, \phi_1))$ to $(A_2, B, (E_2, \phi_2))$, then
		\begin{enumerate}
			\item for all $x_1, \ldots, x_n\in E_1$ and $f_1, \ldots, f_n\in F$,
			$$\left\|\sum\limits_{i, j=1}^n \langle f_i, \pi(\langle \eta(x_i), \eta(x_j)\rangle_{E_2})f_j\rangle_F\right\|_C\leq ||\eta||^2\left\|\sum\limits_{i, j=1}^n \langle f_i, \pi(\langle x_i, x_j\rangle_{E_1})f_j\rangle_F\right\|_C$$
			\item for all $y_1, \ldots, y_n\in E_2$ and $f_1, \ldots, f_n\in F$,
			$$\left\|\sum\limits_{i, j=1}^n \langle f_i, \pi(\langle \eta^*(y_i), \eta^*(y_j)\rangle_{E_1})f_j\rangle_F\right\|_C\leq ||\eta||^2\left\|\sum\limits_{i, j=1}^n \langle f_i, \pi(\langle y_i, y_j\rangle_{E_2})f_j\rangle_F\right\|_C$$
		\end{enumerate}
		
	\end{lemma}
	
	\begin{proof}\
		
		\noindent Fix $n\in \mathbb{N}$. We view $E^n_i$ as a Hilbert $M_n(B)$-module and $F^n$ as a Hilbert $C$-module. Given $\eta:E_1\to E_2$, we obtain a $M_n(B)$-linear map $\tilde{\eta}:E_1^n\to E_2^n$ where
		$$\tilde{\eta}\left(\begin{bmatrix}
			x_1 & \ldots & x_n
		\end{bmatrix}\right)=\begin{bmatrix}
			\eta(x_1) & \ldots & \eta(x_n)
		\end{bmatrix}$$
		We observe $\tilde{\eta}$ is adjointable with respect to the $M_n(B)$-valued inner product with adjoint given by $\widetilde{\eta^*}$ as well as $||\tilde{\eta}||=||\eta||$. As $\tilde{\eta}^*\tilde{\eta}\leq ||\eta||^2 I_{E_1^n}$, then for all $x_1, \ldots, x_n\in E_1$
		\begin{align}
			0\leq \left\langle\tilde{\eta}\begin{bmatrix}
				x_1 & \ldots & x_n    
			\end{bmatrix}, \tilde{\eta}\begin{bmatrix}
				x_1 & \ldots & x_n    
			\end{bmatrix}\right\rangle_{E_2^n}\leq ||\eta||^2\left\langle\begin{bmatrix}
				x_1 & \ldots & x_n    
			\end{bmatrix}, \begin{bmatrix}
				x_1 & \ldots & x_n    
			\end{bmatrix}\right\rangle_{E_1^n}
		\end{align}
		Similarly, $\tilde{\eta}\tilde{\eta}^*\leq ||\eta||^2I_{E_2^n}$ so that for all $y_1, \ldots, y_n\in E_2$
		\begin{align}
			0\leq \left\langle\tilde{\eta}^*\begin{bmatrix}
				y_1 & \ldots & y_n    
			\end{bmatrix}, \tilde{\eta}^*\begin{bmatrix}
				y_1 & \ldots & y_n    
			\end{bmatrix}\right\rangle_{E_1^n}\leq ||\eta||^2\left\langle\begin{bmatrix}
				y_1 & \ldots & y_n    
			\end{bmatrix}, \begin{bmatrix}
				y_1 & \ldots & y_n    
			\end{bmatrix}\right\rangle_{E_2^n}
		\end{align}
		
		As $\pi:B\to \mathcal{L}(F)$ is a non-degenerate $*$-algebra homomorphism, then $\pi$ is a completely positive map. Therefore we obtain a positive map
		$$\pi_n:M_n(B)\to \mathcal{L}(F^n)\cong M_n(\mathcal{L}(F)) \quad\quad \pi_n([b^i_j])\begin{bmatrix}
			f_1\\
			\vdots\\
			f_n
		\end{bmatrix}=\begin{bmatrix}
			\sum\limits_{j=1}^n \pi(b^1_j)f_j\\
			\vdots\\
			\sum\limits_{j=1}^n \pi(b^n_j)f_j
		\end{bmatrix}$$
		Thus, using line $(1)$, we have for all $x_1, \ldots, x_n\in E_1$
		\begin{align*}
			0&\leq \pi_n\left(\left\langle\tilde{\eta}\begin{bmatrix}
				x_1 & \ldots & x_n    
			\end{bmatrix}, \tilde{\eta}\begin{bmatrix}
				x_1 & \ldots & x_n    
			\end{bmatrix}\right\rangle_{E_2^n}\right)\leq ||\eta||^2\pi_n\left(\left\langle\begin{bmatrix}
				x_1 & \ldots & x_n    
			\end{bmatrix}, \begin{bmatrix}
				x_1 & \ldots & x_n    
			\end{bmatrix}\right\rangle_{E_1^n}\right)
		\end{align*}
		Similarly, using line $(2)$, we have for all $y_1, \ldots, y_n\in E_2$
		\begin{align*}
			0&\leq \pi_n\left(\left\langle\tilde{\eta}^*\begin{bmatrix}
				y_1 & \ldots & y_n    
			\end{bmatrix}, \tilde{\eta}^*\begin{bmatrix}
				y_1 & \ldots & y_n    
			\end{bmatrix}\right\rangle_{E_1^n}\right)\leq ||\eta||^2\pi_n\left(\left\langle\begin{bmatrix}
				y_1 & \ldots & y_n    
			\end{bmatrix}, \begin{bmatrix}
				y_1 & \ldots & y_n    
			\end{bmatrix}\right\rangle_{E_2^n}\right)
		\end{align*}
		Therefore for all $f_1, \ldots, f_n\in F$ and $x_1, \ldots, x_n\in E_1$
		\begin{align*}
			0&\leq\left\langle\begin{bmatrix}
				f_1\\
				\vdots\\
				f_n
			\end{bmatrix}, \pi_n\left(\left\langle\tilde{\eta}\begin{bmatrix}
				x_1 & \ldots & x_n    
			\end{bmatrix}, \tilde{\eta}\begin{bmatrix}
				x_1 & \ldots & x_n    
			\end{bmatrix}\right\rangle_{E_2^n}\right)\begin{bmatrix}
				f_1\\
				\vdots\\
				f_n
			\end{bmatrix}\right\rangle_{F^n} \\
			&\leq ||\eta||^2\left\langle\begin{bmatrix}
				f_1\\
				\vdots\\
				f_n
			\end{bmatrix}, \pi_n\left(\left\langle\begin{bmatrix}
				x_1 & \ldots & x_n    
			\end{bmatrix}, \begin{bmatrix}
				x_1 & \ldots & x_n    
			\end{bmatrix}\right\rangle_{E_2^n}\right)\begin{bmatrix}
				f_1\\
				\vdots\\
				f_n
			\end{bmatrix}\right\rangle_{F^n}
		\end{align*}
		Similarly, we have for all $f_1, \ldots, f_n\in F$ and $y_1, \ldots, y_n\in E_2$
		\begin{align*}
			0&\leq\left\langle\begin{bmatrix}
				f_1\\
				\vdots\\
				f_n
			\end{bmatrix}, \pi_n\left(\left\langle\tilde{\eta}^*\begin{bmatrix}
				y_1 & \ldots & y_n    
			\end{bmatrix}, \tilde{\eta}^*\begin{bmatrix}
				y_1 & \ldots & y_n    
			\end{bmatrix}\right\rangle_{E_1^n}\right)\begin{bmatrix}
				f_1\\
				\vdots\\
				f_n
			\end{bmatrix}\right\rangle_{F^n} \\
			&\leq ||\eta||^2\left\langle\begin{bmatrix}
				f_1\\
				\vdots\\
				f_n
			\end{bmatrix}, \pi_n\left(\left\langle\begin{bmatrix}
				y_1 & \ldots & y_n    
			\end{bmatrix}, \begin{bmatrix}
				y_1 & \ldots & y_n    
			\end{bmatrix}\right\rangle_{E_2^n}\right)\begin{bmatrix}
				f_1\\
				\vdots\\
				f_n
			\end{bmatrix}\right\rangle_{F^n}
		\end{align*}
		
		Since for all $x_1, \ldots, x_n\in E_1$ and $f_1, \ldots, f_n\in F$ 
		\begin{align*}
			\sum\limits_{i, j} \langle f_i, \pi(\langle \eta(e_i), &\eta(e_j)\rangle_{E_2})f_j\rangle_F\\
			&=\left\langle\begin{bmatrix}
				f_1\\
				\vdots\\
				f_n
			\end{bmatrix}, \pi_n\left(\left\langle\tilde{\eta}\begin{bmatrix}
				x_1 & \ldots & x_n    
			\end{bmatrix}, \tilde{\eta}\begin{bmatrix}
				x_1 & \ldots & x_n    
			\end{bmatrix}\right\rangle_{E_2^n}\right)\begin{bmatrix}
				f_1\\
				\vdots\\
				f_n
			\end{bmatrix}\right\rangle_{F^n}\\
			&\leq ||\eta||^2\left\langle\begin{bmatrix}
				f_1\\
				\vdots\\
				f_n
			\end{bmatrix}, \pi_n\left(\left\langle\begin{bmatrix}
				x_1 & \ldots & x_n    
			\end{bmatrix}, \begin{bmatrix}
				x_1 & \ldots & x_n    
			\end{bmatrix}\right\rangle_{E_1^n}\right)\begin{bmatrix}
				f_1\\
				\vdots\\
				f_n
			\end{bmatrix}\right\rangle_{F^n}\\
			&=||\eta||^2\sum\limits_{i, j} \langle f_i, \pi(\langle x_i, x_j\rangle_{E_1})f_j\rangle_F
		\end{align*}
		then 
		$$\left\| \sum\limits_{i, j} \langle f_i, \pi(\langle \eta(x_i), \eta(x_j)\rangle_{E_2})f_j\rangle_F\right\|_C\leq ||\eta||^2\left\|\sum\limits_{i, j} \langle f_i, \pi(\langle x_i, x_j\rangle_{E_1})f_j\rangle_F\right\|_C$$
		proving $(1)$. Similarly, for all $y_1, \ldots, y_n\in E_2$ and $f_1, \ldots, f_n\in F$ 
		\begin{align*}
			\sum\limits_{i, j} \langle f_i, \pi(\langle \eta^*(e_i), &\eta^*(e_j)\rangle_{E_1})f_j\rangle_F\\
			&=\left\langle\begin{bmatrix}
				f_1\\
				\vdots\\
				f_n
			\end{bmatrix}, \pi_n\left(\left\langle\tilde{\eta}^*\begin{bmatrix}
				y_1 & \ldots & y_n    
			\end{bmatrix}, \tilde{\eta}\begin{bmatrix}
				y_1 & \ldots & y_n    
			\end{bmatrix}\right\rangle_{E_1^n}\right)\begin{bmatrix}
				f_1\\
				\vdots\\
				f_n
			\end{bmatrix}\right\rangle_{F^n}\\
			&\leq ||\eta||^2\left\langle\begin{bmatrix}
				f_1\\
				\vdots\\
				f_n
			\end{bmatrix}, \pi_n\left(\left\langle\begin{bmatrix}
				y_1 & \ldots & y_n    
			\end{bmatrix}, \begin{bmatrix}
				y_1 & \ldots & y_n    
			\end{bmatrix}\right\rangle_{E_2^n}\right)\begin{bmatrix}
				f_1\\
				\vdots\\
				f_n
			\end{bmatrix}\right\rangle_{F^n}\\
			&=||\eta||^2\sum\limits_{i, j} \langle f_i, \pi(\langle y_i, y_j\rangle_{E_2})f_j\rangle_F
		\end{align*}
		so that
		$$\left\| \sum\limits_{i, j} \langle f_i, \pi(\langle \eta^*(y_i), \eta^*(y_j)\rangle_{E_1})f_j\rangle_F\right\|_C\leq ||\eta||^2\left\|\sum\limits_{i, j} \langle f_i, \pi(\langle y_i, y_j\rangle_{E_2})f_j\rangle_F\right\|_C$$
		proving $(2)$.
		
	\end{proof}
	
	\begin{proposition}\label{Prop: Lift 1}\

		\noindent Let $B$ and $C$ be $C^*$-algebras, and let $(F, \pi)$ be a $C^*$-correspondence from $B$ to $C$. For $i=1, 2$, let $R_i=(A_i, B, (E_i, \phi_i))\in\textnormal{PosCor}(-, B)$. If $(\rho, \eta)\in \textnormal{Hom}_{\textnormal{PosCor}(-, B)}(R_1, R_2)$, then there exists $\hat{\eta}\in \mathcal{L}(E_1\otimes_\pi F, E_2\otimes_\pi F)$ such that
		\begin{enumerate}
			\item For all $x\in E_1$ and $f\in F$, $\hat{\eta}(x\dot{\otimes} f)=\eta(x)\dot{\otimes} f$
			\item For all $y\in E_2$ and $f\in F$, $\hat{\eta}^*(y\dot{\otimes} f)=\eta^*(y)\dot{\otimes} f$
			\item $(\rho, \hat{\eta} )\in\textnormal{Hom}_{\textnormal{PosCor}(-, C)}((A_1, C, (E_1\otimes_\pi F, \tilde{\phi}_1)), (A_2, C, (E_2\otimes_\pi F, \tilde{\phi}_2)))$.
		\end{enumerate}
		
	\end{proposition}
	
	\begin{proof}\
		
		\noindent As the map 
		$$E_1\times F\to E_2\otimes_\pi F\quad\quad (x, f)\mapsto \eta(x)\dot{\otimes} f$$
		is bilinear and $C$-linear in the second slot, then we obtain a unique $C$-linear map 
		$$\xi:E_1\otimes_{\textnormal{alg}}F\to E_2\otimes_\pi F\quad\quad \xi(x\otimes f)=\eta(x)\dot{\otimes} f$$
		Suppose $x\in E_1, b\in B$, and $c\in C$, then
		\begin{align*}
			\xi(xb\otimes f-e\otimes \pi(b)f)&=\eta(xb)\dot{\otimes} f-\eta(x)\dot{\otimes} \pi(b)f=\eta(x)b\dot{\otimes} f-\eta(x)\dot{\otimes} \pi(b)f=0
		\end{align*}
		Therefore we obtain a unique $C$-linear map
		$$\tilde{\xi}:E_1\otimes_B F\to E_2\otimes_\pi F\quad \quad \tilde{\xi}(x\dot{\otimes} f)=\eta(x)\dot{\otimes} f$$
		We observe
		\begin{align*}
			\left\|\tilde{\xi}\left(\sum\limits_{k=1}^n x_k\dot{\otimes} f_k\right)\right\|^2&=\left\|\sum\limits_{i, j} \langle f_i, \pi(\langle \eta x_i, \eta x_j\rangle_{E_2})f_j\rangle_F\right\|_C\\
			&\leq ||\eta||^2\left\|\sum\limits_{i, j} \langle f_i, \pi(\langle  x_i,  x_j\rangle_{E_1})f_j\rangle_F\right\|_C=||\eta||^2\left\|\sum\limits_{k=1}^n x_k\dot{\otimes} f_k\right\|^2
		\end{align*}
		where the inequality comes from the Lemma \ref{Lemma: Bound 1}. Therefore $\tilde{\xi}$ is a bounded $C$-linear map which enables us to extend it to a bounded $C$-linear map on $E_1\otimes_\pi F$ which we denote by $\hat{\eta}$. By a similar argument, we obtain a $C$-linear map
		$$\widehat{\eta^*}:E_2\otimes_\pi F\to E_1\otimes_\pi F\quad \quad \tilde{\gamma}(y\dot{\otimes} f)=\eta^*(y)\dot{\otimes} f$$

		We now show $\hat{\eta}=\widehat{\eta^*}^*$. Using continuity of both $\hat{\eta}$ and $\widehat{\eta^*}$ as well as density of $E_i\otimes_B F$ in $E_i\otimes_\pi F$, it suffices to prove the claim on the dense subspaces $E_i\otimes_B F$. Using linearity, it suffices to proof the claim on elements of the form $x\dot{\otimes} f_1\in E_1\otimes_BF$ and $y\dot{\otimes} f_2\in E_2\otimes_BF$. Given such elements, we have
		\begin{align*}
			\langle \widehat{\eta^*}(y\dot{\otimes} f_2), x\dot{\otimes} f_1\rangle&=\langle \eta^*(y)\dot{\otimes} f_2, x\dot{\otimes} f_1\rangle=\langle f_2, \pi\langle \eta^*(y), x\rangle_{E_1})f_1\rangle_F\\
			&=\langle f_2, \pi(\langle y,\eta( x)\rangle_{E_2})f_1\rangle_F=\langle y\dot{\otimes} f_2, \eta(x)\dot{\otimes} f_1\rangle=\langle e_2\dot{\otimes} f_2, \hat{\eta}(x\dot{\otimes} f_1)\rangle
		\end{align*}
		Hence $\hat{\eta}=\widehat{\eta^*}^*$ so that $\hat{\eta}^*=\widehat{\eta^*}$. From how $\hat{\eta}$ and $\hat{\eta}^*$ are constructed, we know $\hat{\eta}$ satisfies properties $(1)$ and $(2)$ of the assertion.
		
		Now we verify $(\rho, \hat{\eta})$ is a morphism; in particular, we show $\tilde{\phi}_2(\rho(a))\circ\hat{\eta}=\hat{\eta}\circ \tilde{\phi}_1(a)$ for all $a\in A_1$. Fix $a\in A_1$. Using continuity of both $\tilde{\phi}_2(\rho(a))\circ\hat{\eta}$ and $\hat{\eta}\circ \tilde{\phi}_1(a)$ as well as density of $E_1\otimes_B F$, it suffices to show equality on $E_1\otimes_BF$. Using linearity, it suffices to show equality on elements of the form $x\dot{\otimes} f\in E_1\otimes_\pi F$. Given such an element, we have
		\begin{align*}
			(\tilde{\phi}_2(\rho(a))\circ\hat{\eta})(x\dot{\otimes} f)&=\tilde{\phi}_2(\rho(a))(\eta(x)\dot{\otimes} f)=\phi_2(\rho(a))\eta(x)\dot{\otimes} f\\
			&=\eta(\phi_1(a)x)\dot{\otimes} f=\hat{\eta}(\phi_1(a)x\dot{\otimes} f)=(\hat{\eta}\circ\tilde{\phi}_1(a))(x\dot{\otimes} f)
		\end{align*}
		Hence we have $\tilde{\phi}_2(\rho(a))\circ \hat{\eta}=\hat{\eta}\circ\tilde{\phi}_1(a)$ for all $a\in A_1$ as claimed. 
		
	\end{proof}
	
	\begin{theorem}\label{Theorem: Tensor Functor}\

		\noindent Let $B$ and $C$ be $C^*$-algebras. If $(F, \pi)$ is a $C^*$-correspondence from $B$ to $C$, then the assignment
		\begin{align*}
			(-)\otimes_\pi F:\textnormal{PosCor}(-, B)&\to \textnormal{PosCor}(-, C) \\
			(A, B, (E, \phi))&\mapsto (A, C, (E\otimes_\pi F, \phi(-)\otimes_\pi I))\\
			(\rho, \eta)&\mapsto (\rho, \hat{\eta})
		\end{align*}
		defines a covariant functor.
		
	\end{theorem}
	
	\begin{proof}\
		
		\noindent It is clear the assignment on objects is well-defined. Using Proposition \ref{Prop: Lift 1}, the assignment on morphisms is well-defined. That the assignment is functorial in a covariant manner immediately follows from how $\hat{\eta}$ is defined in Proposition \ref{Prop: Lift 1}.
		
	\end{proof}
	
	\begin{proposition}\label{Prop: inclusion tensor functor}\

		\noindent Let $B$ be $C^*$-algebras, and let $\textnormal{inc}:B\to M(B)$ be the inclusion $*$-algebra homomorphism. Then $(-)\otimes_\textnormal{inc} B$ is naturally isomorphic  to $1_{\textnormal{PosCor}(-, B)}$. In particular, for each object $(A, (E, \phi))\in\textnormal{PosCor}(-, B)$, the map 
		$$(1_A, \iota):(A, B,  (E\otimes_{\textnormal{inc}}B, \phi\otimes_{\textnormal{inc}}I))\to (A, B, (E, \phi))$$
		defines an isomorphism where $\iota$ is the unitary isomorphism given by
		$$\iota:E\otimes_{\textnormal{inc}}B\to E\quad\quad \iota(x\dot{\otimes}b)=xb$$
		
	\end{proposition}
	
	\begin{proof}\
		
		\noindent We first verify $\iota$ is a unitary map. Define the map
		$$U:E\otimes_B B\to E\quad\quad U(x\dot{\otimes} b)=xb$$
		It is clear that $U$ is a well-defined $B$-linear map. Furthermore, since
		\begin{align*}
			\left\|U(\sum\limits_{k=1}^n x_k\dot{\otimes} b_k)\right\|^2&=\left\|\sum\limits_{k=1}^nx_kb_k\right\|^2=\left\|\sum\limits_{i, j}\langle x_ib_i, x_jb_j\rangle\right\|_B=\left\|\sum\limits_{i, j} b_i^*\langle x_i, x_j\rangle b_j\right\|_B=\left\|\sum\limits_{k=1}^n x_k\dot{\otimes} b_k\right\|,
		\end{align*}
		then $U$ is an isometry. Therefore, we can extend $U$ to an isometry $\iota:E\otimes_{\textnormal{inc}}B\to E$. 
		
		We claim $\iota$ is surjective. As the image of $U$ is $EB=\textnormal{Span}(\{eb: e\in E, b\in B\})$ which is dense in $E$ and $U$ is an isometry, then we know the image of $\iota$ is $\overline{EB}=E$. Hence $\iota$ is indeed surjective. Since $\iota$ is a surjective $B$-linear isometry, then $\iota$ is a unitary map.
		
		Now we verify 
		$$(1_A, \iota):(A, B,  (E\otimes_{\textnormal{inc}}B, \phi\otimes_{\textnormal{inc}}I))\to (A, B, (E, \phi))$$
		is a morphism in $\textnormal{PosCor}(-, B)$. To simplify notation, we write $\tilde{\phi}=\phi\otimes_{\textnormal{inc}}I$. As $\iota$ is a $B$-linear adjointable map, then we just need to verify $\iota$ intertwines $\tilde{\phi}$ and $\phi$ with respect to $1_A$. Fix $a\in A$. Using density of $E\otimes_{B}B$ in $E\otimes_{\textnormal{inc}}B$ and continuity of both $\iota\circ \tilde{\phi}(a)$ and $\phi(a)\circ\iota$, it suffices to show $\iota\circ \tilde{\phi}(a)=\phi(a)\circ\iota$ on $E\otimes_{B} B$. Using linearity, it suffices to show $\iota\circ \tilde{\phi}(a)=\phi(a)\circ\iota$ on elements of the form $x\dot{\otimes} b\in E\otimes_BB$. Given such an element, we have
		\begin{align*}
			(\iota\circ \tilde{\phi}(a))(x\dot{\otimes} b)&=\iota(\phi(a)x\dot{\otimes} b)=\phi(a)(x)b=\phi(a)(xb)=(\phi(a)\circ \iota)(x\dot{\otimes} b) 
		\end{align*}
		Therefore we conclude $\iota\circ \tilde{\phi}(a)=\phi(a)\circ\iota$ for all $a\in A$ showing $(\textnormal{inc}, \iota)$ defines a morphism as claimed. Observe, as $\iota$ is unitary, the morphism $(1_A, \iota)$ is an isomorphism.
		
		Now we verify these isomorphisms satisfy the naturality condition. For $i=1, 2$, let $(A_i, B, (E_i, \phi_i))\in \textnormal{PosCor}(-, B)$, and let $(\rho, \eta):(A_1, B, (E_1, \phi_1))\to (A_2, B, (E_2, \phi_2))$ be a morphism. Denote the induced objects and morphisms from applying $(-)\otimes_{\textnormal{inc}}B$ as $(A_i, B, (E_i\otimes_{\textnormal{inc}}, \tilde{\phi}_i))$ and $(\rho, \hat{\eta})$, respectively. Let $$(1_{A_i}, \iota_i):(A_i, B, (E_i\otimes_{\textnormal{inc}}B, \tilde{\phi}_i))\to (A_i, B, (E_i, \phi_i))$$
		be the corresponding isomorphisms. To prove naturality, we verify the following diagram commutes:
		\begin{center}
			\begin{tikzcd}
				{(A_1, B, (E_1\otimes_{\textnormal{inc}}B, \tilde{\phi}_1))} &&& {(A_1, B, (E_1, \phi_1))} \\
				\\
				\\
				{(A_2, B, (E_2\otimes_{\textnormal{inc}}B, \tilde{\phi}_2))} &&& {(A_2, B, (E_2, \phi_2))}
				\arrow["{(1_{A_1}, \iota_1)}"{description}, from=1-1, to=1-4]
				\arrow["{(\rho, \hat{\eta})}"{description}, from=1-1, to=4-1]
				\arrow["{(\rho, \eta)}"{description}, from=1-4, to=4-4]
				\arrow["{(1_{A_2}, \iota_2)}"{description}, from=4-1, to=4-4]
			\end{tikzcd}
		\end{center}
		From how composition is defined, we just need to verify $\eta\circ\iota_1=\iota_2\circ\hat{\eta}$. Using linearity and continuity of the adjointable maps as well as density of $E_1\otimes_BB$ in $E_1\otimes_{\textnormal{inc}}B$, it suffices to check equality on elements of the form $x\dot{\otimes} b\in E_1\otimes_{\textnormal{inc}}B$. Given such an element we have
		\begin{align*}
			(\eta\circ \iota_1)(x\dot{\otimes} b)&=\eta(xb)=\eta(x)b=U_2(\eta(x)\dot{\otimes} b)=(\iota_2\circ \hat{\eta})(x\dot{\otimes} b)
		\end{align*}
		Thus the diagram commutes. Since $(1_A, \iota)$ is an isomorphism, then we conclude the two functors are naturally isomorphic.
		
	\end{proof}
	
	We remark the adjoint of the unitary $\iota:E\otimes_\textnormal{inc} B\to E$ can be realized as the strong limit of a net of $\mathbb{C}$-linear maps. Given an approximate unit $(b_\lambda)_{\lambda\in\Lambda}$ for $B$, we can define, for each $\lambda\in\Lambda$, the $B$-linear map
	$$V_\lambda:E_B\to E_B\otimes_{\textnormal{inc}}B\quad\quad V_\lambda(x)=x\dot{\otimes} b_\lambda$$
	A quick computation shows $V_\lambda$ is a contraction. Using that $(b_\lambda)_{\lambda\in\Lambda}$ is an approximate unit, we know $(V_\lambda x)_{\lambda\in\Lambda}$ is a Cauchy net in $E_B\otimes_{\textnormal{inc}} B$ and therefore converges to a point $V(x)$. Let $V:E_B\to E_B\otimes_{\textnormal{inc}}B$ be the induced map. It is clear $V$ is B-linear contraction. Furthermore, 
	\begin{align*}
		\langle V(x), y\dot{\otimes} b\rangle&=\lim\limits_{\lambda\to \infty} \langle x\dot{\otimes} b_\lambda, y\dot{\otimes} b\rangle=\lim\limits_{\lambda\to \infty}b_\lambda\langle x, y\rangle b=\langle x, yb\rangle=\langle x, U(y\dot{\otimes} b)\rangle
	\end{align*}
	showing $\iota^*=V$.

	\begin{proposition}\label{Prop: Composition and tensor}\

		\noindent Let $B, C$, and $D$ be $C^*$-algebras. Let $\rho:B\to M(C)$ and $\chi:C\to M(D)$ be non-degenerate $*$-algebra homomorphisms. Let $\tilde{\chi}$ be the unique extension of $\chi$ to $M(C)$. Then the functor $((-)\otimes_{\rho}C)\otimes_{\chi}D$ is naturally isomorphic to the functor $(-)\otimes_{\tilde{\chi}\circ \rho} D$. In particular, for each object $(A, B, (E, \phi))\in \textnormal{PosCor}(-, B)$, the map 
		$$(1_A, U):(A, D, ((E\otimes_\rho C)\otimes_\chi D, (\phi\otimes_\rho I)\otimes_\chi I))\to (A, D, (E\otimes_{\tilde{\chi}\circ\rho}D, \phi\otimes_{\tilde{\chi}\circ\rho}I))$$
		defines an isomorphism where $U$ is the unitary isomorphism given by
		$$U:(E\otimes_{\rho}C)\otimes_{\chi}D\to E\otimes_{\tilde{\chi}\circ\rho}D\quad\quad U((x\dot{\otimes}c)\dot{\otimes}d)=x\dot{\otimes}\chi(c)d$$
		
	\end{proposition}
	
	\begin{proof}\
		
		\noindent We first verify $U$ is a well-defined unitary map. Define
		$$(E\otimes_{B} C)\otimes_{C} D\to E\otimes_{\tilde{\chi}\circ\rho} D\quad\quad (x\dot{\otimes} c)\dot{\otimes} d\mapsto x\dot{\otimes} \chi(c)d$$
		which is a well-defined $D$-linear map.  Since
		\begin{align*}
			\left\|\sum\limits_{k=1}^n\sum\limits_{\ell=1}^{m_k} x_{k, \ell}\dot{\otimes} \chi(c_{k, \ell})d_k\right\|^2
			&=\left\|\sum\limits_{i, j=1}^n\sum\limits_{k=1}^{m_i}\sum\limits_{ \ell=1}^{m_j} d_i^*\chi(c_{i, k})^*(\tilde{\chi}\circ\rho)(\langle x_{i, k}, x_{j, \ell}\rangle)\chi(c_{j, \ell})d_{j}\right\|_D\\
			&=\left\|\sum\limits_{i, j=1}^n\sum\limits_{k=1}^{m_i}\sum\limits_{ \ell=1}^{m_j} d_i^*\chi\left(c_{i, k}^*\rho(\langle x_{i, k}, x_{j, \ell}\rangle)c_{j \ell}\right)\chi(c_{j, \ell})d_{j}\right\|_D\\
			&=\left\|\sum\limits_{k=1}^n\left(\sum\limits_{\ell=1}^{m_k} x_{k, \ell}\dot{\otimes} c_{k, \ell}\right)\dot{\otimes} d_{k}\right\|^2
		\end{align*}
		then the map is an isometry. As the closure of $(E\otimes_B C)\otimes_CD$ is $(E\otimes_{\rho} C)\otimes_{\chi} D$, then we can extend the map to a $D$-linear isometry, denote by $U$, on $(E\otimes_{\rho} C)\otimes_{\chi} D$. 
		
		We claim $U$ is surjective. To prove this, we show $E\otimes_BD$ is contained in the image of $U$. Fix $x\dot{\otimes}d\in E\otimes_{\tilde{\chi}\circ\rho}D$. Let $(c_\lambda)_{\lambda\in\Lambda}$ be an approximate unit in $C$. As $\chi$ is non-degenerate, then $(\chi(c_\lambda))_{\lambda\in\Lambda}$ converges in the strict topology to $1_D$. Therefore
		$$||x\dot{\otimes}\chi(c_\lambda)-x\dot{\otimes}d||^2=\left\|x\dot{\otimes} (\chi(c_\lambda)d-d)\right\|^2\leq ||\sqrt{(\tilde{\chi}\circ\rho)(\langle x, x\rangle)}||_D^2*||\chi(c_\lambda)d-d||_{D}^2$$
		implies $(x\dot{\otimes}\chi(c_\lambda)d)_{\lambda\in\Lambda}$ converges in $E\otimes_{\tilde{\chi}\circ \rho}D$ to $x\dot{\otimes} d$. As $U((x\dot{\otimes}c_\lambda)\dot{\otimes}d)=x\dot{\otimes}\chi(c_\lambda)d$ for each $\lambda\in\Lambda$ and $U$ is an isometry, then $((x\dot{\otimes}c_\lambda)\dot{\otimes}d)_{\lambda\in\Lambda}$ is a Cauchy net in $(E\otimes_{\rho}C)\otimes_{\chi}D$. Therefore the net converges to a point $z\in (E\otimes_{\rho}C)\otimes_{\chi}D$ for which $U(z)=x\dot{\otimes}d$. Using linearity of $U$, it follows that the image of $U$ contains $E\otimes_{B}D$.
		
		As $U$ is an isometry, then the image of $U$ is closed. Since the image of $U$ contains $E\otimes_B D$ which is dense in $E\otimes_{\tilde{\chi}\circ \rho}D$, then the image of $U$ is all of $E\otimes_{\tilde{\chi}\circ \rho}D$. Hence $U$ is surjective. As $U$ is a surjective $D$-linear isometry, then $U$ is a unitary map.
		
		Now we verify 
		$$(1_A, U):(A, D, ((E\otimes_\rho C)\otimes_\chi D, (\phi\otimes_\rho I)\otimes_\chi I))\to (A, D, (E\otimes_{\tilde{\chi}\circ\rho}D, \phi\otimes_{\tilde{\chi}\circ\rho}I))$$
		is a morphism in $\textnormal{PosCor}(-, D)$. To simplify the notation, we write $$\tilde{\phi}=(\phi(-)\otimes_{\rho} I)\otimes_{\chi}I\quad\textnormal{ and }\quad\hat{\phi}=\phi(-)\otimes_{\tilde{\chi}\circ \rho} I.$$
		As $U$ is an adjointable operator, we only need to verify $U$ intertwines $\tilde{\phi}$ and $\hat{\phi}$ with respect to $\textnormal{inc}$. Fix $a\in A$. Using density and continuity, it suffices to check $U\circ \tilde{\phi}(a)=\hat{\phi}(a)\circ U$ on $(E\otimes_B C)\otimes_CD$. Using linearity, it suffices to check equality on elements of the form $(x\dot{\otimes} c)\dot{\otimes} d\in (E\otimes_B C)\otimes_CD$. Given such an element, we have
		\begin{align*}
			(U\circ\tilde{\phi}(a))((x\dot{\otimes} c)\dot{\otimes} d)&=U((\phi(a)x\dot{\otimes} c)\dot{\otimes} d)=\phi(a)x\dot{\otimes} \rho_2(c)d\\
			&=\hat{\phi}(a)(x\dot{\otimes} \rho_2(c)d)=(\hat{\phi}(a)\circ U)((x\dot{\otimes} c)\dot{\otimes} d)
		\end{align*}
		Therefore $U\circ\tilde{\phi}(a)=\hat{\phi}(a)\circ U$ for all $a\in A$ proving the claim. Observe, as $U$ is a unitary, then $(1_A, U)$ is an isomorphism.
		
		Now we verify the isomorphisms satisfy the naturality condition. Fix objects $(A_1, B, (E_1, \phi_1))$ and $(A_2, B, (E_2, \phi_2))$ in $\textnormal{PosCor}(-, B)$, and let $(\sigma, \eta):(A_1, B, (E_1, \phi_1))\to (A_2, B, (E_2, \phi_2))$ be a morphism. For $i=1, 2$, let 
		$$(1_{A_i}, U_i): (A_i, D, (E_i\otimes_{\rho}C)\otimes_{\chi}D, \tilde{\phi}_i)\to (A_i, D, (E_i\otimes_{\tilde{\chi}\circ\rho}D, \hat{\phi}_i))$$
		be the isomorphisms from above. To prove naturality, we verify the  following diagram commutes in $\textnormal{PosCor}(-, D)$:
		\begin{center}
			\begin{tikzcd}
				{(A_1, D, ( (E_1\otimes_{\rho}C)\otimes_{\chi}D, \tilde{\phi}_1))} &&&& {(A_1, D, (E_1\otimes_{\tilde{\chi}\circ\rho}D, \hat{\phi}_1))} \\
				\\
				\\
				{(A_2, D, ((E_2\otimes_{\rho}C)\otimes_{\chi}D, \tilde{\phi}_2))} &&&& {(A_2, D, (E_2\otimes_{\tilde{\chi}\circ\rho}D, \hat{\phi}_2))}
				\arrow["{(1_{A_1}, U_1)}"{description}, from=1-1, to=1-5]
				\arrow["{(\sigma, \hat{\hat{\eta}})}"{description}, from=1-1, to=4-1]
				\arrow["{(\sigma, \hat{\eta})}"{description}, from=1-5, to=4-5]
				\arrow["{(1_{A_2}, U_2)}"{description}, from=4-1, to=4-5]
			\end{tikzcd}
		\end{center}
		From how composition is defined, we only need to check $U_2\circ\hat{\hat{\eta}}=\hat{\eta}\circ U_1$. Using continuity and linearity of $U_2\circ\hat{\hat{\eta}}$ and $\hat{\eta}\circ U_1$ as well as density of $(E_1\otimes_BC)\otimes_CD$ in $(E_1\otimes_{\rho} C)\otimes_{\chi}D$, it suffices to check equality on elements of the form $(x\dot{\otimes} c)\dot{\otimes} d\in (E\otimes_B C)\otimes_CD$. Given such an element we have
		\begin{align*}
			(U_2\circ\hat{\hat{\eta}})((x\dot{\otimes} c)\dot{\otimes} d)&=U_2((\eta(x)\dot{\otimes} c)\dot{\otimes} d)=\eta(x)\dot{\otimes} \chi(c)d\\
			&=\hat{\eta}(x\dot{\otimes} \chi(c)d)=(\hat{\eta}\circ U_1)((x\dot{\otimes} c)\dot{\otimes} d)
		\end{align*}
		Hence the diagram does indeed commute. As $(1_{A_1}, U_1)$ and $(1_{A_2}, U_2)$ are isomorphisms, then we conclude the two functors are naturally isomorphic.
		
	\end{proof}

	\subsection{The categories $\textnormal{PosCor}$ and $\textnormal{PosCor}_{*\textnormal{-alg}}$}

	Let us return to our morphisms at the start of Section 3.2. Given objects $(A, B, (E_B, \phi))$ and $(C, D, (E_D, \psi))$ in $\textnormal{PosCor}(-, B)$ and  in $\textnormal{PosCor}(-, D)$, respectively, we determined a candidate for a morphism from $(A, B, (E_B, \phi))$ to $(C, D, (E_D, \psi))$ was given by the data $((F, \pi), (\rho, \eta))$ where $(F, \pi)$ is a $C^*$-correspondence from $B$ to $D$ and $(\rho, \eta)$ is a morphism from $(A, D, (E_B\otimes_\pi F, \tilde{\phi}))$ to $(C, D, (E_D, \psi))$ in $\textnormal{PosCor}(-, D)$. While this works for defining morphisms, we restrict our attention from all $C^*$-correspondences $(F, \pi)$ to those of the form $(C, \pi)$. This restriction is made as we want the choice of $C^*$-correspondence to represent a change of the background algebra which $C^*$-correspondences of the form $(C, \pi)$ capture. We note the following category and the results in Section 4.3 likely generalize if one relaxes this assumption.
	
	Given our morphisms, we now need a composition rule for composing them. To see where our composition rule comes from, suppose we have objects $(A_i, B_i, (E_{B_i}, \phi_i))$ for $i=1, 2, 3$, and suppose we have the data $((B_{i+1}, \rho_i), (\chi_i, \eta_i))$ for $i=1, 2$, where $(B_{i+1}, \rho_i)$ is a $C^*$-correspondence from $B_{i}$ to $B_{i+1}$ and 
	$$(\chi_i, \eta_i): (A_i, B_{i+1}, (E_{B_i}\otimes_{\rho_i}B_{i+1}, \phi_i\otimes_{\rho_i} I))\to (A_{i+1}, B_{i+1}, (E_{B_{i+1}}, \phi_{i+1}))$$
	is a morphism in $\textnormal{PosCor}(-, B_{i+1})$. Applying the functor $(-)\otimes_{\rho_2}B_3$ to $(\chi_1, \eta_1)$, we obtain the morphism
	$$(\chi_1, \hat{\eta}_1):(A_1, B_3, (E_{B_1}\otimes_{\rho_1} B_2)\otimes_{\rho_2}B_3, (\phi_1\otimes_{\rho_1}I)\otimes_{\rho_2}I)\to (A_2, B_3, (E_{B_2}\otimes_{\rho_2}B_3, \phi_2\otimes_{\rho_2}I)$$
	As 
	$$(\chi_2, \eta_2):(A_2, B_3, (E_{B_2}\otimes_{\rho_2}B_3, \phi_2\otimes_{\rho_2}I))\to(A_3, B_3, (E_{B_3}, \phi_3)),$$ 
	we can compose $(\chi_2, \eta_2)$ and $(\chi_1, \hat{\eta}_1)$ in $\textnormal{PosCor}(-, B_3)$ to obtain the morphism
	$$(\chi_2\circ\chi_1, \eta_2\circ\hat{\eta}_1):(A_1, B_3, ((E_{B_1}\otimes_{\rho_1} B_2)\otimes_{\rho_2}B_3, (\phi_1\otimes_{\rho_1}I)\otimes_{\rho_2}I)\to (A_3, B_3, (E_{B_3}, \phi_3))$$
	To compose the $C^*$-correspondences $(B_2, \rho_1)$ and $(B_3, \rho_2)$, we compose the maps $\rho_1$ and $\rho_2$ to obtain the $C^*$-correspondence $(B_3, \tilde{\rho}_2\circ\rho_1)$ where $\tilde{\rho}_2$ is the unique extension of $\rho_2$ to $M(B_2)$.
	
	From the computation, we are led to define the composition of $((B_2, \rho_1), (\chi_1, \eta_1))$ and $((B_3, \rho_2), (\chi_2, \eta_2))$ as
	$$((B_3, \rho_2), (\chi_2, \eta_2))\circ ((B_2, \rho_1), (\chi_1, \eta_1))=((B_3, \tilde{\rho}_2\circ\rho_1), (\chi_2\circ\chi_1, \eta_2\circ\hat{\eta}_1))$$
	However, we have a domain issue: $\eta_2\circ\hat{\eta}_1$ is defined on the module $(E_{B_1}\otimes_{\rho_1} B_1)\otimes_{\rho_2}B_3$ and not $E_{B_1}\otimes_{\tilde{\rho}_2\circ\rho_1}B_3$. Using Proposition \ref{Prop: Composition and tensor}, we obtain an isomorphism $(1_{A_1}, U)$ where
	$$U:(E_{B_1}\otimes_{\rho_1}B_2)\otimes_{\rho_2}B_3\to E_{B_1}\otimes_{\tilde{\rho}_2\circ\rho_1}B_3\quad\quad U((x\dot{\otimes} b_2)\dot{\otimes}b_3)=x\dot{\otimes}\rho_2(b_2)b_3$$
	is a unitary isomorphism. Using this isomorphism, we correct the naive composition rule to
	$$((B_3, \rho_2), (\chi_2, \eta_2))\circ ((B_2, \rho_1), (\chi_1, \eta_1))=((B_3, \tilde{\rho}_2\circ\rho_1), (\chi\circ\chi_1, \eta_2\circ\hat{\eta}_1\circ U^{-1}))$$
	
	\begin{definition}\
		
		\noindent Define the category $\textnormal{PosCor}$ where
		\begin{itemize}
			\item objects are triples $(A, B, (E_B, \phi))$ where $A$ and $B$ are $C^*$-algebras, and $(E_B, \phi)$ is a strict positive $C^*$-correspondence from $A$ to $B$.
			\item  a morphism from $(A, B, (E_B, \phi))$ to $(C, D, (E_D, \psi))$ is the data $(\chi, \rho, \eta)$ where 
			\begin{itemize}
				\item $\chi:A\to C$ is a non-degenerate $*$-algebra homomorphism with $\chi(A)\subset C$.
				\item $\rho:B\to D$ is a non-degenerate $*$-algebra homomorphism.
				\item $\eta:E_B\otimes_\rho D\to E_D$ is an adjointable $D$-linear map
			\end{itemize}
			and the data satisfies the condition that $(\chi, \eta)$ is a morphism in $\textnormal{PosCor}(-, D)$ from $(A, (E_B\otimes_\rho D, \phi\otimes_\rho I)$ to $(C, (E_D, \psi))$..
			\item Given morphisms 
			$$(\chi_1, \rho_1, \eta_1):(A_1, B_1, (E_{B_1}, \phi_1))\to (A_2, B_2, (E_{B_2}, \phi_2))$$
			and 
			$$(\chi_2, \rho_2, \eta_2):(A_2, B_2, (E_{B_2}, \phi_2))\to (A_3, B_3, (E_{B_3}, \phi_3))$$
			define the composition $(\chi_2, \rho_2, \eta_2)\circ(\chi_1, \rho_1, \eta_1)$ as the morphism
			$$(\chi_2\circ\chi_1, \tilde{\rho}_2\circ\rho_1, \eta_2\circ\hat{\eta}_1\circ U^{-1}): (A_1, B_1, (E_{B_1}, \phi_1))\to (A_3, B_3, (E_{B_3}, \phi_3))$$
			where $\tilde{\rho}_2$ is the unique extensions of $\rho_2$ and  $$U:(E_{B_1}\otimes_{\rho_1}B_2)\otimes_{\rho_2}B_3\to E_{B_1}\otimes_{\tilde{\rho_2}\circ\rho_1}B_3\quad\quad U((x\dot{\otimes}b_2)\dot{\otimes}b_3)=x\dot{\otimes}\rho_2(b_2)b_3$$
			is the unitary from Proposition \ref{Prop: Composition and tensor}.
			
		\end{itemize}
		Define the category $\textnormal{PosCor}_{*\textnormal{-alg}}$ as the subcategory of $\textnormal{PosCor}$ where  morphisms 
		$$(\chi, \rho, \eta): (A, B, (E_B, \phi))\to (C, D, (E_D, \psi))$$ satisfy the condition that $\rho(B)\subset D$.
		
	\end{definition}
	
	The next two propositions verify $\textnormal{PosCor}$ is indeed a category. From this, it immediately follows that $\textnormal{PosCor}_{*\textnormal{-alg}}$ is a well-defined subcategory.
	
	\begin{proposition}\

		\noindent For each $(A, B, (E_B, \phi))\in \textnormal{PosCor}$, the identity morphism is given by the triple $(1_A, \textnormal{inc}, \iota)$ where 
		$$\iota:E_B\otimes_{\textnormal{inc}}B\to E_B\quad\quad \iota(x\dot{\otimes} b)=xb$$
		is the unitary from Proposition \ref{Prop: inclusion tensor functor}.
		
	\end{proposition}
	
	\begin{proof}\
		
		\noindent Fix an object $(A, B, (E_B, \phi))\in \textnormal{PosCor}$. We simply verify $(1_A, \textnormal{inc}, \iota)$ satisfies the definition to be the identity morphism. Fix another object $(C, (E_D, \psi))\in \textnormal{PosCor}$.
		
		Suppose we have a morphism
		$$(\chi, \rho, \eta): (C, D, (E_D, \psi))\to (A, B, (E_B, \phi))$$
		We claim 
		$$(1_A, \textnormal{inc}, \iota)\circ (\chi, \rho, \eta)=(\chi, \rho, \eta)$$
		First, note
		$$(1_A, \textnormal{inc}, \iota)\circ (\chi, \rho, \eta)=(\chi, \widetilde{\textnormal{inc}}_B\circ\rho, \iota\circ\hat{\eta}\circ U^{-1})$$
		As $\widetilde{\textnormal{inc}}_B\circ \rho=\rho$, we just need to verify $\iota\circ\hat{\eta}\circ U^{-1}=\eta$. As $U$ is a unitary, it suffices to show $\iota\circ\hat{\eta}=\eta\circ U$. As both $\eta\circ U$ and $\iota\circ \hat{\eta}$ are continuous as well as $(E_C\otimes_CB)\otimes_BB$ is dense in $(E_C\otimes_\rho B)\otimes_{\textnormal{inc}}B$, then it suffices to show equality on the dense subspace. Using linearity, it suffices to show equality on elements of the form $(x\dot{\otimes} b_1)\dot{\otimes} b_2\in (E_C\otimes_\rho B)\otimes_{\textnormal{inc}}B$. Given such an element, we have
		\begin{align*}
			(\eta\circ U)((x\dot{\otimes} b_1)\dot{\otimes} b_2)&=\eta(x\dot{\otimes} b_1b_2)=\eta(x\dot{\otimes} b_1)b_2=\iota(\eta(x\dot{\otimes} b_1)\dot{\otimes} b_2)=(\iota\circ \hat{\eta})((x\dot{\otimes} b_1)\dot{\otimes} b_2)
		\end{align*}
		Therefore we conclude $\eta\circ U=\iota\circ \hat{\eta}$ which implies 
		$$(1_A, \textnormal{inc}, \iota)\circ (\chi, \rho, \eta)=(\chi, \rho, \eta).$$

		Now suppose we have a morphism $$(\chi, \rho, \eta):(A, B, (E_B, \phi))\to (C, D, (E_D, \psi))$$
		We claim
		$$(\chi, \rho, \eta)\circ(1_A, \textnormal{inc}, \iota)=(\chi, \rho, \eta)$$
		First, note
		$$(\chi, \rho, \eta)\circ(1_A, \textnormal{inc}, \iota)=(\chi, \tilde{\rho}\circ\textnormal{inc}, \eta\circ\hat{\iota}\circ U^{-1})$$
		As $\tilde{\rho}\circ\textnormal{inc}=\rho$, we just need to verify $\eta\circ\hat{\iota}\circ U^{-1}=\eta$. As $U$ is a unitary map, it suffices to show $\eta\circ\hat{\iota}=\eta\circ U$. As before, it suffices to show $\eta\circ U=\eta\circ\hat{\iota}$ on elements of the form $(x\dot{\otimes} b)\dot{\otimes} c\in (E_B\otimes_{\textnormal{inc}}B)\otimes_\rho C$. Given such an element, we have
		\begin{align*}
			(\eta\circ U)((x\dot{\otimes} b)\dot{\otimes} c)&=\eta(x\dot{\otimes} \rho(b)c)=\eta(xb\dot{\otimes} c)=(\eta\circ\hat{\iota})((x\dot{\otimes} b)\dot{\otimes} c)
		\end{align*}
		Therefore $\eta\circ U=\eta\circ\hat{\iota}$ which implies
		$$(\chi, \rho, \eta)\circ(1_A, \textnormal{inc}, \iota)=(\chi, \rho, \eta).$$
		
	\end{proof}
	
	\begin{proposition}\

		\noindent For $i=1, 2, 3, 4$, let $(A_i, B_i, (E_{B_i}, \phi_i))\in \textnormal{PosCor}$. For $i=1, 2, 3$, let
		$$(\chi_i, \rho_i, \eta_i): (A_i, B_i, (E_{B_i}, \phi_i))\to (A_{i+1}, B_{i+1}, (E_{B_{i+1}}, \phi_{i+1}))$$ 
		be a morphism in $\textnormal{PosCor}$.  Then
		$$(\chi_3, \rho_3, \eta_3)\circ((\chi_2, \rho_2, \eta_2)\circ (\chi_1, \rho_1, \eta_1))=((\chi_3, \rho_3, \eta_3)\circ(\chi_2, \rho_2, \eta_2))\circ (\chi_1, \rho_1, \eta_1)$$
		
	\end{proposition}
	
	\begin{proof}\
		
		\noindent First, note
		$$(\chi_3, \rho_3, \eta_3)\circ((\chi_2, \rho_2, \eta_2)\circ (\chi_1, \rho_1, \eta_1))=(\chi_3\circ\chi_2\circ\chi_1, \tilde{\rho}_3\circ\tilde{\rho}_2\circ\rho_1, \eta_3\circ( \widehat{\eta_2\circ\hat{\eta}_1\circ U_1^{-1}})\circ U_2^{-1} )$$
		where
		$$U_1:(E_{B_1}\otimes_{\rho_1} B_2)\otimes_{\rho_2}B_3\to E_{B_1}\otimes_{\tilde{\rho}_2\circ \rho_1}B_3\quad\quad U_2:(E_{B_1}\otimes_{\tilde{\rho}_2\circ \rho_1}B_3)\otimes_{\rho_3}B_4\to E_{B_1}\otimes_{\tilde{\rho}_3\circ \tilde{\rho}_2\circ \rho_1} B_4$$
		are the unitaries from Proposition \ref{Prop: Composition and tensor}, and
		$$((\chi_3, \rho_3, \eta_3)\circ(\chi_2, \rho_2, \eta_2))\circ (\chi_1, \rho_1, \eta_1)=(\chi_3\circ\chi_2\circ\chi_1, \widetilde{\rho_3\circ\tilde{\rho}_2}\circ\rho_1, \eta_3\circ\hat{\eta}_2\circ U_3^{-1}\circ \hat{\eta}_1\circ U_4^{-1}).$$
		where
		$$U_3:(E_{B_2}\otimes_{\rho_2} B_3)\otimes_{\rho_3}B_4\to E_{B_2}\otimes_{\tilde{\rho}_3\circ \rho_2}B_4\quad\quad U_4:(E_{B_1}\otimes_{\rho_1}B_2)\otimes_{\tilde{\rho}_3\circ\rho_2}B_4\to E_{B_1}\otimes_{\widetilde{\tilde{\rho}_3\circ \rho_2}\circ \rho_1} B_4$$
		are the unitaries from Proposition \ref{Prop: Composition and tensor}. As $\widetilde{\rho_3\circ\tilde{\rho}_2}\circ\rho_1=\tilde{\rho}_3\circ\tilde{\rho}_2\circ\rho_1$, we just need to verify
		$$\eta_3\circ( \widehat{\eta_2\circ\hat{\eta}_1\circ U_1^{-1}})\circ U_2^{-1}=\eta_3\circ\hat{\eta}_2\circ U_3^{-1}\circ \hat{\eta}_1\circ U_4^{-1}.$$
		
		As $\widehat{(\eta_2\circ\hat{\eta}_1\circ U_1^{-1})}=\hat{\eta}_2\circ\hat{\hat{\eta}}_1\circ\widehat{U_1^{-1}}$, it suffices to show
		$$\hat{\hat{\eta}}_1\circ\widehat{U_1^{-1}}\circ U_2^{-1}=U_3^{-1}\circ\hat{\eta}_1\circ U_4^{-1}$$
		As $U_3$ and $U_2$ are unitaries, it is equivalent to show
		$$U_3\circ\hat{\hat{\eta}}_1=\hat{\eta}_1\circ U_4^{-1}\circ U_2\circ \hat{U}_1$$
		As $U_3\circ\hat{\hat{\eta}}$ and $\hat{\eta}_1\circ U_4^{-1}\circ U_2\circ\hat{U}_1$ are continuous and linear, it suffices to show the two maps agree simple tensors $((x\dot{\otimes}b_2)\dot{\otimes}b_3)\dot{\otimes}b_4\in ((E_{B_1}\otimes_{\rho_1}B_2)\otimes_{\rho_2}B_3)\otimes_{\rho_3}B_4$.
		Given such an element, we have
		\begin{align*}
			(U_3\circ\hat{\hat{\eta}}_1)(((x\dot{\otimes}b_2)\dot{\otimes}b_3)\dot{\otimes}b_4)&=U_3(((\eta(x\dot{\otimes}b_2))\dot{\otimes}b_3)\dot{\otimes}b_4)=\eta_1(x\dot{\otimes}b_2)\dot{\otimes}\rho_3(b_3)b_4\\
			&=\hat{\eta}_1((x\dot{\otimes}b_2)\dot{\otimes}\rho_3(b_3)b_4)=(\hat{\eta}_1\circ U_4^{-1})(x\dot{\otimes}\rho_3(\rho_2(b_2))\rho_3(b_3)b_4)\\
			&=(\hat{\eta}_1\circ U_4^{-1})(x\dot{\otimes}\rho_3(\rho_2(b_2)b_3)b_4)=(\hat{\eta}_1\circ U_4^{-1}\circ U_2)((x\dot{\otimes}\rho_2(b_2)b_3)\dot{\otimes}b_4)\\
			&=(\hat{\eta}_1\circ U_4^{-1}\circ U_2\circ U_1)(((x\dot{\otimes}b_2)\dot{\otimes}b_3)\dot{\otimes}b_4)
		\end{align*}
		Therefore $U_3\circ\hat{\hat{\eta}}_1=\hat{\eta}_1\circ U_4^{-1}\circ U_2\circ \hat{U}_1$ proving composition is associative.
		
	\end{proof}
	
	\section{Functoriality of the KSGNS Construction}
	
	Having constructed the category $\textnormal{PosCor}$, we turn our attention to addressing the relation between the KSGNS construction and the category $\textnormal{PosCor}$. As we will see shortly, we are able to view the KSGNS construction as an idemponent endofunctor, but this functor is not defined on all of $\textnormal{PosCor}$. Instead, we are required to restrict to the subcategory $\textnormal{PosCor}_{*\textnormal{-iso}}$ of $\textnormal{PosCor}$ for which morphisms $(\chi, \rho, \eta)$ have the property that $\chi$ is a $*$-algebra isomorphism so that we obtain an assignment on the morphisms. Note, we can define a similar subcategory of $\textnormal{PosCor}(-, B)$ which we denote by $\textnormal{PosCor}_{*\textnormal{-iso}}(-, B)$. Furthermore, we observe the interior tensor product functors map $\textnormal{PosCor}_{*\textnormal{-iso}}(-, B)$ into $\textnormal{PosCor}_{*\textnormal{-iso}}(-, C)$.
	
	We proceed to show the KSGNS construction defines an idempotent endofunctor on $\textnormal{PosCor}_{*\textnormal{-iso}}$ in two steps. First, we show the KSGNS construction defines an idempotent endofunctor on $\textnormal{PosCor}_{*\textnormal{-iso}}(-, B)$. Second, we verify the KSGNS construction and functors coming from the interior tensor product in Theorem \ref{Theorem: Tensor Functor} commute. Using both of these results, we then prove the KSGNS construction defines an idempotent endofunctor on $\textnormal{PosCor}_{*\textnormal{-iso}}$.
	
	\subsection{KSGNS construction and $\textnormal{PosCor}_{*\textnormal{-iso}}(-, B)$}
	
	We now show the KSGNS construction defines an idempotent endofunctor on $\textnormal{PosCor}_{*\textnormal{-iso}}(-, B)$. We begin first showing the KSGNS construction defines an endofunctor on $\textnormal{PosCor}_{*\textnormal{-iso}}(-, B)$. Given an object $(A, B, (E, \phi))$ in $\textnormal{PosCor}_{*\textnormal{-iso}}(-, B)$, the KSGNS construction yields an object $(A, B, (F_\phi, \pi_\phi))\in \textnormal{PosCor}(-, B)$. Therefore at the level of objects, the KSGNS construction provides an assignment. 
	
	At the level of morphisms, the assignment becomes a bit technical. If we have a morphism $(\rho, \eta):(A, B, (E_1, \phi_1))\to (A, B, (E_2, \phi_2))$, then we want to define the map
	$$\tilde{\eta}:\frac{A\otimes_{\textnormal{alg}}E_1}{N_{\phi_1}}\to F_{\phi_2}\quad\quad \tilde{\eta}(a\dot{\otimes} x)=\rho(a)\dot{\otimes} \eta(x)$$ 
	Provided $\tilde{\eta}$  is well-defined, it is $B$-linear. To extend $\tilde{\eta}$ to all of $F_{\phi_1}$ it suffices to know $\tilde{\eta}$ is bounded. As $\rho$ is a $*$-algebra isomorphism, we can us $\rho^{-1}$ to construct a similar map for $\eta^*$ to know $\tilde{\eta}$ is adjointable. Thus, the two hurdles we need to overcome to obtain an assignment at the level of morphisms comes down to $\tilde{\eta}$ being both well-defined and bounded. The following two lemmas will help us in establishing these two items.
	
	\begin{lemma}\label{Lemma: Properties Lemma}\

		\noindent Let $A_1$, $A_2$, and $B$ be $C^*$-algebras. For $i=1, 2$, let $E_i$ be Hilbert $B$-module, and let $\phi_i:A_i\to \mathcal{L}(E_i)$ be a positive map. If $\eta\in\mathcal{L}(E_1, E_2)$ and $\rho:A_1\to A_2$ is a $*$-algebra homomorphism, and the pair satisfies the condition that for all $a\in A_1$, $\phi_2(\rho(a))\circ\eta=\eta\circ\phi_1(a)$, then
		\begin{enumerate}
			\item for all $a\in A_1$, $\eta^*\circ \phi_2(\rho(a))=\phi_1(a)\circ \eta^*$.
			\item for all $a\in A_1$, $\phi_1(a)$ commutes with $\eta^*\eta$ and $\phi_2(\rho(a))$ commutes with $\eta\eta^*$.
			\item for all $a\in A_1$, 
			$$0\leq \phi_1(a^*a)\eta^*\eta\leq ||\eta||^2\phi_1(a^*a)\quad \textnormal{ and }\quad 0\leq \phi_2(\rho(a^*a))\eta\eta^*\leq ||\eta||^2\phi_2(\rho(a^*a)).$$
		\end{enumerate}
		
	\end{lemma}
	
	\begin{proof}\
		
		\begin{enumerate}
			\item  Fix $a\in A_1$. By assumption,  $\eta\circ\phi_1(a^*)=\phi_2(\rho(a^*))\circ\eta$. Taking the adjoint of both sides yields $\phi_1(a^*)^*\circ\eta^*=\eta^*\circ\phi_2(\rho(a^*))^*$. As $\phi_1$ and $\phi_2$ are positive maps, then the maps preserve the adjoint operation. Using this as well as that $\rho$ is a $*$-algebra homomorphism, we have
			$$\phi_1(a)\circ\eta^*=\phi_1(a^*)^*\circ\eta^*=\eta^*\circ\phi_2(\rho(a^*))^*=\eta^*\circ \phi_2(\rho(a))$$
			\item Fix $a\in A_1$, then $\phi_1(a)\eta^*\eta=\eta^*\phi_2(\rho(a))\eta=\eta^*\eta\phi_1(a)$ showing $\phi_1(a)$ commutes with $\eta^*\eta$. By a similar argument, we have $\phi_2(\rho(a))$ commutes with $\eta\eta^*$.
			\item Fix $a\in A_1$. Since $\phi_1$ is positive, then $\phi_1(a^*a)$ is positive and commutes with $\eta^*\eta$. Hence $\phi_1(a^*a)\eta^*\eta$ is positive. Using Gelfand duality, we know $\eta^*\eta$ commutes with $\sqrt{\phi_1(a^*a)}$ as well as $\eta^*\eta\leq ||\eta^*\eta|| I=||\eta||^2I$. Hence
			$$\phi_1(a^*a)\eta^*\eta=\sqrt{\phi_1(a^*a)}\eta^*\eta \sqrt{\phi_1(a^*a)}\leq \sqrt{\phi_1(a^*a)}||\eta||^2 I\sqrt{\phi_1(a^*a)}=||\eta||^2\phi_1(a^*a).$$
			A similar argument proves the other inequality.
		\end{enumerate}
		
	\end{proof}
	\begin{lemma}\label{Lemma: Bound 2}\

		\noindent Let $B$ be $C^*$-algebra, and let $(A_1, B, (E_1, \phi_1)), (A_2, B, (E_2, \phi_2))\in\textnormal{PosCor}(-, B)$. If $(\rho, \eta)$ is a morphism from $(A_1, B, (E_1, \phi_1))$ to $(A_2, B, (E_2, \phi_2))$, then
		\begin{enumerate}
			\item for all $a_1, \ldots, a_n\in A_1$ and $x_1, \ldots, x_n\in E_1$,
			$$\left\|\sum\limits_{i, j} \langle x_i, \phi_1(a_i^*a_j)\eta^*\eta x_j\rangle_{E_1}\right\|_B\leq ||\eta||^2\left\|\sum\limits_{i, j} \langle x_i, \phi_1(a_i^*a_j) x_j\rangle_{E_1}\right\|_B$$
			\item for all $a_1, \ldots, a_n\in A_1$ and $y_1, \ldots, y_n\in E_2$,
			$$\left\|\sum\limits_{i, j} \langle y_i, \phi_2(\rho(a_i)^*\rho(a_j))\eta\eta^* y_j\rangle_{E_2}\right\|_B\leq ||\eta||^2\left\|\sum\limits_{i, j} \langle y_i, \phi_2(\rho(a_i)^*\rho(a_j)) y_j\rangle_{E_2}\right\|_B$$
		\end{enumerate}
		
	\end{lemma}
	
	\begin{proof}\
		
		\noindent Fix $n\in\mathbb{N}$. For the following proof, we view $E_i^n$ as a Hilbert $B$-module. As $\eta:E_1\to E_2$ is an adjointable map, then 
		$$\hat{\eta}:E_1^n\to E_2^n\quad\quad \hat{\eta}\left(\begin{bmatrix}
			x_1\\
			\vdots\\
			x_n
		\end{bmatrix}\right)=\begin{bmatrix}
			\eta(x_1)\\
			\vdots\\
			\eta(x_n)
		\end{bmatrix}$$
		is an adjointable map with  adjoint $\widehat{\eta^*}$ and $||\hat{\eta}||=||\eta||$. As $\phi_i:A_i\to \mathcal{L}(E_i)$ is a completely positive map, then we have a positive map
		$$\phi_{k, n}:M_n(A_k)\to \mathcal{L}(E_k^n)\cong M_n(\mathcal{L}(E_k))\quad\quad \phi_{k, n}([a^i_j])\begin{bmatrix}
			x_1\\
			\vdots\\
			x_n
		\end{bmatrix}=\begin{bmatrix}
			\sum\limits_{j=1}^n \phi_k(a^1_j)x_j\\
			\vdots\\
			\sum\limits_{j=1}^n \phi_k(a^n_j)x_j
		\end{bmatrix}$$
		As $\rho:A_1\to A_2$ is a $*$-algebra homomorphism, then $\rho_n:M_n(A_1)\to M_n(A_2)$ is also a $*$-algebra homomorphism. Using the maps $\hat{\eta}$, $\phi_{i, n}$, and $\rho_n$, we make the following identifications: for all $x_1, \ldots, x_n\in E_1$ and $a_1, \ldots, a_n\in A_1$
		\begin{align*}
			\sum\limits_{i, j} \langle x_i, \phi_1(a_i^*a_j)&\eta^*\eta x_j\rangle_{E_1}\\
			&=\left\langle\begin{bmatrix}
				x_1\\
				\vdots\\
				x_n
			\end{bmatrix}, \phi_{1, n}\left(\begin{bmatrix}
				a_1 & \ldots & a_n\\
				0 & \ldots & 0\\
				\vdots & & \vdots\\
				0 & \ldots & 0
			\end{bmatrix}^*\begin{bmatrix}
				a_1 & \ldots & a_n\\
				0 & \ldots & 0\\
				\vdots & & \vdots\\
				0 & \ldots & 0
			\end{bmatrix}\right)\hat{\eta}^*\hat{\eta}\begin{bmatrix}
				x_1\\
				\vdots\\
				x_n
			\end{bmatrix}\right\rangle_{E_1^n}
		\end{align*}
		and for all $y_1, \ldots, y_n\in E_2$ and $a_1, \ldots, a_n\in A_1$
		\begin{align*}
			\sum\limits_{i, j} \langle y_i, &\phi_2(\rho(a_i)^*\rho(a_j))\eta\eta^* y_j\rangle_{E_2}\\
			&=\left\langle\begin{bmatrix}
				y_1\\
				\vdots\\
				y_n
			\end{bmatrix}, \phi_{2, n}\left(\rho_n\left(\begin{bmatrix}
				a_1 & \ldots & a_n\\
				0 & \ldots & 0\\
				\vdots & & \vdots\\
				0 & \ldots & 0
			\end{bmatrix}^*\begin{bmatrix}
				a_1 & \ldots & a_n\\
				0 & \ldots & 0\\
				\vdots & & \vdots\\
				0 & \ldots & 0
			\end{bmatrix}\right)\right)\hat{\eta}\hat{\eta}^*\begin{bmatrix}
				y_1\\
				\vdots\\
				y_n
			\end{bmatrix}\right\rangle_{E_2^n}.
		\end{align*}
		
		Using $\eta\circ \phi_1(a)=\phi_2(\rho(a))\circ \eta$ for all $a\in A_1$, we have $\hat{\eta}\circ \phi_{1, n}([a^i_j])=\phi_{2, n}[\rho_n([a^i_j])]\circ\hat{\eta}$ for all $[a^i_j]\in M_n(A_1)$. Indeed, given $x_1, \ldots, x_n\in E_1$
		\begin{align*}
			\left(\hat{\eta}\circ \phi_{1, n}([a^i_j])\right)\begin{bmatrix}
				x_1\\
				\vdots\\
				x_n
			\end{bmatrix}&=\hat{\eta}\left(\begin{bmatrix}
				\sum\limits_{j=1}^n \phi(a^1_j)x_j\\
				\vdots\\
				\sum\limits_{j=1}^n \phi(a^n_j)x_j
			\end{bmatrix}\right)=\begin{bmatrix}
				\sum\limits_{j=1}^n \eta\phi_1(a^1_j)x_j\\
				\vdots\\
				\sum\limits_{j=1}^n \eta\phi_1(a^n_j)x_j
			\end{bmatrix}=\begin{bmatrix}
				\sum\limits_{j=1}^n \phi_2(\rho(a^1_j))\eta x_j\\
				\vdots\\
				\sum\limits_{j=1}^n \phi_2(\rho(a^n_j))\eta x_j
			\end{bmatrix}\\
			&=\phi_{2, n}([\rho (a^i_j)])\begin{bmatrix}
				\eta x_1\\
				\vdots\\
				\eta x_n
			\end{bmatrix}=\phi_{2, n}(\rho_n[ a^i_j])\begin{bmatrix}
				\eta x_1\\
				\vdots\\
				\eta x_n
			\end{bmatrix}=\phi_{2, n}(\rho_n([a^i_j]))\hat{\eta}\begin{bmatrix}
				x_1\\
				\vdots\\
				x_n
			\end{bmatrix}
		\end{align*}
		We now apply Lemma \ref{Lemma: Properties Lemma} to know for all $[a^i_j]\in M_n(A_1)$
		$$0\leq \phi_{1, n}([a^i_j]^*[a^i_j])\hat{\eta}^*\hat{\eta}\leq ||\hat{\eta}||^2\phi_{1, n}([a^i_j]^*[a^i_j])=||\eta||^2\phi_{1, n}([a^i_j]^*[a^i_j])$$
		and
		$$0\leq \phi_{2, n}(\rho_n([a^i_j]^*[a^i_j]))\hat{\eta}\hat{\eta}^*\leq ||\hat{\eta}||^2\phi_{2, n}(\rho_n([a^i_j]^*[a^i_j]))=||\eta||^2\phi_{2, n}(\rho_{n}([a^i_j]^*[a^i_j])).$$
		Therefore, for all $x_1, \ldots, x_n\in E_1$ and $a_1, \ldots, a_n\in A_1$, we have
		\begin{align*}
			0&\leq \left\langle\begin{bmatrix}
				x_1\\
				\vdots\\
				x_n
			\end{bmatrix}, \phi_{1, n}\left(\begin{bmatrix}
				a_1 & \ldots & a_n\\
				0 & \ldots & 0\\
				\vdots & & \vdots\\
				0 & \ldots & 0
			\end{bmatrix}^*\begin{bmatrix}
				a_1 & \ldots & a_n\\
				0 & \ldots & 0\\
				\vdots & & \vdots\\
				0 & \ldots & 0
			\end{bmatrix}\right)\hat{\eta}^*\hat{\eta}\begin{bmatrix}
				x_1\\
				\vdots\\
				x_n
			\end{bmatrix}\right\rangle_{E_1^n}\\
			&\leq ||\eta||^2\left\langle\begin{bmatrix}
				x_1\\
				\vdots\\
				x_n
			\end{bmatrix}, \phi_{1, n}\left(\begin{bmatrix}
				a_1 & \ldots & a_n\\
				0 & \ldots & 0\\
				\vdots & & \vdots\\
				0 & \ldots & 0
			\end{bmatrix}^*\begin{bmatrix}
				a_1 & \ldots & a_n\\
				0 & \ldots & 0\\
				\vdots & & \vdots\\
				0 & \ldots & 0
			\end{bmatrix}\right)\begin{bmatrix}
				x_1\\
				\vdots\\
				x_n
			\end{bmatrix}\right\rangle_{E_1^n}
		\end{align*}
		showing
		$$0\leq\sum\limits_{i, j} \langle x_i, \phi_1(a_i^*a_j)\eta^*\eta x_j\rangle_{E_1} \leq ||\eta||^2\sum\limits_{i, j}\langle x_i, \phi_1(a^*_ia_j)x_j\rangle_{E_1}$$
		Hence
		$$\left\|\sum\limits_{i, j} \langle x_i, \phi_1(a_i^*a_j)\eta^*\eta x_j\rangle_{E_1}\right\|_B \leq||\eta||^2 \left\|\sum\limits_{i, j}\langle x_i, \phi_1(a^*_ia_j)x_j\rangle_{E_1}\right\|_B$$
		Similarly, for all $y_1, \ldots, y_n\in E_2$ and $a_1, \ldots, a_n\in A_1$ 
		\begin{align*}
			0&\leq \left\langle\begin{bmatrix}
				y_1\\
				\vdots\\
				e_n
			\end{bmatrix}, \phi_{2, n}\left(\rho_n\left(\begin{bmatrix}
				a_1 & \ldots & a_n\\
				0 & \ldots & 0\\
				\vdots & & \vdots\\
				0 & \ldots & 0
			\end{bmatrix}^*\begin{bmatrix}
				a_1 & \ldots & a_n\\
				0 & \ldots & 0\\
				\vdots & & \vdots\\
				0 & \ldots & 0
			\end{bmatrix}\right)\right)\hat{\eta}\hat{\eta}^*\begin{bmatrix}
				y_1\\
				\vdots\\
				y_n
			\end{bmatrix}\right\rangle_{E_2^n}\\
			&\leq ||\eta||^2\left\langle\begin{bmatrix}
				y_1\\
				\vdots\\
				y_n
			\end{bmatrix}, \phi_{2, n}\left(\rho_n\left(\begin{bmatrix}
				a_1 & \ldots & a_n\\
				0 & \ldots & 0\\
				\vdots & & \vdots\\
				0 & \ldots & 0
			\end{bmatrix}^*\begin{bmatrix}
				a_1 & \ldots & a_n\\
				0 & \ldots & 0\\
				\vdots & & \vdots\\
				0 & \ldots & 0
			\end{bmatrix}\right)\right)\begin{bmatrix}
				y_1\\
				\vdots\\
				y_n
			\end{bmatrix}\right\rangle_{E_2^n}
		\end{align*}
		showing
		$$0\leq\sum\limits_{i, j} \langle y_i, \phi_1(\rho(a_i)^*\rho(a_j))\eta\eta^* y_j\rangle_{E_2} \leq ||\eta||^2\sum\limits_{i, j}\langle y_i, \phi_2(\rho(a_i)^*\rho(a_j))y_j\rangle_{E_2}$$
		Hence
		$$\left\|\sum\limits_{i, j} \langle y_i, \phi_2(\rho(a_i)^*\rho(a_j))\eta\eta^* y_j\rangle_{E_2}\right\|_B \leq||\eta||^2 \left\|\sum\limits_{i, j}\langle y_i, \phi_2(\rho(a_i)^*\rho(a_j))y_j\rangle_{E_2}\right\|_B$$
	\end{proof}
	
	\begin{proposition}\label{Prop: Lifting 2}\

		\noindent Let $B$ be $C^*$-algebra. For $i=1, 2$, let  $R_i=(A_i, B, (E_i, \phi_i))\in\textnormal{PosCor}(-, B)$, let $(F_{\phi_i},\pi_{\phi_i}, V_{\phi_i})$ be the triple from applying the KSGNS construction to  $\phi_i:A_i\to \mathcal{L}(E_i)$, and denote the induced objects as $R_{\phi_i}=(A_i, B, (F_{\phi_i}, \pi_{\phi_i}))\in\textnormal{PosCor}(-, B)$. If $(\rho, \eta)\in \textnormal{Hom}_{\textnormal{PosCor}(-, B)}(R_1, R_2)$ such that $\rho:A_1\to A_2$ is an isomorphism, then there exists $\tilde{\eta}\in\mathcal{L}(F_{\phi_1}, F_{\phi_2})$ such that
		\begin{enumerate}
			\item for all $a\in A_1$ and $x\in E_1$, $\tilde{\eta}( a\dot{\otimes} x)=\rho(a)\dot{\otimes} \eta(x)$.
			\item for all $a\in A_2$ and $y\in E_2$, $\tilde{\eta}^*( a\dot{\otimes} y)=\rho^{-1}(a)\dot{\otimes} \eta^*(y)$.
			\item $(\rho,\tilde{\eta})\in \textnormal{Hom}_{\textnormal{PosCor}(-, B)}(R_{\phi_1}, R_{\phi_2})$
			\item $\tilde{\eta}\circ V_{\phi_1}=V_{\phi_2}\circ \eta$.
		\end{enumerate}
		
	\end{proposition}
	
	\begin{proof}\
		
		\noindent As the map
		$$A_1\times E_1\to F_{\phi_2}\quad\quad (a, x)\mapsto \rho(a)\dot{\otimes}\eta(x)$$
		is bilinear and $B$-linear in the second slot, then we obtain a unique $B$-linear map
		$$\xi:A_1\otimes_{\textnormal{alg}}E_1\to F_{\phi_2}\quad\quad \xi\left(\sum\limits_{k=1}^n a_k\otimes x_k\right)=\sum\limits_{k=1}^n \rho(a_k)\dot{\otimes} \eta(x_k)$$
		We observe
		\begin{align*}
			\left\|\xi\left(\sum\limits_{k=1}^n a_k\otimes x_k\right)\right\|_{\phi_2}^2&=\left\|\sum\limits_{k=1}^n\rho(a_k)\dot{\otimes} \eta(x_k)\right\|_{\phi_2}^2=\left\|\sum\limits_{i, j}\langle \eta(x_i), \phi_2(\rho(a_i^*)\rho(a_j))\eta(x_j)\rangle_{E_2}\right\|_B\\
			&=\left\|\sum\limits_{i, j}\langle x_i, \eta^*\phi_2(\rho(a_i^*a_j))\eta(x_j)\rangle_{E_1}\right\|_B=\left\|\sum\limits_{i, j}\langle x_i, \phi_1(a_i^*a_j)\eta^*\eta x_k\rangle_{E_1}\right\|_B\\
			&\leq ||\eta||^2\left\|\sum\limits_{i, j}\langle x_i, \phi_1(a_i^*a_j) x_k\rangle_{E_1}\right\|_B=||\eta||^2\left\| \sum\limits_{k=1}^n a_k\otimes x_k\right\|_{\phi_1}^2
		\end{align*}
		where the inequality follows from Lemma \ref{Lemma: Bound 2}. Therefore we obtain a well-defined bounded $B$-linear map
		$$\tilde{\xi}:\frac{A_1\otimes_{\textnormal{alg}}E_1}{N_{\phi_1}}\to F_{\phi_2}\quad\quad \tilde{\xi}\left(\sum\limits_{k=1}^n a_k\dot{\otimes} x_k\right)=\sum\limits_{k=1}^n \rho(a_k)\dot{\otimes}\eta(x_k)$$
		Denote the extension of $\tilde{\xi}$ to $F_{\phi_1}$ by $\tilde{\eta}$. By a similar argument, we obtain a well-defined bounded $B$-linear map
		$$\tilde{\gamma}:F_{\phi_2}\to F_{\phi_1}\quad\quad \tilde{\gamma}\left(\sum\limits_{k=1}^n a_k\dot{\otimes} y_k\right)=\sum\limits_{k=1}^n \rho^{-1}(a_k)\dot{\otimes}\eta^*(y_k).$$

		We claim $\tilde{\eta}=\tilde{\gamma}^*$. As $\frac{A_1\otimes_{\textnormal{alg}} E_1}{N_{\phi_1}}\subset F_{\phi_1}$ and $\frac{A_2\otimes_{\textnormal{alg}} E_2}{N_{\phi_2}}\subset F_{\phi_2}$ are dense and both $\tilde{\gamma}$ and $\tilde{\eta}$ are continuous, then it suffices to check $\tilde{\eta}=\tilde{\gamma}^*$ on these subspaces. Using linearity, it suffices to prove the claim on elements of the form $a_1\dot{\otimes} x\in \frac{A_1\otimes_{\textnormal{alg}} E_1}{N_{\phi_1}}$ and $a_2\dot{\otimes} y\in\frac{A_2\otimes_{\textnormal{alg}} E_2}{N_{\phi_2}}$. Given such elements, we have
		\begin{align*}
			\langle\tilde{\gamma}(a_2\dot{\otimes} y), a_1\dot{\otimes} x\rangle_{\phi_1}&=\langle \rho^{-1}(a_2)\dot{\otimes} \eta^*(y), a_1\dot{\otimes} x\rangle_{\phi_1}=\langle \eta^*(y), \phi_1(\rho^{-1}(a_2)^*a_1) x\rangle_{E_1}\\
			&=\langle y, \eta\phi_1(\rho^{-1}(a_2)^*a_1)x\rangle_{E_2}=\langle  y, \phi_2(a_2^*\rho(a_1))\eta(x)\rangle_{E_2}\\
			&=\langle a_2\dot{\otimes} y, \rho(a_1)\dot{\otimes} \eta(x)\rangle_{\phi_2}=\langle a_2\dot{\otimes} y, \tilde{\eta}(a_1\dot{\otimes} x)\rangle_{\phi_2}
		\end{align*}
		Thus $\tilde{\eta}=\tilde{\gamma}^*$ showing $\tilde{\eta}\in \mathcal{L}(F_{\phi_1}, F_{\phi_2})$. From the construction of $\tilde{\eta}$ and $\tilde{\eta}^*$, we know $\tilde{\eta}$ satisfies $(1)$ and $(2)$ of the assertion.
		
		We now verify $(\rho, \tilde{\eta})$ intertwines $\pi_{\phi_1}$ and $\pi_{\phi_2}$. Fix $a'\in A_1$. As $\pi_{\phi_2}(\rho(a'))\circ\tilde{\eta}$ and $\tilde{\eta}\circ \pi_{\phi_1}(a)$ are both continuous and $\frac{A_1\otimes_{\textnormal{alg}}E_1}{N_{\phi_1}}$ is dense in $F_{\phi_1}$, then it suffices to check the two maps are the same on the dense subspace. Using linearity, it suffices to show the claim holds on elements of the form $a\dot{\otimes} x\in F_{\phi_1}$. Given such an element, we have
		\begin{align*}
			(\pi_{\phi_2}(\rho(a'))\circ\tilde{\eta})(a\dot{\otimes} x)&=\pi_{\phi_2}(\rho(a'))(\rho(a)\dot{\otimes} \eta(x))=\rho(a')\alpha(a)\dot{\otimes} \eta(x)\\
			&=\rho(a'a)\dot{\otimes} \eta(x)=\tilde{\eta}(a'a\dot{\otimes} x)=(\tilde{\eta}\circ\pi_{\phi_1}(a'))(a\dot{\otimes} x)
		\end{align*}
		Therefore we conclude $(\rho, \tilde{\eta})$ intertwines $\pi_{\phi_1}$ and $\pi_{\phi_2}$. Hence $(\rho, \tilde{\eta})\in\textnormal{Hom}_{\textnormal{PosCor}(-, B)}(F_{\phi_1}, F_{\phi_2})$.
		
		Finally, we verify $(4)$. Let $(a_\lambda)_{\lambda\in \Lambda}$ be an approximate unit in $A_1$, then $(\rho(a_\lambda))_{\lambda\in\Lambda}$ is an approximate unit in $A_2$. Therefore, given $x\in E_1$ we have
		\begin{align*}
			\tilde{\eta}V_{\phi_1}x&=\lim\limits_{\lambda\to \infty}\tilde{\eta}\left(a_\lambda\dot{\otimes} x\right)=\lim\limits_{\lambda\to \infty}\rho(a_\lambda)\dot{\otimes} \eta(x)=V_{\phi_2}(\eta(x))
		\end{align*}
		proving the claim.
		
	\end{proof}	
	
	\begin{theorem}\label{Theorem: KSGNS Endofucntor}\
		
		\noindent Let $B$ be a $C^*$-algebra. Then the assignment
		\begin{align*}
			\textnormal{KSGNS}:\textnormal{PosCor}_{*\textnormal{-iso}}(-, B)&\to \textnormal{PosCor}_{*\textnormal{-iso}}(-, B)
			\\
			(A, B, (E, \phi))&\mapsto (A, B, (F_\phi, \pi_\phi))\\
			(\rho, \eta)&\mapsto (\tilde{\eta}, \alpha)
		\end{align*}
		defines a covariant endofunctor on $\textnormal{PosCor}(-, B)$. 
		
	\end{theorem}
	
	\begin{proof}\
		
		\noindent The KSGNS construction shows the assignment is well-defined on objects. Proposition \ref{Prop: Lifting 2} shows the assignment on morphisms is well-defined. From the construction of $\tilde{\eta}$ in Proposition \ref{Prop: Lifting 2}, it is clear the assignment is functorial in a covariant manner.
		
	\end{proof}
	
	We now show the endofunctor $\textnormal{KSGNS}$ is idempotent in the sense that $\textnormal{KSGNS}^2$ is naturally isomorphic to $\textnormal{KSGNS}$. The proof of this claim comes down to the uniqueness of the data obtained from the KSGNS construction. Suppose $\phi:A\to \mathcal{L}(E)$ is a strict completely positive map for $E$ a Hilbert $B$ module. Applying the KSGNS construction to this correspondence yields a unitarily unique triple $(F_\phi, \pi_\phi, V_\phi)$. As $(F_\phi, \pi_\phi)$ is a $C^*$-correspondence, we can apply the KSGNS construction again to obtain a unitarily unique triple $(F_{\pi_\phi}, \pi_{\pi_\phi}, V_{\pi_\phi})$. Since
	$(F_\phi, \pi_\phi, I_{F_\phi})$ satisfy the two defining conditions for the triple when applying the KSGNS construction to $(F_\phi, \pi_\phi)$, then there exists a unitary $U\in \mathcal{L}(F_\phi, F_{\pi_\phi})$ such that for all $a\in A$, $\pi_{\pi_\phi}(a)=U\pi_\phi(a)U^*$ and $V_{\pi_{\phi}}=UI_{F_\phi}=U$. Thus the pair $(1_A, V_{\pi_\phi})$ defines an isomorphism from $(A, B, (\pi_\phi, F_\phi))$ to $(A, B, (\pi_{\pi_\phi}, F_{\pi_{\pi_\phi}}))$ in $\textnormal{PosCor}(-, B)$. Using these isomorphisms and an approximate unit in $A$, we obtain a natural isomorphism from $\textnormal{KSGNS}$ to $\textnormal{KSGNS}^2$.
	
	\begin{proposition}\label{Prop: Idempotency of KSGSN}\

		\noindent Let $B$ be a $C^*$-algebra. As functors on $\textnormal{PosCor}_{*\textnormal{-iso}}(-, B)$, $\textnormal{KSGNS}^2$ is naturally isomorphic to $\textnormal{KSGNS}$.
		
	\end{proposition}
	
	\begin{proof}\
		
		\noindent Fix objects $(A_1, (E_1, \pi_1)), (A_2, (E_2, \phi_2))\in \textnormal{PosCor}(-, B)$, and let $(\rho, \eta)$ be a morphism from $(A_1, B, (E_1, \phi_1))$ to $(A_2, B, (E_2, \phi_2))$. To prove naturality between the two functors, we verify the following diagram commutes:
		\begin{center}
			\begin{tikzcd}
				{(A_1, B, (F_{\phi_1}, \pi_{\phi_1}))} &&& {(A_1, B, (F_{\pi_{\phi_1}}, \pi_{\pi_{\phi_1}}))} \\
				\\
				\\
				{(A_2, B, (F_{\phi_2}, \pi_{\phi_2}))} &&& {(A_2, B, (F_{\pi_{\phi_2}}, \pi_{\pi_{\phi_2}}))}
				\arrow["{(1_{A_1}, V_{\pi_{\phi_1}})}"{description}, from=1-1, to=1-4]
				\arrow["{(\rho, \tilde{\eta})}"{description}, from=1-1, to=4-1]
				\arrow["{(\rho, \tilde{\tilde{\eta}})}"{description}, from=1-4, to=4-4]
				\arrow["{(1_{A_2}, V_{\pi_{\phi_2}})}"{description}, from=4-1, to=4-4]
			\end{tikzcd}
		\end{center}
		From how we have defined composition, we only need check $V_{\pi_{\phi_2}}\circ\tilde{\eta}=\tilde{\tilde{\eta}}\circ V_{\pi_{\phi_1}}$. As $\frac{A_1\otimes_{\textnormal{alg}} E_1}{N_{\phi_1}}$ is a dense subspace of $F_{\phi_1}$ and both the $V_{\pi_{\phi_2}}\circ\tilde{\eta}$ and $\tilde{\tilde{\eta}}\circ V_{\pi_{\phi_1}}$ are continuous, it suffices to show the equality on $\frac{A_1\otimes_{\textnormal{alg}} E_1}{N_{\phi_1}}$. Using linearity, it suffices to show equality on elements of the form $a\dot{\otimes} x\in \frac{A_1\otimes_{\textnormal{alg}}E_1}{N_{\phi_1}}$. Let $(a_\lambda)_{\lambda\in\Lambda}$ be an approximate unit for $A_2$, then $(\rho^{-1}(a_\lambda))_{\lambda\in\Lambda}$ is an approximate unit in $A_1$. Therefore
		\begin{align*}
			(V_{\pi_{\phi_2}}\circ\tilde{\eta})(a\dot{\otimes} x)&=V_{\pi_{\phi_2}}(\rho(a)\dot{\otimes} \eta(x))=\lim\limits_{\lambda\to \infty} a_{\lambda}\dot{\otimes} (\rho(a)\dot{\otimes} \eta(x))\\
			&=\lim\limits_{\lambda\to \infty} a_\lambda\dot{\otimes} \tilde{\eta}(a\dot{\otimes} x)=\lim\limits_{\lambda\to \infty} \tilde{\tilde{\eta}}\left(\rho^{-1}(a_\lambda)\dot{\otimes} (a\dot{\otimes} x)\right)=(\tilde{\tilde{\eta}}\circ V_{\pi_{\phi_1}})(a\dot{\otimes} x)
		\end{align*}
		allowing us to conclude the diagram does indeed commute.
		
	\end{proof}
	
	\subsection{KSGNS construction and the interior tensor product}
	
	Having shown the KSGNS construction defines an idempotent endofunctor on $\textnormal{PosCor}_{*\textnormal{-iso}}(, B)$, we would like to extend the result to $\textnormal{PosCor}_{*\textnormal{-iso}}$. Due to the use of the interior tensor product in the morphisms and composition rule of $\textnormal{PosCor}_{*\textnormal{-iso}}$, we start by determining the compatibility of the KSGNS construction and the interior tensor product. As we can realize both the KSGNS construction and interior tensor product as functors on common categories, we can ask whether the two functors commute. As we prove in the following theorem, this is indeed the case.
	
	\begin{theorem}\label{Theorem: KSNGS commutes with ITP}\
		
		\noindent Let $B$ and $D$ be $C^*$-algebras, and let $(F, \pi)$ be a $C^*$-correspondence from $B$ to $D$. As functors from $\textnormal{PosCor}_{*\textnormal{-iso}}(-, B)$ to $\textnormal{PosCor}_{*\textnormal{-iso}}(-, D)$, $\textnormal{KSGNS}\circ (-)\otimes_{\pi}F$ is naturally isomorphic to 
		$(-)\otimes_\pi F\circ\textnormal{KSGNS}$. In particular, for each object $(A, B, (E, \phi))$ in $\textnormal{PosCor}_{*\textnormal{-iso}}(-, B)$, the map
		$$(1_A, V):(A, B, (A\otimes_{\phi\otimes_\pi I}(E\otimes_\pi F), \pi_{\phi\otimes_\pi I}))\to (A, B, ((A\otimes_\phi E)\otimes_\pi F, \pi_\phi\otimes_\pi I))$$
		defines an isomorphism where $V$ is the unitary isomorphism given by
		$$V:A\otimes_{\phi\otimes_\pi I}(E\otimes_\pi F)\to (A\otimes_\phi E)\otimes_{\pi}F\quad\quad V(a\dot{\otimes}(x\dot{\otimes}f))=(a\dot{\otimes}x)\dot{\otimes}f$$
		
	\end{theorem}
	
	\begin{proof}\
		
		\noindent We start by proving the map $V$ is a well-defined unitary. To simplify notation, denote $\phi\otimes_\pi I=\tilde{\phi}$ and $\pi_\phi \otimes_\pi I=\widetilde{\pi_{\phi}}$. Define
		$$U:A\otimes_{\textnormal{alg}}(E\otimes_B F)\to (A\otimes_\phi E)\otimes_\pi F\quad\quad U(a\otimes (x\dot{\otimes} f))=(a\dot{\otimes} x)\dot{\otimes} f$$
		It is clear $U$ is well-defined as well as $C$-linear. Furthermore, as
		\begin{align*}
			\left\|U\left(\sum\limits_{k=1}^n a_k\otimes\left(\sum\limits_{\ell=1}^m x_{\ell, k}\dot{\otimes} f_{\ell, k}\right)\right)\right\|^2&=\left\|\sum\limits_{k=1}^n\sum\limits_{\ell=1}^m(a_k\dot{\otimes} x_{\ell, k})\dot{\otimes} f_{\ell, k}\right\|^2\\
			&=\left\|\sum\limits_{i, j=1}^n\sum\limits_{k, \ell=1}^m\langle f_{k, i}, \pi(\langle x_{k, i}, \phi(a_i^*a_j)x_{\ell, j}\rangle_E)f_{\ell, j}\rangle_F\right\|_C\\
			&=\left\|\sum\limits_{k=1}^n a_k\otimes\left(\sum\limits_{\ell=1}^m x_{\ell, k}\dot{\otimes} f_{\ell, k}\right)\right\|^2,
		\end{align*}
		then $U$ is an isometry with respect to the semi-norm on $A\otimes_{\textnormal{alg}}(E\otimes_B F)$ and norm on $(A\otimes_\phi E)\otimes_\pi F$. Therefore, we obtain a $C$-linear isometry
		$$U:\frac{A\otimes_{\textnormal{alg}}(E\otimes_B F)}{N_{\tilde{\phi}}}\to (A\otimes_\phi E)\otimes_\pi F\quad\quad U(a\dot{\otimes} (x\dot{\otimes} f))=(a\dot{\otimes} x)\dot{\otimes} f$$
		
		We now claim that we can extend $U$ to $A\otimes_{\tilde{\phi}}(E\otimes_\pi F)$ and that the extension is unitary. For all $a\dot{\otimes} x, a\dot{\otimes} y\in \frac{A\otimes_{\textnormal{alg}}(E\otimes_\pi F)}{N_{\tilde{\phi}}}$, we have
		\begin{align*}
			||a\dot{\otimes} (x-y)||^2&=||\langle x-y, \tilde{\phi}(a^*a)(x-y)\rangle_{E\otimes_\pi F}||_C\leq ||\sqrt{\tilde{\phi}(a^*a)}||^2*||x-y||^2_{E\otimes_\pi F}
		\end{align*}
		which implies $\frac{A\otimes_{\textnormal{alg}}(E\otimes_B F)}{N_{\tilde{\phi}}}$ is dense in $\frac{A\otimes_{\textnormal{alg}} (E\otimes_\pi F)}{N_{\tilde{\phi}}}$. Therefore we can extend $U$ to a $C$-linear isometry on $A\otimes_{\tilde{\phi}} (E\otimes_\pi F)$. As the map $U$ has image $(\frac{A\otimes_{\textnormal{alg}} E}{N_\phi})\otimes_{C}F$ which is dense in $(A\otimes_{\tilde{\phi}}E)\otimes_\pi F$, then the extension of $U$ is surjective. Thus the extension defines a unitary map from $A\otimes_{\tilde{\phi}}(E\otimes_\pi F)$ to $(A\otimes_\phi E)\otimes_\pi F$. Denote the extentions  $V$.
		
		Now we verify
		$$(1_A, V):(A, C, (A\otimes_{\tilde{\phi}}(E\otimes_\pi F), \pi_{\tilde{\phi}}))\to (A, C, ((A\otimes_\phi E)\otimes_\pi F, \widetilde{\pi_{\phi}}))$$
		is a morphism in $\textnormal{PosCor}_{*\textnormal{-iso}}(-, C)$. As $1_A$ is a non-degenerate $*$-algebra map which is an isomorphism, we just need to verify $V$ intertwines $\pi_{\tilde{\phi}}$ and $\widetilde{\pi_\phi}$ with respect to $1_A$. Fix $a'\in A$. As both $V\circ \pi_{\tilde{\phi}}(a')$ and $\widetilde{\pi_\phi}(a')\circ V$ are continuous and $\frac{A\otimes_{\textnormal{alg}}(E\otimes_B F)}{N_{\tilde{\phi}}}$ is dense in $A\otimes_{\tilde{\phi}}(E\otimes_\pi F)$, then it suffices to show the two maps agree on this dense subspace. Using linearity, it suffices to check equality on elements of the form $a\dot{\otimes}(x\dot{\otimes}f)\in \frac{A\otimes_{\textnormal{alg}}(E\otimes_B F)}{N_{\tilde{\phi}}}$. Given such an element, we have
		\begin{align*}
			(V\circ \pi_{\tilde{\phi}}(a'))(a\dot{\otimes}(x\dot{\otimes}f))&=U(a'a\dot{\otimes}(x\dot{\otimes}f))=(a'a\dot{\otimes}x)\dot{\otimes}f\\
			&=\widetilde{\pi_\phi}(a')((a\dot{\otimes}x)\dot{\otimes}f)=(\widetilde{\pi_\phi}(a')\circ V)(a\dot{\otimes}(x\dot{\otimes}f))
		\end{align*}
		Hence $(1_A, V)$ defines a morphism. Note, as $V$ is a unitary map, then $(1_A, V)$ is an isomorphism in $\textnormal{PosCor}_{*\textnormal{-iso}}(-, C)$.
		
		Now we verify the isomorphisms satisfy the naturality condition. Fix objects $(A_1, B, (E_1, \phi_1))$ and $(A_2, B, (E_2, \phi_2))$ in $\textnormal{PosCor}_{*\textnormal{-iso}}(-, B)$, and let $(\rho, \eta):(A_1, B, (E_1, \phi_1))\to (A_2, B, (E_2, \phi_2))$ be a morphism. For $i=1, 2,$ let
		$$(1_{A_i}, V):(A_i, C, (A_i\otimes_{\tilde{\phi}_i}(E_i\otimes_\pi F), \pi_{\tilde{\phi}_i}))\to (A_i, C, ((A_i\otimes_{\phi_i} E_i)\otimes_\pi F, \widetilde{\pi_{\phi_i}}))$$
		be the isomorphisms constructed above. To prove naturality, we claim the following diagram commutes:
		\begin{center}
			\begin{tikzcd}
				{(A_1, C, (A_1\otimes_{\tilde{\phi}_1}(E_1\otimes_\pi F), \pi_{\tilde{\phi}_1}))} &&&& {(A_1, ((A_1\otimes_{\phi_1} E)\otimes_\pi F, \tilde{\pi}_{\phi_1}))} \\
				\\ 
				\\
				{(A_2, C, (A_2\otimes_{\tilde{\phi}_2}(E_2\otimes_\pi F), \pi_{\tilde{\phi}_2}))} &&&& {(A_2, C, ((A\otimes_{\phi_2} E)\otimes_\pi F, \tilde{\pi}_{\phi_1}))}
				\arrow["{(1_{A_1}, V_1)}", from=1-1, to=1-5]
				\arrow["{(\rho, \widetilde{\hat{\eta}})}"{description}, from=1-1, to=4-1]
				\arrow["{(\rho, \widehat{\tilde{\eta}})}"{description}, from=1-5, to=4-5]
				\arrow["{(1_{A_2}, V_2)}", from=4-1, to=4-5]
			\end{tikzcd}
		\end{center}
		From how composition is defined, we just need to verify $\widehat{\tilde{\eta}}\circ V_1=V_2\circ\widetilde{\hat{\eta}}$. As both $\widehat{\widetilde{\eta}}\circ V_1$ and $V_2\circ \widetilde{\widehat{\eta}}$ are continuous as well as $\frac{A\otimes_{\textnormal{alg}}(E_1\otimes_BF)}{N_{\tilde{\phi_1}}}$ is dense in $A\otimes_{\tilde{\phi}_1}(E_1\otimes_\pi F)$, it suffices to show equality on the dense subspace. Using linearity, it suffices to show equality on elements of the form $a\dot{\otimes} (x\dot{\otimes} f)$. Given such an element, we have
		\begin{align*}
			(\widehat{\tilde{\eta}}\circ V_1)(a\dot{\otimes}(x\dot{\otimes} f))&=\widehat{\tilde{\eta}}((a\dot{\otimes} x)\dot{\otimes} f)=\tilde{\eta}(a\dot{\otimes} x)\dot{\otimes} f\\
			&=(\rho(a)\dot{\otimes} \eta(x))\dot{\otimes} f=V_2(\rho(a)\dot{\otimes} (\eta(x)\dot{\otimes} f))\\
			&=V_2(\rho(a)\dot{\otimes} \hat{\eta}(x\dot{\otimes} f))=(V_2\circ\widetilde{\hat{\eta}})(a\dot{\otimes} (x\dot{\otimes} f))
		\end{align*}
		Thus the diagram above commutes. As $(1_{A_1}, V_1)$ and $(1_{A_2}, V_2)$ are isomorphisms, then we conclude the two functors in the assertion are indeed naturally isomorphic.
		
	\end{proof}
	
	\subsection{KSGNS construction and $\textnormal{PosCor}$}
	
	We now extend the idempotent KSGNS endofunctor on $\textnormal{PosCor}_{*\textnormal{-iso}}(-, B)$ to $\textnormal{PosCor}_{*\textnormal{-iso}}$. The assignment on objects will be given by the KSGNS construction. For the assignment on morphisms, we leverage the KSGNS endofunctor on $\textnormal{PosCor}_{*\textnormal{-iso}}(-, B)$ as well as the unitary constructed in Theorem \ref{Theorem: KSNGS commutes with ITP}.
	
	\begin{theorem}\label{Theorem: KSGNS Functor 2}\

		\noindent The assignment
		\begin{align*}
			\textnormal{KSGNS}:\textnormal{PosCor}_{*\textnormal{-iso}}&\to \textnormal{PosCor}_{*\textnormal{-iso}}\\
			(A, B, (E_B, \phi))&\mapsto (A, B, (F_{\phi, B}, \pi_\phi))\\
			\left((\chi, \rho, \eta):(A, B, (E_B, \phi))\to (C, D, (E_D, \psi)\right)&\mapsto (\chi, \rho, \tilde{\eta}\circ V^{-1})
		\end{align*}
		defines a covariant functor where $V$ is the unitary given in Theorem \ref{Theorem: KSNGS commutes with ITP}: $$V:A\otimes_{\phi\otimes_\rho I}(E_{B}\otimes_\rho D)\to F_{\phi, B}\otimes_{\rho}D\quad\quad V(a\dot{\otimes}(x\dot{\otimes}d))=(a\dot{\otimes}x)\dot{\otimes}d$$ 
	\end{theorem}
	\begin{proof}\
		
		\noindent The assignment being well-defined at the level of objects follows from the KSGNS construction. The assignment being well-defined at the level of morphisms follows from the construction of the isomorphism $(1_A, V)$ in Theorem \ref{Theorem: KSNGS commutes with ITP} as well as the construction of  $(\chi, \tilde{\eta})$ in Proposition \ref{Prop: Lifting 2}. Therefore, we just verify the assignment is functorial in a covariant manner.
		
		Let us first show the assignment preserves the identity morphism. Fix an object $(A, B, (E_B, \phi))$ in $\textnormal{PosCor}_{*\textnormal{-iso}}$. The assignment maps the identity morphism $(1_A, \textnormal{inc}, \iota)$ for $(A, B, (E_B, \phi))$  to $(1_A, \textnormal{inc}, \tilde{\iota}\circ V^{-1})$. Let $(1_A, \textnormal{inc}, \iota')$ be the identity morphism for $(A, B, (F_{\phi, B}, \pi_\phi))$. We claim 
		$$(1_A, \textnormal{inc}, \tilde{\iota}\circ V^{-1})=(1_A, \textnormal{inc}, \iota').$$
		To establish this claim, we prove $\tilde{\iota}\circ V^{-1}=\iota'$. As $V$ is a unitary, it suffices to show $\tilde{\iota}=\iota'\circ V$. By the standard reductions, it suffices to show the equality on elements of the form $a\dot{\otimes} (x\dot{\otimes} b)\in A\otimes_{\phi\otimes_{\textnormal{inc}} I}(E_B\otimes_\textnormal{inc} B)$. Given such an element, we have
		\begin{align*}
			\tilde{\iota}(a\dot{\otimes} (x\dot{\otimes} b))&=a\dot{\otimes} \iota(x\dot{\otimes} b)=a\dot{\otimes} xb=(a\dot{\otimes} x)b\\
			&=\iota'((a\dot{\otimes} x)\dot{\otimes} b)=(\iota'\circ V)(a\dot{\otimes} (x\dot{\otimes} b))
		\end{align*}
		showing $\iota'\circ V=\tilde{\iota}$. Therefore we conclude $(\textnormal{inc}, (\tilde{\iota}\circ V^{-1}, 1_A))=(\textnormal{inc}, (\iota', 1_A))$ showing the assignment on morphisms preserves the identity morphisms.
		
		Now we show the assignment on morphisms preserves composition. For $i=1, 2, 3$, let $(A_i, B_i, (E_{B_i}, \phi_i))$ be objects of $\textnormal{PosCor}_{*\textnormal{-iso}}$. For $i=1, 2$, let
		$$(\chi_i, \rho_i, \eta_i):(A_i, B_i, (E_{B_i}, \phi_i))\to (A_{i+1}, B_{i+1}, (E_{B_{i+1}}, \phi_{i+1}))$$
		be a morphism in $\textnormal{PosCor}_{*\textnormal{-iso}}$. Composing the two morphisms, we obtain a third morphism given by 
		$$(\chi_2\circ\chi_1, \tilde{\rho}_2\circ\rho_1, \eta_2\circ\hat{\eta}_2\circ U_1^{-1}):(A_1, B_1, (E_{B_1}, \phi_1))\to (A_3, B_3, (E_{B_3}, \phi_3))$$
		where
		$$U_1: (E_{B_1}\otimes_{\rho_1}B_2)\otimes_{\rho_2}B_3\to E_{B_1}\otimes_{\tilde{\rho}_2\circ\rho_1} B_3$$
		is the unitary from Proposition \ref{Prop: Composition and tensor}. Applying the assignment to the three morphisms, we obtain morphisms for $i=1, 2$
		$$(\chi_i, \rho_i, \tilde{\eta}_i\circ V_i^{-1}):(A_i, B_i, (F_{ \phi_i, B_i}, \pi_{\phi_i}))\to (A_{i+1}, B_{i+1}, (F_{ \phi_{i+1}, B_{i+1}}, \pi_{\phi_{i+1}}))$$
		where
		$$V_1: A_1\otimes_{\phi_1\otimes_{\rho_1}I}(E_{B_1}\otimes_{\rho_1}B_2)\to F_{\phi_1, B_1}\otimes_{\rho_1}B_2\quad\quad V_2:A_2\otimes_{\phi_2\otimes_{\rho_2}I}(E_{B_2}\otimes_{\rho_2}B_3)\to F_{\phi_2, B_2}\otimes_{\rho_2}B_3$$
		are the unitaries from Theorem \ref{Theorem: KSNGS commutes with ITP} as well as the morphism
		$$(\chi_2\circ\chi_1, \tilde{\rho}_2\circ\rho_1, \widetilde{\eta_2\circ\hat{\eta}_1\circ U_1^{-1}}\circ V_3^{-1}):(A_1, B_1, (F_{\phi_1, B_1}, \pi_{\phi_1}))\to (A_3, B_3, (F_{\phi_3, B_3}, \pi_{\phi_3}))$$
		where
		$$V_3: A_1\otimes_{\phi_1\otimes_{\tilde{\rho}_2\circ\rho_1}I}(E_{B_1}\otimes_{\tilde{\rho}_2\circ\rho_1}B_3)\to F_{\phi_1, B_1}\otimes_{\tilde{\rho}_2\circ\rho_1}B_3$$
		is the unitary from Theorem \ref{Theorem: KSNGS commutes with ITP}. Taking the composition of the morphisms $(\chi_1, \rho_1, \tilde{\eta}_1\circ V_1^{-1})$ and $(\chi_2, \rho_2, \tilde{\eta}_2\circ V_2^{-1})$ yields the morphism
		$$(\chi_2\circ\chi_1, \tilde{\rho}_2\circ\rho_1, \tilde{\eta}_2\circ V_2^{-1}\circ\widehat{\tilde{\eta}_1\circ V_1^{-1}}\circ U_2^{-1}):(A_1, B_1, (F_{ \phi_1, B_1}, \pi_{\phi_1}))\to (A_3, B_3, (F_{\phi_3, B_3}, \pi_{\phi_3}))$$
		where
		$$U_2:(F_{\phi_1, B_1}\otimes_{\rho_1}B_2)\otimes_{\rho_2}B_3\to F_{B_1, \phi_1}\otimes_{\tilde{\rho}_2\circ\rho_1}B_3$$
		is the unitary from Proposition \ref{Prop: Composition and tensor}.
		
		We claim
		$$(\chi_2\circ\chi_1, \tilde{\rho}_2\circ\rho_1, \tilde{\eta}_2\circ V_2^{-1}\circ\widehat{\tilde{\eta}_1\circ V_1^{-1}}\circ U_2^{-1})=(\chi_2\circ\chi_1, \tilde{\rho}_2\circ\rho_1, \widetilde{\eta_2\circ\hat{\eta}_1\circ U_1^{-1}}\circ V_3^{-1}).$$
		To prove this claim, we just need to show
		$$\tilde{\eta}_2\circ V_2^{-1}\circ\widehat{\tilde{\eta}_1\circ V_1^{-1}}\circ U_2^{-1}=\widetilde{\eta_2\circ\hat{\eta}_1\circ U_1^{-1}}\circ V_3^{-1}.$$
		Let $V$ be the unitary
		$$V:(A\otimes_{\phi_1\otimes_{\rho_1} I}(E_{B_1}\otimes_{\rho_1}B_2))\otimes_{\rho_2} B_3\to A\otimes_{(\phi_1\otimes_{\rho_1} I)\otimes_{\rho_2}I}((E_{B_1}\otimes_{\rho_1}B_2)\otimes_{\rho_2}B_3)$$
		given by Theorem \ref{Theorem: KSNGS commutes with ITP}. Consider the following diagram:
		\begin{center}\Small{
				\begin{tikzcd}
					& {F_{\phi_2, B_2}\otimes_{\rho_2} B_3} && {A_1\otimes_{\phi_2\otimes_{\rho_2}1_{B_3}}(E_{B_2}\otimes_{\rho_2}B_3)} & \\
					\\
					& {(A_1\otimes_{\phi_1\otimes_{\rho_1}I}(E_{B_1}\otimes_{\rho_1}B_2))\otimes_{\rho_2} B_3} && {A_1\otimes_{(\phi_1\otimes_{\rho_1}I)\otimes_{\rho_2}I}((E_{B_1}\otimes_{\rho_1}B_2)\otimes_{\rho_2}B_3)} \\
					\\
					& {F_{\phi_1, B_1}\otimes_{\tilde{\rho}_2\circ\rho_1}B_3} && {A_1\otimes_{\phi_1\otimes_{\tilde{\rho_2}\circ\rho_1}I}(E_{B_1}\otimes_{\tilde{\rho}_2\circ\rho_1}B_3)} \\
					\\
					& {(F_{\phi_1, B_1}\otimes_{\rho_1}B_2)\otimes_{\rho_2}B_3} && {F_{\phi_3, B_3}} &
					\arrow["{V_2^{-1}}"{description}, from=1-2, to=1-4]
					\arrow["{\tilde{\eta}_2}"{description}, bend right=-90, from=1-4, to=7-4]
					\arrow["{\hat{\tilde{\eta}}_1}"{description}, from=3-2, to=1-2]
					\arrow["V"{description, pos=0.6}, from=3-2, to=3-4]
					\arrow["{\widetilde{\widehat{\eta}_1}}"{description}, from=3-4, to=1-4]
					\arrow["{\widehat{V}_1\circ U_2^{-1}}"{description}, from=5-2, to=3-2]
					\arrow["{V_3^{-1}}"{description}, from=5-2, to=5-4]
					\arrow["{U_2^{-1}}"{description}, from=5-2, to=7-2]
					\arrow["{\widetilde{U_1^{-1}}}"{description}, from=5-4, to=3-4]
					\arrow["{\widetilde{\eta_2\circ\hat{\eta}_1\circ U_1^{-1}}}"{description}, from=5-4, to=7-4]
					\arrow["{\widehat{\tilde{\eta}_1\circ V_1^{-1}}}"{description}, bend left=90, from=7-2, to=1-2]
					\arrow["{\tilde{\eta}_2\circ V_2^{-1}\circ\widehat{\tilde{\eta}_1\circ V_1^{-1}}}"{description}, from=7-2, to=7-4]
			\end{tikzcd}}
		\end{center}
		We claim all the subdiagrams commute. Using functorality of the tensor product functor in Theorem \ref{Theorem: Tensor Functor}, the left side and right side diagrams commute.
		
		Let us show the top square commutes. As $V_2$ is a unitary, it suffices to show $\widehat{\widetilde{\eta}_1}=V_2\circ \widetilde{\widehat{\eta}}\circ V$. As both $V_2\circ \widetilde{\widehat{\eta}}\circ V$ and $\widehat{\widetilde{\eta}_1}$ are continuous as well as 
		$$\left(\frac{A_1\otimes_{\textnormal{alg}}(E_{B_1}\otimes_{B_1}B_2)}{N_{\phi_1\otimes_{\rho_1}I}}\right)\otimes_{B_2}B_3\subset A_1\otimes_{\phi_1\otimes_{\rho_1}I}(E_{B_1}\otimes_{\rho_1}B_2))\otimes_{\rho_2} B_3$$
		is dense, then it suffices to show equality on this dense subspace. Using linearity, it suffices to prove equality on elements of the form $(a\dot{\otimes} (x\dot{\otimes} b_1))\dot{\otimes} b_2\in\left(\frac{A_1\otimes_{\textnormal{alg}}(E_{B_1}\otimes_{B_1}B_2)}{N_{\phi_1\otimes_{\rho_1}I}}\right)\otimes_{B_2}B_3$. Given such an element, we have
		\begin{align*}
			\widehat{\widetilde{\eta}}_1((a\dot{\otimes} (x\dot{\otimes} b_1))\dot{\otimes} b_2)&=\widetilde{\eta}_1(a\dot{\otimes}(x\dot{\otimes}b_1))\dot{\otimes}b_2=(\chi_1(a)\dot{\otimes}\eta_1(x\dot{\otimes} b_1))\dot{\otimes}b_2\\
			&=V_2(\chi_1(a)\dot{\otimes}(\eta_1(x\dot{\otimes}b_1)\dot{\otimes}b_2))=V_2(\chi_1(a)\dot{\otimes}\hat{\eta}_1((x\dot{\otimes} b_1)\dot{\otimes}b_2))\\
			&=(V_2\circ \widetilde{\widehat{\eta}}_1)(a\dot{\otimes}((x\dot{\otimes} b_1)\dot{\otimes} b_2))=(V_2\circ \widetilde{\widehat{\eta}}_1\circ V)((a\dot{\otimes} (x\dot{\otimes} b_1))\dot{\otimes} b_2)
		\end{align*}
		Therefore we conclude  $\widehat{\widetilde{\eta}_1}=V_2\circ \widetilde{\widehat{\eta}}\circ V$.
		
		Now let us prove the middle square commutes. As all the maps are unitary, it suffices to show $U_2=V_3\circ\widetilde{U_1}\circ V\circ \widehat{V}_1$. By the standard reduction, it suffices to check equality on elements of the form $((a\dot{\otimes}x)\dot{\otimes}b_2)\dot{\otimes}b_3\in (F_{\phi_1, B_1}\otimes_{\rho_1}B_2)\otimes_{\rho_2}B_3$. Given such an element, we have
		\begin{align*}
			U_2(((a\dot{\otimes} x)\dot{\otimes}b_2)\dot{\otimes}b_3)&=(a\dot{\otimes}x)\dot{\otimes}\rho_2(b_2)b_3=V_3(a\dot{\otimes}(x\dot{\otimes}\rho_2(b_2)b_3))\\
			&=V_3(a\dot{\otimes}U_1((x\dot{\otimes}b_2)\dot{\otimes}b_3))=(V_3\circ \tilde{U}_1)(a\dot{\otimes}((x\dot{\otimes}b_2)\dot{\otimes}b_3))\\
			&=(V_3\circ \tilde{U}_1\circ V)((a\dot{\otimes}(x\dot{\otimes} b_2))\dot{\otimes}b_3)\\
			&=(V_3\circ \tilde{U}_1\circ V)(V_1((a\dot{\otimes}x)\dot{\otimes}b_2)\dot{\otimes}b_3)\\
			&=(V_3\circ \tilde{U}_1\circ V\circ \widehat{V}_1)(((a\dot{\otimes}x)\dot{\otimes}b_2)\dot{\otimes}b_3)
		\end{align*}
		Thus $U_2=V_3\circ\widetilde{U_1}\circ V\circ \widehat{V}_1$ showing $\widetilde{U_1^{-1}}\circ V_3^{-1}=V\circ \widehat{V}_1\circ U_2^{-1}$.
		
		Using commutativity of the other sub-diagrams, we obtain commutativity of the bottom square:
		\begin{align*}
			\tilde{\eta}_2\circ V_{2}^{-1}\circ \widehat{(\tilde{\eta}_1\circ V_1^{-1})}\circ U_2^{-1}&= \tilde{\eta}_2\circ V_{2}^{-1}\circ \widehat{\tilde{\eta}}_1\circ \widehat{V_1^{-1}}\circ U_2^{-1}\\
			&=\tilde{\eta}_2\circ \widetilde{\hat{\eta}}_1\circ V\circ \widehat{V_1^{-1}}\circ U_2^{-1}\\
			&=\tilde{\eta}_2\circ \widetilde{\hat{\eta}}_1\circ \widetilde{U_1^{-1}}\circ V_3^{-1}\\
			&=\widetilde{(\eta_2\circ\hat{\eta}_1\circ U_1^{-1})}\circ V^{-1}_3
		\end{align*}
		Thus we conclude the assignment on morphism preserves the composition rule. 
		
	\end{proof}
	
	\begin{theorem}\label{Theorem: Idempotency 2}\

		\noindent As functors on $\textnormal{PosCor}_{*\textnormal{-iso}}$, $\textnormal{KSGNS}^2$ is naturally isomorphic to $\textnormal{KSGNS}$. In particular, for each object $(A, B, (E_B, \phi))\in\textnormal{PosCor}_{*\textnormal{-iso}}$, the map
		$$(1_A, \textnormal{inc}, V_{\pi_\phi}\circ \iota):(A, B, (F_{\phi, B}, \pi_\phi))\to (A, B, (F_{\pi_\phi, B}, \pi_{\pi_\phi}))$$
		defines an isomorphism where
		$$\iota: F_{\phi, B}\otimes_{\textnormal{inc}}B\to F_{\phi, B}\quad\quad \iota((a\dot{\otimes}x)\dot{\otimes}b)=a\dot{\otimes}xb$$ is the unitary from Proposition \ref{Prop: inclusion tensor functor}.
		
	\end{theorem}
	
	\begin{proof}\
		
		\noindent Using Proposition \ref{Prop: Idempotency of KSGSN}, it follows the morphism  
		$$(1_A, \textnormal{inc}, V_{\pi_\phi}\circ \iota):(A, B, (F_{\phi, B}, \pi_\phi))\to (A, B, (F_{\pi_\phi, B}, \pi_{\pi_\phi}))$$
		is well defined and defines an isomorphism for an object $(A, B, (E_B, \phi))$ in $\textnormal{PosCor}_{*\textnormal{-iso}}$. Therefore, we verify the isomorphisms satisfy the naturality condition. 
		
		Fix objects $(A, B, (E_B, \phi))$ and $(C, D, (E_D, \psi))$ in $\textnormal{PosCor}_{*\textnormal{-iso}}$, and let
		$$(\chi, \rho, \eta):(A, B, (E_B, \phi))\to  (C, D, (E_D, \psi))$$
		be a morphism in $\textnormal{PosCor}_{*\textnormal{-iso}}$. Applying $\textnormal{KSGNS}$ and $\textnormal{KSGNS}^2$ to $(\chi, \rho, \eta)$ yields, respectively the morphisms
		$$(\chi, \rho, \tilde{\eta}\circ V_1^{-1}):(A, B, (F_{\phi, B}, \pi_\phi))\to (C, D, (F_{\psi, D}, \pi_{\psi}))$$
		and
		$$(\chi, \rho, \widetilde{\tilde{\eta}\circ V_1^{-1}}\circ V_2^{-1}):(A, B, (F_{\pi_\phi, B}, \pi_{\pi_\phi}))\to (C, D, (F_{\pi_\psi, D}, \pi_{\pi_{\psi}})).$$
		Let
		$$(1_A, \textnormal{inc}_B, V_{\pi_\phi}\circ \iota_B):(A, B, (F_{\phi, B}, \pi_\phi))\to (A, B, (F_{\pi_\phi, B}, \pi_{\pi_\phi}))$$
		and
		$$(1_C, \textnormal{inc}_D, V_{\pi_\phi}\circ \iota_D):(C, D, (F_{\phi, D}, \pi_\psi))\to (C, D, (F_{\pi_\phi, D}, \pi_{\pi_\psi}))$$
		be the isomorphisms in the assertion.
		
		To prove naturality, we verify the following diagram commutes in $\textnormal{PosCor}_{*\textnormal{-iso}}$:
		\begin{center}
			\begin{tikzcd}
				{(A, B, (F_{\phi, B}, \pi_\phi))} &&&& {(A, B, (F_{\pi_\phi, B}, \pi_{\pi_\phi}))} \\
				\\
				\\
				{(C, D, (F_{\psi, D}, \pi_{\psi}))} &&&& {(C, D, (F_{\pi_\psi, D}, \pi_{\pi_\psi}))}
				\arrow["{(1_A, \textnormal{inc}_B, V_{\pi_\phi}\circ \iota_B, 1_A)}"{description}, from=1-1, to=1-5]
				\arrow["{(\chi, \rho, \tilde{\eta}\circ V^{-1}_1)}"{description}, from=1-1, to=4-1]
				\arrow["{(\chi, \rho, (\widetilde{\tilde{\eta}\circ V^{-1}_1})\circ V_2^{-1})}"{description}, from=1-5, to=4-5]
				\arrow["{(1_C, \textnormal{inc}_D, V_{\pi_\psi}\circ \iota_D)}"{description}, from=4-1, to=4-5]
			\end{tikzcd}
		\end{center}
		From how composition is defined in $\textnormal{PosCor}_{*\textnormal{-iso}}$, we just need to show
		$$V_{\pi_\psi}\circ \iota_D\circ \widehat{\tilde{\eta}\circ V_1^{-1}}\circ U_1^{-1}=(\widetilde{\tilde{\eta}_1\circ V_1^{-1}})\circ V_2^{-1}\circ \widehat{(V_{\pi_\phi}\circ \iota_B)}\circ U_2^{-1}$$
		where $U_1$ and $U_2$ are the unitaries
		$$U_1:(F_{\phi, B}\otimes_{\rho}D)\otimes_{\textnormal{inc}_D}D\to F_{\phi, B}\otimes_{\rho}D \quad\quad U_2:(F_{\pi_\phi, B}\otimes_{\textnormal{inc}_B}B)\otimes_{\rho}D\to F_{\pi_\phi, D}\otimes_{\rho}D$$
		given in Proposition \ref{Prop: Composition and tensor}. From functoriality in Theorem \ref{Theorem: KSGNS Endofucntor} and Theorem \ref{Theorem: Tensor Functor}, we have
		$$\widehat{\tilde{\eta}\circ V_1^{-1}}=\widehat{\tilde{\eta}}\circ \widehat{V_1^{-1}}\quad\quad \widetilde{\tilde{\eta}\circ V_1^{-1}}=\widetilde{\tilde{\eta}}\circ \widetilde{V_1^{-1}}\quad\quad \widehat{V_{\pi_\phi}\circ \iota_B^{-1}}=\widehat{V_{\pi_\phi}}\circ \widehat{\iota_B}$$
		Thus, it suffices to show
		$$V_{\pi_\psi}\circ \iota_D\circ \widehat{\tilde{\eta}}\circ \widehat{V_1^{-1}}\circ U_1^{-1}= \widetilde{\tilde{\eta}}\circ \widetilde{V_1^{-1}}\circ V_2^{-1}\circ \widehat{V_{\pi_\phi}}\circ\widehat{\iota_B}\circ U_2^{-1}$$
		
		Let 
		$$\iota':(F_{\phi, B}\otimes_\rho D)\otimes_{\textnormal{inc}_D}D\to F_{\phi, B}\otimes_\rho D$$
		be the unitary from Proposition \ref{Prop: inclusion tensor functor}. Consider the following diagram:
		\begin{center}
			\begin{tikzcd}
				{F_{\phi, B}\otimes_{\widetilde{\textnormal{inc}}\circ\rho}D} && {F_{\phi, B}\otimes_{\tilde{\rho}\circ\textnormal{inc}_D}D} \\
				\\
				{(F_{\phi, B}\otimes_{\rho}D)\otimes_{\textnormal{inc}_D}D} && {(F_{\phi, B}\otimes_{\textnormal{inc}_B}B)\otimes_\rho D} \\
				\\
				{(A\otimes_{\phi\otimes_\rho I}(E_B\otimes_\rho D))\otimes_{\textnormal{inc}_D}D} && {F_{\phi, B}\otimes_{\rho}D} \\
				& {F_{\pi_\psi, D}} \\
				{F_{\psi, D}\otimes_{\textnormal{inc}_D}D} && {F_{\pi_\phi, B}\otimes_\rho D}
				\arrow["id"{description}, from=1-1, to=1-3]
				\arrow["{U^{-1}_1}"{description}, from=1-1, to=3-1]
				\arrow["{U_2^{-1}}"{description}, from=1-3, to=3-3]
				\arrow["{\widehat{V_1^{-1}}}"{description}, from=3-1, to=5-1]
				\arrow["{\iota'}"{description}, from=3-1, to=5-3]
				\arrow["{\widehat{\iota}_B}"{description}, from=3-3, to=5-3]
				\arrow["{\widehat{\widetilde{\eta}}}"{description}, from=5-1, to=7-1]
				\arrow["{\widehat{V_{\pi_\phi}}}"{description}, from=5-3, to=7-3]
				\arrow["{V_{\pi_\psi}\circ \iota_D}"{description}, from=7-1, to=6-2]
				\arrow["{\widetilde{\widetilde{\eta}}\circ \widetilde{V_1^{-1}}\circ V_2^{-1}}"{description}, from=7-3, to=6-2]
			\end{tikzcd}
		\end{center}
		Note, we obtain the identity map at the top for $\widetilde{\textnormal{inc}}\circ\rho=\rho=\tilde{\rho}\circ\textnormal{inc}$. We claim the sub-diagrams given by the top trapezoid and bottom hexagon commute. 
		
		Let us begin with the top trapezoid. As $U_1^{-1}$, $\widehat{\iota}_B$, and $U_2^{-1}$ are unitaries, then it suffices to show $U_1=U_2\circ \widehat{\iota}^{-1}_B\circ \iota'$. Let $(b_\lambda)_{\lambda\in\Lambda}$ be an approximate unit for $B$. By the standard reductions, it suffices to show the two maps agree on elements of the form $((a\dot{\otimes}x)\dot{\otimes}d_1)\dot{\otimes}d_2$ in $(F_{\phi, B}\otimes_{\rho}D)\otimes_{\textnormal{inc}_D}D$. Given such an element and using the remark after Proposition \ref{Prop: inclusion tensor functor}, we have
		\begin{align*}
			(U_2\circ \widehat{\iota_B^{-1}}\circ \iota')(((a\dot{\otimes}x)\dot{\otimes}d_1)\dot{\otimes}d_2)&=(U_2\circ \widehat{\iota_B^{-1}})((a\dot{\otimes}x)\dot{\otimes}d_1d_2)=\lim\limits_{\lambda\to\infty}U_2(((a\dot{\otimes}x)\dot{\otimes}b_\lambda)\dot{\otimes}d_1d_2)\\
			&=\lim\limits_{\lambda\to\infty}(a\dot{\otimes}x)\dot{\otimes}\rho(b_\lambda)d_1d_2=(a\dot{\otimes}x)\dot{\otimes}d_1d_2\\
			&=U_1(((a\dot{\otimes}x)\dot{\otimes}d_1)\dot{\otimes}d_2)
		\end{align*}
		Thus we conclude the top trapezoid commutes.
		
		Now we show the bottom hexagon commutes. Let $(a_\lambda)_{\lambda\in\Lambda}$ be an approximate unit in $A$, then $(\chi(a_\lambda))_{\lambda\in\Lambda}$ is also an approximate unit for $C$. Using the standard reduction, it suffices to show commutativity of the diagram on elements of the form $((a\dot{\otimes}x)\dot{\otimes}d_1)\dot{\otimes}d_2$ in $(F_{\phi, B}\otimes_\rho D)\otimes_{\textnormal{inc}_D}D$. Given such an element we have
		\begin{align*}
			(V_{\pi_\psi}\circ \iota_D\circ \widehat{\tilde{\eta}}\circ\widehat{V_1^{-1}})(((a\dot{\otimes}x)\dot{\otimes}d_1)\dot{\otimes}d_2)&=(V_{\pi_\psi}\circ \iota_D\circ \widehat{\tilde{\eta}})((a\dot{\otimes}(x\dot{\otimes}d_1))\dot{\otimes}d_2)\\
			&=(V_{\pi_\psi}\circ \iota_D)((\chi(a)\dot{\otimes}\eta(x\dot{\otimes}d_1))\dot{\otimes}d_2)\\
			&=V_{\pi_\psi}(\chi(a)\dot{\otimes}\eta(x\dot{\otimes}d_1)d_2)=V_{\pi_\psi}(\chi(a)\dot{\otimes}\eta(x\dot{\otimes}d_1d_2))\\
			&=\lim\limits_{\lambda\to\infty}\rho(a_\lambda)\dot{\otimes}(\rho(a)\dot{\otimes}\eta(x\dot{\otimes}d_1d_2))\\
			&=\lim\limits_{\lambda\to\infty}\tilde{\tilde{\eta}}(a_\lambda\dot{\otimes}(a\dot{\otimes}(x\dot{\otimes}d_1d_2)))\\
			&=\lim\limits_{\lambda\to\infty}(\tilde{\tilde{\eta}}\circ \widetilde{V_1^{-1}})(a_\lambda\dot{\otimes}((a\dot{\otimes}x)\dot{\otimes}d_1d_2))\\
			&=\lim\limits_{\lambda\to\infty}(\tilde{\tilde{\eta}}\circ \widetilde{V_1^{-1}}\circ V_2^{-1})((a_\lambda\dot{\otimes}(a\dot{\otimes}x))\dot{\otimes}d_1d_2)\\
			&=(\tilde{\tilde{\eta}}\circ \widetilde{V_1^{-1}}\circ V_2^{-1}\circ\widehat{V_{\pi_\phi}})((a\dot{\otimes}x)\dot{\otimes}d_1d_2)\\
			&=(\tilde{\tilde{\eta}}\circ \widetilde{V_1^{-1}}\circ V_2^{-1}\circ\widehat{V_{\pi_\phi}}\circ \iota')(((a\dot{\otimes}x)\dot{\otimes}d_1)\dot{\otimes}d_2)
		\end{align*}
		Thus we conclude the bottom hexagon in the diagram commutes. Observe, using commutativity of the two sub-diagrams, we have commutativity of the outer diagram:
		$$V_{\pi_\psi}\circ \iota_D\circ \widehat{\tilde{\eta}}\circ \widehat{V_1^{-1}}\circ U_1^{-1}= \widetilde{\tilde{\eta}}\circ \widetilde{V_1^{-1}}\circ V_2^{-1}\circ \widehat{V_{\pi_\phi}}\circ\widehat{\iota}_B\circ U_2^{-1}$$
		Therefore the isomorphisms satisfy the naturality condition proving $\textnormal{KSGNS}^2$ is naturally isomorphic to $\textnormal{KSGNS}$.
		
	\end{proof}
	
	\begin{corollary}\
		
		\noindent The $\textnormal{KSGNS}$ endofunctor on $\textnormal{PosCor}_{*\textnormal{-iso}}$ restricts to an idempotent endofunctor on the subcategory $\textnormal{PosCor}_{*\textnormal{-iso}, *\textnormal{-alg}}$. 
	\end{corollary}

	\section{Enrichment of $\textnormal{PosCor}(-, B)$ and $\textnormal{PosCor}_{*\textnormal{-alg}}$ over Topological Spaces}
	
	As the individual components in the data for morphisms in both $\textnormal{PosCor}(-, B)$ and $\textnormal{PosCor}$ naturally carry a topology, we look to induce a topology on the hom-sets for the categories. While this extrapolation is straightforward for the category $\textnormal{PosCor}(-, B)$, we run into domain issues for $\textnormal{PosCor}$. As we are able to resolve these issues using $\textnormal{PosCor}_{*\textnormal{-alg}}$, we opt to work with the subcategory rather than all of $\textnormal{PosCor}$. Once we have the topologies defined, we determine in what sense composition is continuous as well as in what sense the functors from the interior tensor product and KSGNS construction induce continuous maps on the hom-sets.
	
	\subsection{Enrichment of $\textnormal{PosCor}(-, B)$ and $\textnormal{PosCor}_{*\textnormal{-alg}}$}

	\begin{definition}\

		\noindent Let $B$ be a $C^*$-algebra. Fix objects $(A_1, B, (E_1, \phi_1)), (A_2, B, (E_2, \phi_2))\in \textnormal{PosCor}(-, B)$. For each $x\in E_1$ and $a\in A_1$, define the pseudo-metric
		\begin{align*}
			d_{x, a}:\textnormal{Hom}((A_1, B_1, (E_1, \phi_1)), (A_2, B, (E_2, \phi_2)))^2&\to [0, \infty)\\
			d_{x, a}((\rho, \eta), (\chi, \xi))=||\eta(x)-\xi(x)||&+||\rho(a)-\chi(a)||_{A_2}
		\end{align*}
		Endow $\textnormal{Hom}((A_1, B_1, (E_1, \phi_1)), (A_2, B, (E_2, \phi_2)))$ with the weakest topology for which each pseudo-metric $d_{x, a}$ is continuous.
		
	\end{definition}
	
	Let us make a few remarks on the topology. First, using properties of norms, it follows each $d_{x, a}$ is a pseudo-metric. Second, as the topology on each hom-set is generated by a collection of pseudo-metrics, then the topology on each hom-set is given a by subbasis generated by sets of the form
	$$B_\epsilon((\rho, \eta), d_{x, a})=\{(\chi, \xi)\in \textnormal{Hom}((A_1, B_1, (E_1, \phi_1)), (A_2, B, (E_2, \phi_2))): d_{x,a}((\rho, \eta), (\chi, \xi))<\epsilon\}$$
	for all $\epsilon>0$, $x\in E_1$, and $a\in A_1$. Third, as the topology on each hom-set is generated by the collection of pseudo-metrics, we have a nice characterization of net convergence; in particular, a net of morphism $((\rho_\lambda, \eta_\lambda))_{\lambda\in\Lambda}$ converges to $(\rho, \eta)$ if and only if for each $(x, a)\in E_1\times A_1$ the net converges to $(\rho, \eta)$ with respect to $d_{x, a}$. From the construction of $d_{x, a}$, we equivalently have $((\rho_\lambda, \eta_\lambda))_{\lambda\in\Lambda}$ converges to $(\rho, \eta)$ if and only if $\eta_\lambda\to \eta$ in the strong operator topology on $\mathcal{L}(E_1, E_2)$ and $\rho_\lambda\to \rho$ in the point-norm topology on $\textnormal{Hom}_{*\textnormal{-alg}}(A_1, A_2)$. 
	
	We would like to consider a similar topology on the hom-sets of $\textnormal{PosCor}$; however, we run into an issue for discussing strong convergence for the adjointable maps. To see this issue explicitly, let us fix objects $(A, B, (E_B, \phi))$ and $(C, D, (E_D, \psi))$ in $\textnormal{PosCor}$. If $(\chi, \rho, \eta)$ and $(\mu, \sigma, \xi)$ are morphisms from $(A, B, (E_B, \phi))$ to $(C, D, (E_D, \psi))$, then $\eta$ and $\xi$ may not be comparable in the strong operator topology for the domains of the function, respectively, are $E_B\otimes_\rho D$ and $E_B\otimes_\sigma D$. If $\rho$ and $\sigma$ are non-degenerate $*$-algebra homomorphisms with $\rho(B)\subset D$ and $\sigma(B)\subset D$, then, as the following lemma shows, we can pullback the domains for $\eta$ and $\xi$ to $E_B$ enabling us to compare the maps. As the morphisms in $\textnormal{PosCor}_{*\textnormal{-alg}}$ have this exact property, we will define topologies on the hom-sets of this subcategory rather than all of $\textnormal{PosCor}$.

	\begin{lemma}\label{Lemma:C linear maps}\
		
		\noindent Let $A$ and $B$ be $C^*$-algebras, and let $E$ be a Hilbert $A$-module. Let $\rho:A\to M(B)$ be a non-degenerate $*$-algebra homomorphism with $\rho(A)\subset B$. Let $(b_\lambda)_{\lambda\in\Lambda}$ be an approximate unit in $B$. For each $\lambda\in\Lambda$, the function
		$$V_\lambda:E\to E\otimes_\rho B\quad\quad V_\lambda(x)=x\dot{\otimes} b_\lambda $$
		is $\mathbb{C}$-linear contraction. Furthermore, the net  $(V_\lambda)_{\lambda\in\Lambda}$ converges in the strong operator topology on $\mathcal{B}(E, E\otimes_\rho B)$ to a $\mathbb{C}$-linear contraction $V_\rho$ which is independent of $(b_\lambda)_{\lambda\in\Lambda}$.
		
	\end{lemma}
	
	\begin{proof}\
		
		\noindent It is clear $V_\lambda$ is a $\mathbb{C}$-linear function. To see $V_\lambda$ is a contraction, let $x\in E$, then
		\begin{align*}
			||V_\lambda x||^2&=||x\dot{\otimes} b_\lambda ||^2=||b_\lambda\rho(\langle x, x\rangle)b_\lambda||\leq ||\rho||*||\langle x, x\rangle||=||\rho||*||x||^2\leq ||x||^2
		\end{align*}
		Thus $||V_\lambda||\leq 1$.
		
		We now verify $(V_\lambda)_{\lambda\in\Lambda}$ converges in the strong operator topology to a $\mathbb{C}$-linear contraction. Let $\lambda, \mu\in\Lambda$ and $x\in E$, then
		\begin{align*}
			||V_\lambda x-V_\mu x||^2&=||x\dot{\otimes}(b_\lambda-b_\mu)||^2=||(b_\lambda-b_\mu)\rho(\langle x, x\rangle)(b_\lambda-b_\mu)||\\
			&=||\sqrt{\rho(\langle x, x\rangle)}b_\lambda -\sqrt{\rho(\langle x, x\rangle}b_\mu||^2
		\end{align*}
		Since $(b_\lambda)_{\lambda\in\Lambda}$ is an approximate unit in $B$ and $\rho$ takes value in $B\subset M(B)$, then the equality above implies $(V_\lambda x)_{\lambda\in\Lambda}$ is a Cauchy net in $E\otimes_\rho C$. As $E\otimes_\rho C$ is complete, then $(V_\lambda x)_{\lambda\in\Lambda}$ converges to a point $V_\rho(x)$. Let $V_\rho:E\to E\otimes_\rho C$ be the induced function. It is clear $V_\rho$ is a $\mathbb{C}$-linear contraction. By construction $V_\rho$ is the strong limit of $(V_\lambda)_{\lambda\in\Lambda}$.
		
		Now we verify $V_\rho$ is independent of the approximate unit $(b_\lambda)_{\lambda\in\Lambda}$. To see this, suppose $(b_\mu')_{\mu\in M}$ is another approximate unit in $B$. Let $(V_\mu')_{\mu\in M}$ and $V'_\rho$ be the associated $\mathbb{C}$-linear contractions. Given $x\in E$, we have
		\begin{align*}
			||V_\rho x-V'_\rho x||&\leq ||V_\rho x-V_\lambda x||+||V_\lambda x-V'_\mu x||+||V'_\rho x-V_\mu'x||\\
			&\leq ||V_\rho x-V_\lambda x||+||V'_\rho x-V_\mu' x||+||\sqrt{\rho(\langle x, x\rangle} b_\lambda -\sqrt{\rho(\langle x, x\rangle)}||\\
			&+||\sqrt{\rho(\langle x, x\rangle)} b_\mu' -\sqrt{\rho(\langle x, x\rangle)}||
		\end{align*}
		Since each term on the second line can be made within $\frac{\epsilon}{4}$ for any $\epsilon>0$, then we conclude $||V_\rho x-V'_\rho x||=0$. Hence $V_\rho x=V'_\rho x$ showing $V_\rho=V'_\rho$. 
		
	\end{proof}
	
	\begin{definition}\

		\noindent Fix objects $(A, B, (E_B, \phi)), (C, D, (E_D, \psi))\in \textnormal{PosCor}_{*\textnormal{-alg}}$. For each $b\in B, x\in E_B$, and $a\in A$, define the pseudo-metrics
		\begin{align*}
			d_{b,x, a}:\textnormal{Hom}((A, B, (E_B, \phi)), (C, D, (E_D, \psi)))^2&\to [0, \infty)\\
			d_{b, x, a}((\chi, \rho, \eta), (\mu, \sigma, \xi))=&||\rho(b)-\sigma(b)||_{D}+||(\eta\circ V_\rho)(x)-(\xi\circ V_{\sigma})(x)||\\
			&+||\chi(a)-\mu(a)||_C
		\end{align*}
		Define the topology on $\textnormal{Hom}((A, B, (E_B, \phi)), (C, D, (E_D, \psi)))$ as the weakest topology such that each $d_{b, x, a}$ is continuous.
		
	\end{definition}
	
	Let us make a few remarks. First, using properties of norms, it immediately follows each $d_{b, x, a}$ is a pseudo-metric. Second, in an analogous manner to $\textnormal{PosCor}(-, B)$, the topology is generated by a subbasis of balls as well as that we are able to utilize the psuedo-metrics to characterize convergence of nets. In particular, from how the pseudo-metrics are defined, we have that a net of morphisms $((\chi_\lambda, \rho_\lambda, \eta_\lambda))_{\lambda\in\Lambda}$ converges to $(\chi, \rho, \eta)$ if and only if $\chi_\lambda\to \chi$ and $\rho_\lambda\to \rho$ converge in the point norm topology on $\textnormal{Hom}_{*\textnormal{-alg}}(A, C)$ and $\textnormal{Hom}_{*\textnormal{-alg}}(B, M(D))$, respectively, as well as $\eta_\lambda\circ V_{\rho_\lambda}\to \eta\circ V_\rho$ in the strong operator topology on $\mathcal{B}(E_B, E_D)$.
	
	\subsection{Continuity of composition in $\textnormal{PosCor}(-, B)$ and $\textnormal{PosCor}_{*\textnormal{-alg}}$}
	
	We now describe conditions for when composition in $\textnormal{PosCor}(-, B)$ and $\textnormal{PosCor}_{*\textnormal{-alg}}$ is continuous. We first begin with the category $\textnormal{PosCor}(-, B)$.
	
	\begin{proposition}\

		\noindent Let $B$ be a $C^*$-algebra. For $i=1, 2, 3$, let $O_i\in\textnormal{PosCor}(-, B)$. The function
		\begin{align*}
			\textnormal{Maps}(O_1, O_2)\times \textnormal{Maps}(O_2, O_3)&\to \textnormal{Maps}(O_1, O_3)\\
			((\rho, \eta), (\chi, \xi))&\mapsto (\chi\circ\rho, \xi'\circ\eta)
		\end{align*}
		is jointly continuous when the first factor is restricted to a subspace for which the adjointable maps are uniformly norm bounded.
		
	\end{proposition}
	
	\begin{proof}\
		
		\noindent For $i=1, 2, 3$, let $O_i=(A_i, B, (E_i, \phi_i))$. Let $X\subset \textnormal{Maps}(O_1, O_2)$ for which there exists $C_X>0$ such that for all $(\rho, \eta)\in X$, $||\eta||\leq C_X$. Consider the map in the assertion restrict to the subspace $X\times \textnormal{Maps}(O_2, O_3)$. 
		
		To see the restriction is continuous, it suffices to show convergent nets are mapped to convergent nets. Therefore, let $((\rho_\lambda, \eta_\lambda), (\chi_\lambda, \xi_\lambda))_{\lambda\in\Lambda}$ be a nets in the domain which converges to $((\rho, \eta), (\chi, \xi))$. As the domain is given the product topology, then $((\rho_\lambda, \eta_\lambda))_{\lambda\in\Lambda}$ converges to $(\rho, \eta)$ and $((\chi_\lambda, \xi_\lambda))_{\lambda\in\Lambda}$ converges to $(\chi, \xi)$. From the topology on the mapping spaces and that $X$ has the subspace topology, we know $\eta_\lambda\to \eta$ and  $\xi_\lambda\to \xi$ in with respect to the strong operator topology as well as $\rho_\lambda\to \rho$ and $ \chi_\lambda\to \chi$ with respect to the point-norm topology.
		
		Let $B_{C_X}(E_1, E_2)$ denote the ball of radius $C_X$ in $\mathcal{L}(E_1, E_2)$, then the map
		$$B_{C_X}(E_1, E_2)\times \mathcal{L}(E_2, E_3)\to \mathcal{L}(E_1, E_3)\quad\quad (T, S)\mapsto T\circ S$$
		is continuous with respect to the strong operator topology. Additionally composition of $*$-algebra homomorphisms is continuous with respect to the point norm topology. Therefore, we know
		$$\xi_\lambda\circ \eta_\lambda\to \xi\circ \eta\quad\quad \chi_\lambda\circ \rho_\lambda\to \chi\circ\rho$$
		which implies 
		$$(\chi_\lambda\circ\rho_\lambda, \xi_\lambda\circ \eta_\lambda)\to (\chi\circ\rho, \xi\circ \eta)$$
		Hence composition restricted to $X\times \textnormal{Maps}(O_2, O_3)$ is continuous as claimed.
		
	\end{proof}
	
	We now look to extend this result to the category $\textnormal{PosCor}_{*\textnormal{-alg}}$. As the topology on the hom-sets for $\textnormal{PosCor}_{*\textnormal{-alg}}$ requires us to pullback the adjointable maps using the $\mathbb{C}$-linear maps $V_\rho$, we first describe how the pullback interacts with composition. For $i=1, 2, 3,$ let us fix objects $(A_i, B_i, (E_{B_i}, \phi_i))\in \textnormal{PosCor}_{*\textnormal{-alg}}$, and let 
	$$(\chi_i, \rho_i, \eta_i):(A_i, B_i, (E_{B_i}, \phi_i))\to (A_{i+1}, B_{i+1}, (E_{B_{i+1}}, \phi_{i+1}))$$ 
	be morphisms for $i=1, 2$. Taking the composition of the morphisms yields
	$$(\chi_2\circ\chi_1, \rho_2\circ\rho_1, \eta_2\circ\hat{\eta}_1\circ U^{-1}):(A_1, B_1, (E_{B_1}, \phi_1))\to (A_3, B_3, (E_{B_3}, \phi_3))$$
	Let
	\begin{align*}
		V_{\rho_2\circ\rho_1}&:E_{B_1}\to E_{B_1}\otimes_{\rho_2\circ\rho_1} B_{3} & V_{\rho_1}&:E_{B_1}\to E_B\otimes_{\rho_1}B_2 \\
		V_{\rho_2}&:E_{B_2}\to E_{B_2}\otimes_{\rho_2} B_3 & V_{\rho_2}'&:E_{B_1}\otimes_{\rho_1}B_2\to (E_{B_1}\otimes_{\rho_1}B_2)\otimes_{\rho_2}B_3
	\end{align*}
	be the $\mathbb{C}$-linear contractions from Lemma \ref{Lemma:C linear maps}. Combining these four maps together with $\eta_2\circ\hat{\eta}_1\circ U^{-1}$, we obtain the following diagram:
	
	\begin{center}
		\begin{tikzcd}
			{E_{B_1}\otimes_{\rho_2\circ\rho_1}B_3} && {(E_{B_1}\otimes_{\rho_1}B_2)\otimes_{\rho_2}B_3} && {E_{B_2}\otimes_{\rho_2}B_3} && \\
			\\
			{E_{B_1}} && {E_{B_1}\otimes_{\rho_1}B_3} && {E_{B_2}} && {E_{B_3}}
			\arrow["{U^{-1}}"{description}, from=1-1, to=1-3]
			\arrow["{\hat{\eta}_1}"{description}, from=1-3, to=1-5]
			\arrow["{\eta_2}"{description}, from=1-5, to=3-7]
			\arrow["{V_{\rho_2\circ\rho_1}}"{description}, from=3-1, to=1-1]
			\arrow["{V_{\rho_1}}"{description}, from=3-1, to=3-3]
			\arrow["{V_{\rho_2}'}"{description}, from=3-3, to=1-3]
			\arrow["{\eta_1}"{description}, from=3-3, to=3-5]
			\arrow["{V_{\rho_2}}"{description}, from=3-5, to=1-5]
			\arrow["{\eta_2\circ V_{\rho_2}}"{description}, from=3-5, to=3-7]
		\end{tikzcd}
	\end{center}
	
	As $\eta_2\circ\hat{\eta}_1\circ U^{-1}\circ V_{\rho_2\circ\rho_1}$ is the pullback of $\eta_2\circ\hat{\eta}_1\circ U^{-1}$ to $E_{B_1}$, then, provided the outside diagram commutes, we are able to realize this as the composition of the pullback versions of $\eta_1$ and $\eta_2$. As we will have assumptions on $\eta_1\circ V_{\rho_1}$ and $\eta_2\circ V_{\rho_2}$ when proving continuity of composition, we look to prove commutativity of the outside diagram. As the far right triangle definitionaly commutes, we only need to verify the two squares commute. We verify this commutativity in the next two lemmas.
	
	\begin{lemma}\label{Lemma: Diagram1}\
		
		\noindent Let $A, B,$ and $C$ be a $C^*$-algebras, and let $E$ be a Hilbert $A$-module. Let $\rho:A\to M(B)$ and $\chi:B\to M(C)$ be a non-degenerate $*$-algebra homomorphisms for which $\rho(A)\subset B$ and $\chi(B)\subset C$. Then the following diagram commutes:
		\begin{center}
			\begin{tikzcd}
				{E\otimes_{\chi\circ\rho}C} && {(E\otimes_\rho B)\otimes_\chi C} \\
				\\
				E && {E\otimes_\rho B}
				\arrow["U"{description}, from=1-3, to=1-1]
				\arrow["{V_{\chi\circ\rho}}"{description}, from=3-1, to=1-1]
				\arrow["{V_\rho}"{description}, from=3-1, to=3-3]
				\arrow["{V_\chi}"{description}, from=3-3, to=1-3]
			\end{tikzcd}
		\end{center}
		where 
		$$U:(E\otimes_\rho B)\otimes_\chi C\to E\otimes_{\chi\circ\rho}C\quad\quad U((x\dot{\otimes}b)\dot{\otimes}c)=x\dot{\otimes}\chi(b)c$$
		is the $C$-linear unitary given in Proposition \ref{Prop: Composition and tensor}.
		
	\end{lemma}
	
	\begin{proof}\
		
		\noindent To prove commutativity of the diagram, let $(b_\mu)_{\mu\in M}$ and $(c_\gamma)_{\gamma\in G}$ be approximate units for $B$ and $C$, respectively. As $\chi$ is a non-degenerate $*$-algebra homomorphism, we know for all $\gamma\in G$, $\chi(b_\mu)c_\gamma\to c_\gamma$ in $C$ as $\mu\to \infty$. Using this as well as continuity of $U$, we have for each $x\in E$
		\begin{align*}
			V_{\chi\circ\rho}(x)&=\lim\limits_{\gamma\to \infty} x\dot{\otimes} c_{\gamma}=\lim\limits_{\gamma\to \infty}\lim\limits_{\mu\to \infty}x\dot{\otimes} \chi(b_\mu)c_\gamma=\lim\limits_{\gamma\to \infty}\lim\limits_{\mu\to\infty}U((x\dot{\otimes} b_\mu)\dot{\otimes} c_\gamma)\\
			&=\lim\limits_{\gamma\to \infty}U(V_{\rho}(x)\dot{\otimes} c_\gamma)=(U\circ V_{\chi}\circ V_\rho)(x)
		\end{align*}
		Thus $V_{\chi\circ\rho}=U\circ V_{\chi}\circ V_\rho$. 
		
	\end{proof}
	
	\begin{lemma}\label{Lemma: Diagram2}\
		
		\noindent Let $A, B$, and $C$ be a $C^*$-algebra, let $E_A$ be a Hilbert $A$-module, and let $E_B$ be a Hilbert $B$-module. Let $\rho:A\to M(B)$ and $\chi:B\to M(C)$ be non-degenerate $*$-algebra homomorphisms with $\rho(A)\subset B$ and $\chi(B)\subset C$. For each  $\eta\in\mathcal{L}(E_A\otimes_\rho B, E_B)$, the following diagram commutes:
		\begin{center}
			\begin{tikzcd}
				{(E_A\otimes_\rho B)\otimes_\chi C} &&&& {E_B\otimes_\chi C} \\
				\\
				&& {E_A\otimes_\rho B} \\
				\\
				{E_A} &&&& {E_B}
				\arrow["{\hat{\eta}}"{description}, from=1-1, to=1-5]
				\arrow["{V'_\chi}"{description}, from=3-3, to=1-1]
				\arrow["{\hat{\eta}\circ V'_\chi}"{description}, from=3-3, to=1-5]
				\arrow["\eta"{description}, from=3-3, to=5-5]
				\arrow["{V_{\chi}'\circ V_\rho}"{description}, from=5-1, to=1-1]
				\arrow["{V_\rho}"{description}, from=5-1, to=3-3]
				\arrow["{\eta\circ V_\rho}"{description}, from=5-1, to=5-5]
				\arrow["{V_\chi}"{description}, from=5-5, to=1-5]
			\end{tikzcd}
		\end{center}
		
	\end{lemma}
	
	\begin{proof}\
		
		\noindent Clearly the top triangle, bottom triangle, and left triangle in the square commute. Let us show the right triangle commutes. Let $(c_\gamma)_{\mu\in M}$ be an approximate unit for $C$. As $E_A\otimes_A B$ is a dense subspace of $E_A\otimes_\rho B$ and both $\hat{\eta}\circ V_\chi'$ and $V_\chi\circ \eta$ are continuous, then it suffices to show commutativity on the dense subspace $E_A\otimes_AB$. Using linearity, it suffices to prove commutativity on elements of the form $x\dot{\otimes} b\in E_A\otimes_\rho B$. Given such an element, we have
		\begin{align*}
			(\hat{\eta}\circ V_\chi')(x\dot{\otimes} b)&=\lim\limits_{\mu\to \infty}\hat{\eta}((x\dot{\otimes} b)\dot{\otimes} c_\mu)=\lim\limits_{\mu\to \infty}\eta(x\dot{\otimes} b)\dot{\otimes} c_\mu=V_\chi(\eta(x\dot{\otimes} b))
		\end{align*}
		Hence $\hat{\eta}\circ V'_\chi=V_\chi\circ\eta$. Using commutativity of the four inner triangles, we have
		\begin{align*}
			\hat{\eta}\circ V_\chi'\circ V_\rho=V_\chi\circ \eta\circ V_\rho
		\end{align*}
		showing the outer square commutes.
		
	\end{proof}

	\begin{proposition}\

		\noindent For $i=1, 2, 3$, let $O_i\in\textnormal{PosCor}_{*\textnormal{-alg}}$. Then the function
		\begin{align*}
			\textnormal{Maps}(O_1, O_2)\times \textnormal{Maps}(O_2, O_3)&\to \textnormal{Maps}(O_1, O_3)\\
			((\chi, \rho, \eta)), (\mu, \sigma, \xi))&\mapsto (\mu\circ\chi, \tilde{\sigma}\circ\rho, \xi\circ\hat{\eta}\circ U^{-1})
		\end{align*}
		is jointly continuous when the first factor is restricted to a subspace for which the adjointable maps are uniformly norm bounded.
		
	\end{proposition}
	
	\begin{proof}\
		
		\noindent For $i=1, 2, 3$, let $O_i=(A_i, B_i, (E_{B_i}, \phi_i))$. Let $X\subset \textnormal{Maps}(O_1, O_2)$ for which there exists $C_X>0$ such that for all $(\chi, \rho, \eta)\in X$, $||\eta||\leq C_X$. Consider the composition map restricted to $X\times \textnormal{Maps}(O_2, O_3)$. 
		
		To prove the map is continuous, it suffices to show convergent net map to convergent nets. Let $(((\chi_\lambda, \rho_\lambda, \eta_\lambda), (\mu_\lambda, \sigma_\lambda, \xi_\lambda)))_{\lambda\in\Lambda}$ be a net in $X\times \textnormal{Maps}(O_2, O_3)$ which converges to $((\chi, \rho, \eta), (\mu, \sigma, \xi))$. As the domain is equipped with the product topology, we know $(\chi_\lambda, \rho_\lambda, \eta_\lambda)\to(\chi, \rho, \eta)$ and $(\mu_\lambda, \sigma_\lambda, \xi_\lambda)\to (\mu, \sigma, \xi)$. From convergence in the mapping spaces, we have $\chi_\lambda\to \chi$ with respect to the point-norm topology on $\textnormal{Hom}_{*\textnormal{-alg}}(A_1, A_2)$, $\rho_\lambda\to \rho$ with respect to the point norm topology on $\textnormal{Hom}_{*\textnormal{-alg}}(B_1, M(B_2))$, $\mu_\lambda\to \mu$ with respect to the point norm topology on $\textnormal{Hom}_{*\textnormal{-alg}}(A_2, A_3)$, $\sigma_\lambda\to \sigma$ with respect to the point-norm topology on $\textnormal{Hom}_{*\textnormal{-alg}}(B_2, M(B_3))$, $\eta_\lambda\circ V_{\rho_\lambda}\to \eta\circ V_\rho$ with respect to the strong operator topology on $\mathcal{B}(E_{B_1}, E_{B_2})$, and $\xi_\lambda\circ V_{\sigma_\lambda}\to \xi\circ V_{\sigma}$ with respect to the strong operator topology on $\mathcal{B}(E_{B_2}, E_{B_3})$.
		
		We observer the maps
		$$\textnormal{Hom}_{*\textnormal{-alg}}(A_1, A_2)\times\textnormal{Hom}_{*\textnormal{-alg}}(A_2, A_3)\to \textnormal{Hom}_{*\textnormal{-alg}}(A_1, A_3)\quad\quad (\chi, \mu)\mapsto \chi\circ \mu$$ 
		and
		$$\textnormal{Hom}_{*\textnormal{-alg}}(B_1, M(B_2))\times\textnormal{Hom}_{*\textnormal{-alg}}(B_2, M(B_3))\to \textnormal{Hom}_{*\textnormal{-alg}}(B_1, M(B_3))\quad\quad (\rho, \sigma)\mapsto \tilde{\sigma}\circ \rho$$
		are continuous with respect to the point norm topologies on each hom-set. Therefore we have $\mu_\lambda\circ\chi_\lambda\to \mu\circ\chi$ with respect to the point norm topology on $\textnormal{Hom}_{*\textnormal{-alg}}(A_1, A_3)$ and $\tilde{\sigma}_\lambda\circ\rho_\lambda\to \tilde{\sigma}\circ \rho$ with respect to the point-norm topology on $\textnormal{Hom}_{*\textnormal{-alg}}(B_1, M(B_3))$.
		
		Using Lemma \ref{Lemma: Diagram1} and \ref{Lemma: Diagram2}, we have for all $\lambda\in\Lambda$
		$$\xi_\lambda\circ \hat{\eta}_\lambda\circ U^{-1}_\lambda\circ V_{\tilde{\sigma}_\lambda\circ\rho_\lambda}=\xi_\lambda\circ V_{\sigma_\lambda}\circ \eta_\lambda\circ V_{\rho_\lambda}$$
		as well as
		$$\xi\circ \hat{\eta}\circ U^{-1}\circ V_{\tilde{\sigma}\circ\rho}=\xi\circ V_{\sigma}\circ \eta\circ V_{\rho}.$$
		Let $B_{C_X}(E_{B_1}, E_{B_2})$ denote the ball of radius $C_X$ in $\mathcal{B}(E_{B_1}, E_{B_2})$. As 
		$$B_{C_X}(E_{B_1}, E_{B_2})\times \mathcal{B}(E_{B_2}, E_{B_3})\to \mathcal{B}(E_B, E_D)\quad\quad (S, T)\mapsto T\circ S$$ is continuous with respect to strong operator topology, then we know
		$$\chi_\lambda\circ V_{\chi_\lambda}\circ \eta_\lambda\circ V_{\rho_\lambda}\to \chi\circ V_{\chi}\circ \eta\circ V_{\rho}$$
		with respect to the strong operator topology on $\mathcal{B}(E_{B_1}, E_{B_3})$.
		Therefore we conclude $$\xi_\lambda\circ\hat{\eta}_\lambda\circ U^{-1}_\lambda\circ V_{\chi_\lambda\circ \rho_\lambda}\to \xi\circ\hat{\eta}\circ U^{-1}\circ V_{\chi\circ \rho}$$
		with respect to the strong operator topology on $\mathcal{B}(E_{B_1}, E_{B_3})$.
		
		Having shown
		$$\mu_\lambda\circ\chi_\lambda\to \mu\circ\chi\quad\quad \tilde{\sigma}_\lambda\circ\rho_\lambda\to \tilde{\sigma}\circ\rho\quad\quad\xi_\lambda\circ\hat{\eta}_\lambda\circ U^{-1}_\lambda\circ V_{\tilde{\sigma}_\lambda\circ \rho_\lambda}\to \xi\circ\hat{\eta}\circ U^{-1}\circ V_{\tilde{\sigma}\circ \rho}$$
		we conclude 
		$$(\mu_\lambda\circ \chi_\lambda, \tilde{\sigma}_\lambda\circ\rho_\lambda, \xi_\lambda\circ\hat{\eta}_\lambda U^{-1}_\lambda)\to (\mu\circ \chi, \tilde{\sigma}\circ\rho, \xi\circ\hat{\eta} U^{-1})$$ showing the composition map restricted to $X\times \textnormal{Maps}(O_2, O_3)$ is continuous.
		
	\end{proof}
	
	\subsection{Continuity of the KSNGS endofunctors and functors from the interior tensor product}
	
	We conclude the section by describing when the functors from the KSGNS construction and interior product define continuous map on the hom-sets of $\textnormal{PosCor}(-, B)$ and $\textnormal{PosCor}_{*\textnormal{-alg}}$. We begin first with the category $\textnormal{PosCor}(-, B)$.
	
	\begin{proposition}\
		
		\noindent Let $B$ and $C$ be $C^*$-algebras.
		\begin{enumerate}
			\item For each $O_1, O_2\in \text{PosCor}_{*\text{-alg}}(-, B)$, the map
            $$\textnormal{Maps}(O_1, O_2)\xrightarrow{(\rho, \eta)\mapsto (\rho, \tilde{\eta})} \textnormal{Maps}(\textnormal{KSGNS}(O_1), \textnormal{KSGNS}(O_2))$$
            is continuous when the domain is restricted to a subspace for which the adjointable maps are uniformly norm bounded.
			\item Let $(F, \pi)$ be a $C^*$-correspondence from $B$ to $C$.  For each $O_1, O_2\in \text{PosCor}_{*\text{-alg}}(-, B)$, the map
            $$\textnormal{Maps}(O_1, O_2)\xrightarrow{(\rho, \eta)\mapsto (\rho, \tilde{\eta})} \textnormal{Maps}(((-)\otimes_\pi F)(O_1), ((-)\otimes_\pi F)(O_2))$$
            is continuous when the domain is restricted to a subspace for which the adjointable maps are uniformly norm bounded.
		\end{enumerate}
		
	\end{proposition}
	
	\begin{proof}\
		
		\noindent For $i=1,2$, let $O_i=(A_i, B, (E_i, \phi_i))\in \textnormal{PosCor}_{*\textnormal{-iso}}(-, B)$. Let $X\subset \textnormal{Maps}(O_1, O_2)$ with the property that there exists $C_X>0$ such that for all $(\rho, \eta)\in X$, $||\eta||\leq C_X$. 
        
        We start by proving $(1)$. Consider the induced map from the KSGNS endofunctor restricted to $X$. To show the restricted map is continuous it suffices to show convergent nets map to convergent nets. Let $((\rho_\lambda, \eta_\lambda))_{\lambda\in\Lambda}$ be a net of morphisms in $X$ which converge to $(\rho, \eta)$, then $\rho_\lambda\to \rho$ with respect to the point-norm topology and $\eta_\lambda\to \eta$ with respect to the strong operator topology. 
        
        We claim $(\rho_\lambda, \tilde{\eta}_\lambda)\to (\rho, \tilde{\eta})$. As $\rho_\lambda\to \rho$, we just need to show $\tilde{\eta}_\lambda\to \tilde{\eta}$ with respect to the strong operator topology on $\mathcal{B}(F_{\phi_1}, F_{\phi_2})$.By construction of $\tilde{\eta}_\lambda$ in Proposition \ref{Prop: Lifting 2}, we know $||\tilde{\eta}_\lambda||\leq ||\eta_\lambda||^2$. Therefore $(\tilde{\eta}_\lambda)_{\lambda\in\Lambda}$ is uniformly bounded by $C_X$. As $\frac{A_1\otimes_{\textnormal{alg}}E_1}{N_{\phi_1}}$ is dense in $F_{\phi_1}$ and $(\tilde{\eta}_\lambda)_{\lambda\in\Lambda}$ is uniformly bounded, then the standard approximation argument can be utilized to conclude strong convergence of $(\tilde{\eta}_\lambda)_{\lambda\in\Lambda}$ to $\tilde{\eta}$ on $F_{\phi_1}$ from such convergence on the dense subspace. Using linearity, it suffices to check strong convergence on elements of the form $a\dot{\otimes}x$ in $\frac{A_1\otimes_{\textnormal{alg}}E_1}{N_{\phi_1}}$. For each $\lambda\in\Lambda$, we have
		\begin{align*}
			||\tilde{\eta}_\lambda(a\dot{\otimes}x)-\tilde{\eta}(a\dot{\otimes}x)||&=||\rho_\lambda(a)\dot{\otimes}\eta_\lambda(x)-\rho(a)\dot{\otimes}\eta(x)||\\
			&\leq ||(\rho_\lambda(a)-\rho(a))\dot{\otimes}\eta_\lambda (x)||+||\rho(a)\dot{\otimes}(\eta_\lambda(x)-\eta(x))||\\
			&\leq C_X||x||*||\phi_2||*||\rho_\lambda(a)-\rho(a)||+||\phi_2||*||a||*||\eta_\lambda(x)-\eta(x)||^2
		\end{align*}
		Since both terms on third line can be made within $\frac{\epsilon}{2}$ for large enough $\lambda$, then we conclude $\tilde{\eta}_\lambda(a\dot{\otimes}x)\to \tilde{\eta}(a\dot{\otimes}x)$. Hence $\tilde{\eta}_\lambda\to \tilde{\eta}$ with respect to the strong operator topology. Therefore we conclude the restricted map to $X$ is continuous.
		
		Now we prove $(2)$. Consider the induced map from $(-)\otimes_\pi F$ restricted to $X$. To show the restricted map is continuous, it suffices to show convergent nets map to convergent nets. Let $((\rho_\lambda, \eta_\lambda))_{\lambda\in\Lambda}$ be a net in $X$ which converge to $(\rho, \eta)$, then $\rho_\lambda\to \rho$ with respect to the point-norm and $\eta_\lambda\to \eta$ with respect to the strong operator topology.
		
		We claim $(\rho_\lambda, \hat{\eta}_\lambda)\to (\rho, \hat{\eta}_\lambda)$. As $\rho_\lambda\to \rho$, we just need to show $\hat{\eta}_\lambda\to \hat{\eta}$. By construction of $\hat{\eta}_\lambda$ in Proposition \ref{Prop: Lift 1}, we know $||\hat{\eta}_\lambda||\leq ||\eta_\lambda||$. Therefore $(\hat{\eta}_\lambda)_{\lambda\in\Lambda}$ is uniformly bounded by $C_X$. As $E_1\otimes_B F$ is dense in $E\otimes_\pi F$ and $(\hat{\eta}_\lambda)_{\lambda\in\Lambda}$ is uniformly bounded, then the standard approximation argument can be utilized to conclude strong convergence of $(\tilde{\eta}_\lambda)_{\lambda\in\Lambda}$ to $\tilde{\eta}$ on $E\otimes_\pi F$ from such convergence on the dense subspace. Using linearity, it suffices to check strong convergence on simple tensors $x\dot{\otimes}f\in E\otimes_\pi F$. For each $\lambda\in\Lambda$, we have
		\begin{align*}
			||\hat{\eta}_\lambda(x\dot{\otimes}f)-\hat{\eta}(x\dot{\otimes}f)||&=||(\eta_\lambda(x)-\eta(x))\dot{\otimes}f||\leq ||f||^2*||\eta_\lambda(x)-\eta(x)||
		\end{align*}
		Since the term on the end can be made within $\epsilon$ for large enough $\lambda$, then we conclude $\hat{\eta}_\lambda(x\dot{\otimes}f)\to \hat{\eta}(x\dot{\otimes}f)$. Hence $\hat{\eta}_\lambda\to \hat{\eta}$ with respect to the strong operator topology. Therefore we conclude the map restricted to $X$ is continuous.
		
	\end{proof}

	\begin{proposition}\label{Prop: Continuity of KSGNS}\

		\noindent Let $O_1, O_2\in \textnormal{PosCor}_{*\textnormal{-iso}, *\textnormal{-alg}}$.Then
		$$\textnormal{Maps}(O_1, O_2)\to \textnormal{Maps}(\textnormal{KSGNS}(O_1), \textnormal{KSGNS}(O_2))\quad (\chi, \rho, \eta)\mapsto (\chi,, \rho, \tilde{\eta}\circ V^{-1})$$
		is continuous when restricted in the domain to a subspace for which the adjointable maps are uniformly norm bounded.
		
	\end{proposition}
	
	\begin{proof}\
		
		\noindent For $i=1, 2$, let $O_i=(A_i, B_i, (E_{B_i}, \phi_i))$. Let $X\subset \textnormal{Maps}(O_1, O_2)$ such that for all morphisms $(\chi, \rho, \eta)\in X$, $||\eta||\leq C$ which there exists. To prove the map in the assertion is continuous when the domain is restricted to $X$, it suffices to show convergent nets map to convergent nets. Let $((\chi_\lambda, \rho_\lambda, \eta_\lambda))_{\lambda\in\Lambda}$ be a net of morphisms from $O_1$ to $O_2$ which converges to $(\chi, \rho, \eta)$, then $\chi_\lambda\to \chi$ with respect to the point-norm topology on $\textnormal{Hom}_{*\textnormal{-alg}}(A_1, A_2)$, $\rho_\lambda\to \rho$ with respect to the point-norm topology on $\textnormal{Hom}_{*\textnormal{-alg}}(B_1, M(B_2))$, and $\eta_\lambda\circ V_{\rho_\lambda}\to \eta\circ V_\rho$ with respect to the strong operator topology on $\mathcal{B}(E_{B_1}, E_{B_2})$.
		
		We claim $(\chi_\lambda, \rho_\lambda, \tilde{\eta}_\lambda\circ V^{-1}_\lambda)\to (\chi, \rho, \tilde{\eta}\circ V^{-1})$. As $\chi_\lambda\to \chi$ and $\rho_\lambda\to \rho$, we just need to verify $\tilde{\eta}_\lambda\circ V^{-1}_\lambda\circ V_{\rho_\lambda}'\to \tilde{\eta}\circ V^{-1}\circ V_\rho'$. Since $\frac{A_1\otimes E_{B_1}}{N_{\phi_1}}$ is dense in $F_{\phi_1, B_1}$ and $(\tilde{\eta}_\lambda\circ V_\lambda^{-1}\circ V_{\rho_\lambda}')_{\lambda\in\Lambda}$ is uniformly bounded by $C$, then, utilizing the standard approximation argument, it suffices to show  $\tilde{\eta}_\lambda\circ V_\lambda^{-1}\circ V_{\rho_\lambda}'\to \tilde{\eta}\circ V^{-1}\circ V_\rho'$ on $\frac{A_1\otimes_{\textnormal{alg}} E_{B_1}}{N_{\phi_1}}$. Using linearity, it suffices to show strong convergence on elements of the form $a\dot{\otimes} x\in \frac{A_1\otimes_{\textnormal{alg}} E_{B_1}}{N_{\phi_1}}$. Fix an approximate unit $(b_\mu)_{\mu\in M}$ for $B_2$. Using continuity of $\widetilde{\eta}_\lambda$ and $V_\lambda^{-1}$, we have
		\begin{align*}
			(\tilde{\eta}_\lambda\circ V_\lambda^{-1}\circ V_{\rho_\lambda}')(a\dot{\otimes} x)&=\lim\limits_{\mu\to \infty}(\tilde{\eta}_\lambda\circ V_{\lambda}^{-1})((a\dot{\otimes} x)\dot{\otimes} b_\mu ) =\lim\limits_{\mu\to \infty}\tilde{\eta}_\lambda(a\dot{\otimes}(x\dot{\otimes} b_\mu))\\
			&=\tilde{\eta}_\lambda(a\dot{\otimes} V_{\rho_\lambda}(x))=\chi_\lambda(a)\dot{\otimes} (\eta_\lambda\circ V_{\rho_\lambda})(x)
		\end{align*}
		Similarly, $(\tilde{\eta}\circ V^{-1}\circ V_{\rho}')(a\dot{\otimes} x)=\chi(a)\dot{\otimes} (\eta\circ V_{\rho})(x)$. Using the Cauchy-Schwarz Inequality, $V_{\rho_\lambda}$ is a contraction, and $\eta_\lambda$ are uniformly bounded by $C$, we have
		\begin{align*}
			||\chi_\lambda(a)\dot{\otimes} &(\eta_\lambda\circ V_{\rho_\lambda})(x)-\chi(a)\dot{\otimes} (\eta\circ V_{\rho})(x)||\\
			&\leq ||(\chi_\lambda(a)-\chi(a))\dot{\otimes}(\eta_\lambda\circ V_{\rho_\lambda})(x)||+||\chi(a)\dot{\otimes}((\eta_\lambda\circ V_{\rho_\lambda})(x)-(\eta\circ V_{\rho})(x))||\\
			&\leq C^{1.5}||\phi_2||*||x||^{1.5}*||\chi_\lambda(a)-\chi(a)||^2+||\phi_2||*||a||^2*||(\eta_\lambda\circ V_{\rho_\lambda})(x)-(\eta\circ V_{\rho})(x)||^{1.5} 
		\end{align*}
		As $\eta_\lambda\circ V_{\rho_\lambda}\to \eta\circ V_\rho$ and $\chi_\lambda\to \chi$, the inequality implies $\tilde{\eta}_\lambda\circ V_\lambda^{-1}\circ V_{\rho_\lambda}'\to \tilde{\eta}\circ V^{-1}\circ V_\rho'$ on elements of the form $a\dot{\otimes} x\in \frac{A_1\otimes_{\textnormal{alg}}E_{B_1}}{N_{\phi_1}}$. Hence we conclude $\tilde{\eta}_\lambda\circ V_\lambda^{-1}\circ V_{\rho_\lambda}'\to \tilde{\eta}\circ V^{-1}\circ V_\rho'$ on all of $F_{\phi_1, {B_1}}$. 
		
	\end{proof}
	
	\section{KSGNS Construction and Equivariant Correspondences}
	
	We conclude the paper with an application to equivariant correspondences. Just as $C^*$-correspondences are viewed as generalized $*$-algebra homomorphisms, equivariant $C^*$-correspondences can be viewed as generalized equivariant $*$-algebra homomorphisms between $C^*$-dynamical systems. When the group for the $C^*$-dynamical systems is trivial, the data of a strict positive equivariant $C^*$-correspondence is equivalent to the data of a strict positive $C^*$-correspondence. As the KSGNS construction transforms strict positive $C^*$-correspondences, one can ask if the KSGNS construction extends to strict positive equivariant $C^*$-correspondences. Due to the data involved in a strict positive equivariant $C^*$-correspondence, it is reasonable to expect such an extension exists; however, trying to construct an extension directly becomes unwieldy due to the amount of bookkeeping needed. 
	
	Using the category $\textnormal{PosCor}_{*\textnormal{-iso},*\textnormal{-alg}}$ with its enrichment as well as the endofunctor perspective of the KSGNS construction, we can overcome this bookkeeping issue in two steps. The first step is to realize a strict positive equivariant $C^*$-correspondence as a unitary valued topological functor $\mathcal{F}:BG\to \textnormal{PosCor}_{*\textnormal{-iso},*\textnormal{-alg}}$ where $G$ is the group acting on the $C^*$-algebras. The second step is to compose $\mathcal{F}$ with the topological endofunctor $\textnormal{KSGNS}$ on $\textnormal{PosCor}_{*\textnormal{-iso},*\textnormal{-alg}}$ to obtain a unitary valued topological functor $\textnormal{KSGNS}\circ\mathcal{F}:BG\to \textnormal{PosCor}_{*\textnormal{-iso},*\textnormal{-alg}}$. By unpacking the data encoded by $\textnormal{KSGNS}\circ\mathcal{F}$, we obtain an equivariant $C^*$-correspondence which is the dilated version of the  correspondence encoded by $\mathcal{F}$. Using the uniqueness of the KSGNS construction, the dilated equivariant $C^*$-correspondence is unitarily unique. We note, if one drops the topological requirement for the functors, we still obtain a unitarily uniquely dilated equivariant $C^*$-correspondence. The goal of this section is to rigorously implement these steps to obtain an equivariant KSGNS construction.
	
	\subsection{Review of equivariant correspondences}
	
	Before proving our assertions, we begin with a brief review of $C^*$-dynamical systems and equivariant correspondences. We include this review for completeness as well as to establish both notation and terminology. Additional information about $C^*$-dynamical systems can be found in \cite{williams2007crossed} while additional information about equivariant correspondences can found in \cite{echterhoff2000naturalityinducedrepresentations} or \cite{echterhoff2005categoricalapproachimprimitivitytheorems}.
	
	\subsubsection{$C^*$-dynamical systems}

	Let $A$ be a $C^*$-algebra, and let $G$ be a group. A group action of $G$ on $A$ is the data of a group homomorphism $\alpha:G\to \textnormal{Aut}_{*\textnormal{-alg}}(A)$. If $G$ is a topological group, then the action is continuous if $\alpha$ is continuous with respect to the point-norm topology on $\textnormal{Aut}_{*\textnormal{-alg}}(A)$. We refer to the triple $(A, G, \alpha)$ as a \textit{$C^*$-dynamical system}. In the case $G$ is a topological group and $\alpha$ is continuous, we refer to the triple as a \textit{continuous $C^*$-dynamical system}.
	
	We note that there is typically no distinction made between $C^*$-dynamical systems and continuous $C^*$-dynamical systems; in fact, if $G$ is not a topological group, one can equip $G$ with the discrete topology to make any $C^*$-dynamical system into a continuous $C^*$-dynamical system. In this case, we make the distinction as we want to keep track whether a group is topological or if the action from the group is continuous. Additionally, we note that  topological groups for continuous $C^*$-dynamical systems are typically assumed to be locally compact and Hausdorff. We do not make such an assumption here as we only need continuity of the group operation and inversion map on $G$ to obtain our main result.
	
	\subsubsection{Twisting Hilbert modules}
	
	Let $E$ be a Hilbert $B$-module, and let $\beta$ be a $*$-algebra automorphism of $B$. If $\langle\cdot, \cdot\rangle$ is the $B$-valued pairing on $E$ and $r$ is the right action of $B$ on $E$, then we can twist the pairing and action by $\beta$ as follows:
	$$(\cdot, \cdot):E\times E\to B\quad\quad (x, y)=\beta(\langle x, y\rangle)$$
	and
	$$r':E\times B\to E\quad\quad r'(x, b)=r(x, \beta^{-1}(b)).$$
	It is clear that $r'$ makes $E$ into a right $B$-module as well as $(\cdot, \cdot)$ makes $E$ into an inner product $B$-module with respect to the action $r'$. Since the induced norm from $(\cdot, \cdot)$ is the same as $\langle\cdot, \cdot\rangle$, then $E$ is a Hilbert $B$-module with respect to $(\cdot, \cdot)$ and $r'$. Let us denote $E$ with this $\beta$-twisted structure as $E^\beta$
	
	\begin{lemma}\label{Lemma: beta^{-1} unitary}\
		
		\noindent Let $E$ be a Hilbert $B$-module, and let $\beta$ be a $*$-automorphism of $B$. Then the map
		$$U_\beta:E\otimes_{\textnormal{inc}\circ\beta} B\to E^\beta\quad\quad U_\beta(x\dot{\otimes}b)=x\beta^{-1}(b)$$
		defines a $B$-linear unitary.
		
	\end{lemma}
	
	\begin{proof}\
		
		\noindent It is clear the map
		$$U_\beta:E\otimes_B B\to E^\beta\quad\quad U_\beta(x\dot{\otimes}b)=x\beta^{-1}(b)$$
		is a well-defined, $B$-linear map. Since
		\begin{align*}
			\left\|U_\beta\left(\sum\limits_{k=1}^n x_k\dot{\otimes} b_k\right)\right\|^2&=\left\|\sum\limits_{k=1}^n x_k\beta^{-1}(b_k)\right\|=\left\|\sum\limits_{i, j}b_i^*\beta(\langle x_i, x_j\rangle)b_j\right\|=\left\|\sum\limits_{k=1}^n x_k\dot{\otimes} b_k\right\|^2
		\end{align*}
		then $U_\beta$ defines a $B$-linear isometry. Therefore, we can extend $U_\beta$ to a $B$-linear isometry on $E\otimes_\beta B$ which we also denote by $U_\beta$. 
		
		We claim $U_\beta$ is a surjective map on $E\otimes_{\textnormal{inc}\circ\beta} B$. Let $x\in E$ and $(b_\lambda)_{\lambda\in\Lambda}$ be any approximate unit for $B$. Since $(xb_\lambda)_{\lambda\in\Lambda}$ converges to $x$ in $E^\beta$, $U_\beta(x\dot{\otimes}\beta(b_\lambda))=xb_\lambda$, and $U_\beta$ is an isometry, then $(x\dot{\otimes}b_\lambda)_{\lambda\in\Lambda}$ is a Cauchy net in $E\otimes_\beta B$. As $E\otimes_{\textnormal{inc}\circ\beta} B$ is complete, then the net converges to a point $z\in E\otimes_{\textnormal{inc}\circ\beta} B$. Using continuity of $U_\beta$, we have $U_\beta(z)=x$. Hence $U_\beta$ is surjective. As $U_\beta$ is $B$-linear bijective isometry, then $U_\beta$ is a unitary map.
		
	\end{proof}
	
	\begin{remark}\
		
		\noindent We observe the inverse (and thus adjoint) of $U_\beta$ is given by the map $V_{\textnormal{inc}\circ\beta}$ from Lemma \ref{Lemma:C linear maps}. 
		
	\end{remark}
	
	As the underlying Banach spaces of $E$ and $E^\beta$ are the same, then the map $U_\beta$ defines a well-defined, bijective $\mathbb{C}$-linear isometry from $E\otimes_{\textnormal{inc}\circ\beta} B$ to $E$ satisfying the following two properties:
	\begin{enumerate}
		\item for all $x\in E\otimes_\beta B$ and $y\in E$, $\langle U_\beta(x), y\rangle_{E}=\beta^{-1}(\langle x, U^{-1}_\beta(y)\rangle_{E\otimes_{\textnormal{inc}\circ\beta} B})$.
		\item for all $x\in E\otimes_\beta B$ and $b\in B$, $U_\beta(xb)=U_\beta(x)\beta^{-1}(b)$.
		
	\end{enumerate}
	Generalizing these properties for general $\mathbb{C}$-linear bounded maps between Hilbert modules yields the following definitions.

	\begin{definition}\

		\noindent Let $E_1$ and $E_2$ be Hilbert $B$-modules, and let $\beta\in\textnormal{Aut}_{*\textnormal{-alg}}(B)$. 
		\begin{itemize}
			\item A $\beta$-linear map from $E_1$ to $E_2$ is a bounded $\mathbb{C}$-linear map $f:E_1\to E_2$ such that for all $x\in E_1$ and $b\in B$, $f(xb)=f(x)\beta(b)$.
			\item A function $f:E_1\to E_2$  is $\beta$-adjointable if $f$ is $\beta$-linear and if there exists an $\beta^{-1}$-linear map $f^*:E_2\to E_1$ such that for all $x\in E_1$ and $y\in E_2$, 
			$$\langle f(x), y\rangle_{E_2}=\beta(\langle x, f^*(y)\rangle_{E_2}).$$
			Denote the set of $\beta$-adjointable maps from $E_1$ to $E_2$ as $\mathcal{L}^\beta(E_1, E_2)$.
			\item An $\beta$-adjointable map $f:E_1\to E_2$ is unitary if $ff^*=1_{E_2}$ and $f^*f=1_{E_1}$. Denote the set of $\beta$-adjointable unitaries as $\mathcal{U}^\beta(E_1, E_2)$.
			
		\end{itemize}
		
	\end{definition}

	Using the $\beta^{-1}$-adjointable unitary $U_\beta$, we obtain the following identification of $\beta$-linear maps on $E$ and adjointable $B$-linear map on $E\otimes_\beta B$.
	
	\begin{proposition}\label{Prop: rel between B-linear and beta-linear}\

		\noindent Let $E_1$ and $E_2$ be Hilbert $B$-modules, and let $\beta\in\textnormal{Aut}_{*\textnormal{-alg}}(B)$. Let $U_\beta$ be the $\beta^{-1}$-adjointable unitary given by
		$$U_\beta:E_1\otimes_{\textnormal{inc}\circ\beta} B\to E_1\quad\quad U_\beta(x\dot{\otimes} b)=x\beta^{-1}(b)$$
		Then the map
		$$\mathcal{L}^\beta(E_1, E_2)\to \mathcal{L}(E_1\otimes_{\textnormal{inc}\circ\beta} B, E_2)\quad\quad T\mapsto T\circ U_\beta$$
		defines a linear bijection. Furthermore, the map restricts to a bijection $$\mathcal{U}^\beta(E_1, E_2)\to \mathcal{U}(E_1\otimes_{\textnormal{inc}\circ\beta} B, E_2).$$
		
	\end{proposition}
	
	\begin{proof}\
		
		\noindent Provided the map is well-defined, it is clear the map is linear. Therefore, let us check the map is well-defined; that is, $T\circ U_\beta$ is $B$-linear and adjointable. Let us start by verifying $T\circ U_\beta$ is $B$-linear. As the map is $\mathbb{C}$-linear, then we just need to verify $T\circ U_\beta$ is a map of the $B$-modules. As $T\circ U_\beta$ is continuous and the action of $B$ on $E_1\otimes_{\textnormal{inc}\circ\beta} B$ and $E_2$ is continuous, then it suffices to show $B$-linearity on $E_1\otimes_BB$. Using linearity, it suffices to prove $B$-linearity on elements of the form $x\dot{\otimes}b\in E_1\otimes_BB$. Given such an element and $b'\in B$ we have
		\begin{align*}
			(T\circ U_\beta)((x\dot{\otimes} b)b')&=T(x\beta^{-1}(bb'))=T(x\beta^{-1}(b))\beta(\beta^{-1}(b'))\\
			&=T(x\beta^{-1}(b))b'=(T\circ U_\beta)(x\dot{\otimes} b)b'
		\end{align*}
		Hence $T\circ U_\beta$ is $B$-linear. 
		
		Now let us verify $T\circ U_\beta$ is adjointable. Let $T^*$ be the $\beta$-adjoint of $T$. We claim $U^{-1}_\beta\circ T^*$ is the adjoint of $T\circ U_\beta$. Since both $T\circ U_\beta$ and $U^{-1}_\beta\circ T^*$ are continuous, then it suffices to check $U^{-1}_\beta\circ T^*$ is the adjoint on the dense subspace $E_1\otimes_BB$ and $E_2$. Using linearity, it suffices to check the claim on an element $x\dot{\otimes} b\in E_1\otimes_B B$ and $y\in E_2$. Given such elements and an approximate unit $(b_\lambda)_{\lambda\in\Lambda}$ for $B$,  we have
		\begin{align*}
			\langle (T\circ U_\beta)(x\dot{\otimes} b), y\rangle_{E_2}&=\langle T(x)b, y\rangle_{E_2}=b^*\langle T(x), y\rangle_{E_2}\\
			&=b^*\beta(\langle x, T^*(y)\rangle_{E_1})=\lim\limits_{\lambda\to \infty}b^*\beta(\langle x, T^*(y)\rangle_{E_1})b_\lambda\\
			&=\lim\limits_{\lambda\to \infty}\langle x\dot{\otimes} b, T^*(y)\dot{\otimes} b_\lambda\rangle_{E_1\otimes_{\textnormal{inc}\circ\beta} B}=\lim\limits_{\lambda\to \infty}\langle x\dot{\otimes} b, V_\lambda(T^*(y))\rangle_{E_1\otimes_{\textnormal{inc}\circ\beta} B}\\
			&=\langle x\dot{\otimes} b, (U^{-1}_\beta\circ T^*)(y)\rangle_{E_1\otimes_{\textnormal{inc}\circ\beta} B}
		\end{align*}
		Therefore we conclude $T\circ U_\beta$ is adjointable with adjoint $U^{-1}_\beta\circ T^*$.
		
		Now we verify the map is a bijection. To see this, define
		$$\mathcal{L}(E_1\otimes_{\textnormal{inc}\circ\beta} B, E_2)\to \mathcal{L}^\beta(E_1, E_2)\quad\quad T\mapsto T\circ U^{-1}_\beta.$$
		Provided the map is well-defined, it is clear that it is linear as well as the inverse of the map in the claim of the proposition. To show the map above is well defined, let us begin by showing $T\circ U^{-1}_\beta$ is $\beta$-linear. Since $T\circ U^{-1}_\beta$ is $\mathbb{C}$-linear, we just need to check the $\beta$-linearity with respect to the $B$-module structure. Let $x\in E_1$ and $b\in B$. Let $(b_\lambda)_{\lambda\in\Lambda}$ be an approximate unit in $B$, then
		\begin{align*}
			(T\circ U^{-1}_\beta)(xb)&=\lim\limits_{\lambda\to\infty}T(xb\dot{\otimes} b_\lambda)=\lim\limits_{\lambda\to\infty}T(x\dot{\otimes} \beta(b)b_\lambda)\\
			&=T(x\dot{\otimes} \beta(b))=\lim\limits_{\lambda\to\infty}T(x\dot{\otimes} b_\lambda\beta(b))\\
			&=T(x\dot{\otimes} b_\lambda)\beta(b)=(T\circ U^{-1}_\beta)(x)\beta(b)
		\end{align*}
		Hence $T\circ U^{-1}_\beta$ is $\beta$-linear. 
		
		Now we claim $T\circ U^{-1}_\beta$ is a $\beta$-adjointable map with adjoint $U_\beta\circ T^*$. Given $x\in E_2$ and $b\in B$,
		\begin{align*}
			(U_\beta\circ T^*)(xb)&=U_\beta(T^*(xb))=U_\beta(T^*(x)b)=U_\beta(T^*(x))\beta^{-1}(b)=(U_\beta\circ T^*)(x)\beta^{-1}(b)
		\end{align*}
		which shows $U_\beta\circ T^*$ is $\beta^{-1}$-linear. Given $x\in E_1$ and $y\in E_2$, we have
		\begin{align*}
			\langle (T\circ U_\beta^{-1})(x), y\rangle_{E_1}&=\langle U^{-1}_\beta(x), T^*(y)\rangle_{E_1\otimes_{\textnormal{inc}\circ\beta} B}\\
			&=\langle U^{-1}_\beta(x),U^{-1}_\beta(U_\beta( T^*(y))\rangle_{E_1\otimes_{\textnormal{inc}\circ\beta} B}=\beta(\langle x, U_\beta(T^*(y))\rangle_{E_2})
		\end{align*}
		Therefore we conclude $T\circ U^{-1}_\beta$ is an adjointable $\beta$-linear map with adjoint $U_\beta\circ T^*$ proving
		$$\mathcal{L}(E_1\otimes_{\textnormal{inc}\circ\beta} B, E_2)\to \mathcal{L}^\beta(E_1, E_2)\quad\quad T\mapsto T\circ U^{-1}_\beta$$
		is well-defined.
		
		Finally, we verify the map in the proposition preserves unitaries. Suppose $T\in\mathcal{L}^\beta(E_1, E_2)$ is an adjointable $\beta$-linear unitary, then $T^*T=1_{E_1}$ and $TT^*=1_{E_2}$. Since $(T\circ U_\beta)^*=U^{-1}_\beta\circ T^*$, then
		$$(T\circ U_\beta)^*(T\circ U_\beta)=U^{-1}_\beta T^*TU_\beta=1_{E_{1}\otimes_{\textnormal{inc}\circ\beta} B}\quad\quad (T\circ U_\beta)(T\circ U_\beta)^*=TU_\beta U^{-1}_\beta T^*=1_{E_2}$$
		Hence $T\circ U_\beta$ is a $B$-linear unitary map. Now suppose $T\in\mathcal{L}(E_1\otimes_{\textnormal{inc}\circ\beta} B, E_2)$ is unitary, then $T^*T=1_{E_{1}\otimes_{\textnormal{inc}\circ\beta} B}$ and $TT^*=1_{E_2}$. Since $(T\circ U^{-1}_\beta)^*=U_\beta\circ T^*$, then
		$$(T\circ U^{-1}_\beta)^*(T\circ U^{-1}_\beta)=U_\beta T^*TU^{-1}_\beta=1_{E_1}\quad\quad (T\circ U^{-1}_\beta)(T\circ U^{-1}_\beta)^*=TU^{-1}_\beta U_\beta T^*=1_{E_2}$$
		Hence $T\circ U^{-1}_\beta$ is an adjointable $\beta$-linear unitary.
		
	\end{proof}
	
	\subsubsection{Equivariant correspondences}

	The notion of a $C^*$-correspondence is a generalization of a non-degenerate $*$-algebra homomorphism between the algebras. In a similar manner, an equivariant $C^*$-correspondence between $C^*$-dynamical systems is viewed as a generalization of equivariant $*$-algebra maps between $C^*$-dynamical systems.  To understand the data we take in the definition of such a correspondence, let us look at an equivariant $*$-algebra homomorphism between two $C^*$-dynamical systems.
	
	Let $(A, G, \alpha)$ and $(B, G, \beta)$ be $C^*$-dynamical systems, and let $\pi:A\to B$ be an equivariant $*$-algebra homomorphism; that is, for all $g\in G$ and $a\in A$, $\pi(\alpha_g(a))=\beta_g(\pi(a))$. Let $E=B$ viewed as a Hilbert $B$-module. For each $g\in G$, $U_g=\beta_g$ defines an adjoinable $\beta_g$-linear unitary with adjoint given by $U_{g^{-1}}$.Observe, for all $g\in G$, $x\in E$, $a\in A$,
	$$U_g(\pi(a)x)=U_g(\pi(a))U_g(x)=\pi(\alpha_g(a))U_g(x)$$
	Generalizing the maps $U_g$ and generalizing $E$ to a general Hilbert $B$-module yields the appropriate notion of an equivariant $C^*$-correspondence.
	
	\begin{definition}\

		\noindent Let $(A, G, \alpha)$ and $(B, G, \beta)$ be $C^*$-dynamical systems. An equivariant $C^*$-correspondence from $(A, G, \alpha)$ to $(B, G, \beta)$ is the data $((E, \pi), U)$ where 
		\begin{enumerate}
			\item $(E, \pi)$ is a $C^*$-correspondence from $A$ to $B$.
			\item $U:G\to \mathcal{B}(E)$ is a group homomorphism 
		\end{enumerate}
		and the data satisfies the following conditions:
		\begin{enumerate}
			\item For all $g\in G$, $U_g$ is a $\beta_g$-adjointable unitary on $E$.
			\item For all $g\in G$ and $a\in A$, $U_g\circ \pi(a)=\pi(\alpha_g(a))\circ U_g$.
		\end{enumerate}
		If $(E, \pi)$ is a strict (non-degenerate) positive $C^*$-correspondence, then $((E, \pi), U)$ is called a strict (non-degenerate) positive equivariant $C^*$-correspondence. If $(A, G, \alpha)$ and $(B, G, \beta)$ are continuous $C^*$-dynamical systems, then the (strict/non-degenerate positive) equivariant $C^*$-correspondence $((E, \pi), U)$ is continuous if $U$ is strongly continuous.  
		
	\end{definition}
	
	\subsection{Functorial representation of strict positive equivariant $C^*$-correspondences}
	
	Let $G$ be a group. Let $BG$ denote the category given by a single object (denote by $\textnormal{pt}$), morphisms given by $G$, and composition given by the group law of $G$. If $G$ is a topological group, we enrich $BG$ over topological spaces so that the topology on the hom-set $\textnormal{hom}(\textnormal{pt}, \textnormal{pt})=G$ agrees with the topology on $G$.
	
	\begin{definition}\
		
		\noindent Let $A$ be a $C^*$-algebra, and let $G$ be a group. A functor $\mathcal{F}:BG\to \textnormal{PosCor}_{*\textnormal{-alg}}$ is unitary valued if for all $g\in G$, $\mathcal{F}(g)=(\rho_g, (\eta_g, \alpha_g))$ has the property that $\eta_g$ is a unitary.
		
	\end{definition}

	\begin{theorem}\label{Theorem: ESPC as Functors}\
		
		\begin{enumerate}
			\item Let $(A, G, \alpha)$ and $(B, G, \beta)$ be continuous $C^*$-dynamical systems. Every continuous strict positive equivariant $C^*$-correspondence from $(A, G, \alpha)$ to $(B, G, \beta)$ determines and is determined by a unitary valued topological functor $\mathcal{F}:BG\to \textnormal{PosCor}_{*\textnormal{-alg}}$ such that
			\begin{enumerate}
				\item $\mathcal{F}(\textnormal{pt})=(A, B, (E, \phi))$.
				\item For all $g\in G$, $\mathcal{F}(g)=(\alpha_g, \textnormal{inc}\circ\beta_g, \eta_g)$.
			\end{enumerate}
			\item Let $(A, G, \alpha)$ and $(B, G, \beta)$ be $C^*$-dynamical systems. Every strict positive equivariant $C^*$-correspondence from $(A, G, \alpha)$ to $(B, G, \beta)$ determines and is determined by a unitary valued functor $\mathcal{F}:BG\to \textnormal{PosCor}_{*\textnormal{-alg}}$ such that
			\begin{enumerate}
				\item $\mathcal{F}(\textnormal{pt})=(A, B, (E, \phi))$.
				\item For all $g\in G$, $\mathcal{F}(g)=(\alpha_g, \textnormal{inc}\circ\beta_g, \eta_g)$.
			\end{enumerate}
		\end{enumerate}
		
	\end{theorem}
	
	\begin{proof}\
		
		\noindent As the proof of $(1)$ and $(2)$ are similar, we only prove $(1)$; however, we point out where the lack of continuity in $(2)$ shows up. First suppose we have a unitary valued topological functor $\mathcal{F}:BG\to \textnormal{PosCor}_{*\textnormal{-alg}}$ satisfying the two conditions. Therefore, we obtain a continuous group homomorphism
		$$G\to \textnormal{Aut}(A, B, (E, \phi))\quad\quad g\mapsto (\alpha_g, \textnormal{inc}\circ\beta_g, \eta_g)$$
		Define
		$$W:G\to \mathcal{B}(E)\quad\quad W(g)=\eta_g\circ V_{\textnormal{inc}\circ\beta_g}$$
		where $V_{\textnormal{inc}\circ\beta_g}$ is the $\mathbb{C}$-linear contraction from Lemma \ref{Lemma:C linear maps}. As $\beta_g$ is a $*$-automorphism of $B$, then $V_{\textnormal{inc}\circ\beta_g}$ is the adjoint of the $\beta^{-1}$-adjointable unitary $U_{\beta_g}$ from Lemma \ref{Lemma: beta^{-1} unitary}. Using Proposition \ref{Prop: rel between B-linear and beta-linear}, the map $W(g)$ is a $\beta_g$-adjointable unitary on $E$. 
		
		We claim $W$ is a group homomorphism. Let $g, h\in G$, then
		$$(\alpha_g, \textnormal{inc}\circ\beta_g, \eta_g)\circ (\alpha_h, \textnormal{inc}\circ\beta_h, \eta_h)=(\alpha_g\circ\alpha_h, \textnormal{inc}\circ\beta_g\circ\beta_h, \eta_g\circ\hat{\eta}_h\circ U^{-1})=(\alpha_{gh}, \textnormal{inc}\circ\beta_{gh},\eta_g\circ\hat{\eta}_h\circ U^{-1})$$ where
		$$U:(E\otimes_{\textnormal{inc}\circ\beta_g}B)\otimes_{\textnormal{inc}\circ\beta_h} B\to E_{\beta_{g}\circ\beta_h}B\quad\quad U((x\dot{\otimes}b_1)\dot{\otimes}b_2)=x\dot{\otimes}\beta_h(b_1)b_2$$ 
		is the unitary from Proposition \ref{Prop: Composition and tensor}. From functoriality of $\mathcal{F}$, we know
		$$\eta_{gh}=\eta_g\circ\hat{\eta}_h\circ U^{-1}$$
		Thus
		$$W(gh)=\eta_{gh}\circ V_{\beta_{gh}}=\eta_g\circ\hat{\eta}_h\circ U^{-1}\circ V_{\textnormal{inc}\circ\beta_{gh}}.$$
		Using Lemma \ref{Lemma: Diagram1} and \ref{Lemma: Diagram2}, we can write
		$$\eta_g\circ\hat{\eta}_h\circ U^{-1}\circ V_{\textnormal{inc}\circ\beta_{gh}}=(\eta_g\circ V_{\textnormal{inc}\circ\beta_g})\circ (\eta_h\circ V_{\textnormal{inc}\circ\beta_h})=W(g)\circ W(h)$$
		Hence $W(gh)=W(g)W(h)$ showing $W$ is a group homomorphism.
		
		Now we claim $((E, \phi), W)$ defines a strict positive equivariant $C^*$-correspondence from $(A, G, \alpha)$ to $(B, G, \beta)$. The only condition on the pair $((E, \phi), W)$ that we need to check at this point is that for all $g\in G$ and $a\in A$, $W_g\circ \phi(a)=\phi(\alpha_g(a))\circ W_g$.	Fix $a\in A$, and consider the following diagram:
		\begin{center}
			\begin{tikzcd}
				E &&& {E\otimes_{\textnormal{inc}\circ\beta_g}B} &&& E \\
				\\
				E &&& {E\otimes_{\textnormal{inc}\circ\beta_g}B} &&& E
				\arrow["{V_{\textnormal{inc}\circ\beta_g}}"{description}, from=1-1, to=1-4]
				\arrow["{W_g}"{description}, bend left=20, from=1-1, to=1-7]
				\arrow["{\phi(a)}"{description}, from=1-1, to=3-1]
				\arrow["{\eta_g}"{description}, from=1-4, to=1-7]
				\arrow["{\phi(a)\otimes_{\textnormal{inc}\circ\beta_g}I}"{description}, from=1-4, to=3-4]
				\arrow["{\phi(\alpha_g(a))}"{description}, from=1-7, to=3-7]
				\arrow["{V_{\textnormal{inc}\circ\beta_g}}", from=3-1, to=3-4]
				\arrow["{W_g}"{description}, bend right=20, from=3-1, to=3-7]
				\arrow["{\eta_g}"{description}, from=3-4, to=3-7]
			\end{tikzcd}
		\end{center}
		As $(\alpha_g, \textnormal{inc}\circ\beta_g, \eta_g)$ is a morphism from $(A, B, (E, \phi))$ to $(A, B, (E, \phi))$, then the middle right square commutes. Using an approximate unit for $B$, the middle left square commutes. Using commutativity of these two diagrams, it follows the outside commutes; that is,  $W_g\circ \phi(a)=\phi(\alpha_g(a))\circ W_g$. Hence the data $((E, \phi), W)$ determines a strict positive equivariant $C^*$-correspondence from $(A, G, \alpha)$ to $(B, G, \beta)$.
		
		Notice, we have not yet used continuity of the group homomorphism at the start as well as that the $C^*$-dynamical systems are continuous. We now use continuity of the group homomorphism to show the strict positive equivariant $C^*$-correspondence is continuous; that is, $W$ is a strongly continuous. To prove continuity, it suffices to show $W$ maps convergent nets to convergent nets. Let $(g_\lambda)_{\lambda\in \Lambda}$ be a net in $G$ which converges to $g\in G$. As
		$$G\to \textnormal{Aut}(A, B, (E, \phi))\quad\quad g\mapsto (\alpha_g, \textnormal{inc}\circ\beta_g, \eta_g)$$
		is continuous, then the net $((\alpha_{g_\lambda}, \textnormal{inc}\circ \beta_{g_\lambda}, \eta_{g_\lambda}))_{\lambda\in \Lambda}$ converges to $(\alpha_g, \textnormal{inc}\circ\beta_{g}, \eta_{g})$. From the topology on the hom-sets of $\textnormal{PosCor}_{*\textnormal{-alg}}$, this implies $(W_{g_\lambda})_{\lambda\in\Lambda}=(\eta_{g_\lambda}\circ V_{\textnormal{inc}\circ\beta_{g_\lambda}})_{\lambda\in\Lambda}$ converges to $\eta_g\circ V_{\textnormal{inc}\circ\beta_g}=W_g$ in the strong operator topology. Hence $W$ is strongly continuous.
		
		We now prove the reverse direction. Let $((E, \phi), W)$ be a continuous strict positive equivariant $C^*$-correspondence from $(A, G, \alpha)$ to $(B, G, \beta)$. For each $g\in G$, let $\eta_g=W_g\circ U_{\beta_g}$ which, from Proposition \ref{Prop: rel between B-linear and beta-linear}, is a $B$-linear unitary from $E\otimes_{\textnormal{inc}\circ\beta_g}B$ to $E$. Define the assignment
		\begin{align*}
			\mathcal{F}:BG&\to \textnormal{PosCor}_{*\textnormal{-alg}}\\
			\{\textnormal{pt}\}&\mapsto (A, B, (E, \phi))\\
			g & \mapsto (\alpha_g, \textnormal{inc}\circ\beta_g, \eta_g)
		\end{align*}
		We claim the assignment defines a unitary valued topological functor. 
		
		We begin by checking the assignment is well-defined for each $g\in G$; that is, we verify for all $g\in G$, $(\alpha_g, \textnormal{inc}\circ\beta_g, \eta_g)$ is a morphism from $(A, B, (E, \phi))$ to $(A, B, (E, \phi))$. To conclude the pair is a morphism, we just need to show for all $a\in A$, $ \eta_g\circ (\phi(a)\otimes_{\textnormal{inc}\circ\beta_g}I)=\phi(\alpha_g(a))\circ \eta_g$. Fix $a\in A$, and consider the diagram
		\begin{center}
			\begin{tikzcd}
				E &&& {E\otimes_{\textnormal{inc}\circ\beta_g}B} &&& E \\
				\\
				E &&& {E\otimes_{\textnormal{inc}\circ\beta_g}B} &&& E
				\arrow["{V_{\textnormal{inc}\circ\beta_g}}"{description}, from=1-1, to=1-4]
				\arrow["{W_g}"{description}, bend left=20, from=1-1, to=1-7]
				\arrow["{\phi(a)}"{description}, from=1-1, to=3-1]
				\arrow["{\eta_g}"{description}, from=1-4, to=1-7]
				\arrow["{\phi(a)\otimes_{\textnormal{inc}\circ\beta_g}I}"{description}, from=1-4, to=3-4]
				\arrow["{\phi(\alpha_g(a))}"{description}, from=1-7, to=3-7]
				\arrow["{V_{\textnormal{inc}\circ\beta_g}}", from=3-1, to=3-4]
				\arrow["{W_g}"{description}, bend right=20, from=3-1, to=3-7]
				\arrow["{\eta_g}"{description}, from=3-4, to=3-7]
			\end{tikzcd}
		\end{center}
		By assumption, the outer diagram commutes. As before, the small square on the left commutes. Using commutativity of these two diagrams as well as that $V_{\textnormal{inc}\circ\beta_g}$ is invertible, it follows the small right square commutes. Hence $ \eta_g\circ (\phi(a)\otimes_{\textnormal{inc}\circ\beta_g}I)=\phi(\alpha_g(a))\circ \eta_g$ so that the pair $(\alpha_g, \textnormal{inc}\circ\beta_g, \eta_g)$ is a morphism.
		
		Now we check the assignment is functorial. Since $\mathcal{F}(e)=(1_A, \textnormal{inc}, \iota)$ which is the identity morphism for $(A, B, (E, \phi))$, then $\mathcal{F}$ preserves the identity morphism. Now we check the assignment preserves composition. Let $g, h\in G$, then
		\begin{align*}
			\mathcal{F}(g)\circ\mathcal{F}(h)&=(\alpha_g, \textnormal{inc}\circ\beta_g, \eta_g)\circ (\alpha_h, \textnormal{inc}\circ\beta_{h}, \eta_h)\\
			&=(\alpha_g\circ\alpha_h, \textnormal{inc}\circ\beta_{g}\circ \beta_{h}, \eta_g\circ \hat{\eta}_h\circ U^{-1})\\
			&=(\alpha_{gh}, \textnormal{inc}\circ\beta_{gh}, \eta_g\circ \hat{\eta}_h\circ U^{-1})
		\end{align*}
		Using Lemma \ref{Lemma: Diagram1} and Lemma \ref{Lemma: Diagram2}, we know
		\begin{align*}
			\eta_g\circ \hat{\eta}_h\circ U^{-1}\circ V_{\textnormal{inc}\circ\beta_{gh}}&=\eta_g\circ \hat{\eta}_h\circ U^{-1}\circ V_{\textnormal{inc}\circ\beta_{g}\circ\textnormal{inc}\circ\beta_h}=(\eta_g\circ V_{\textnormal{inc}\circ\beta_g})\circ (\eta_{h}\circ V_{\textnormal{inc}\circ\beta_h})\\
			&=W_g\circ W_h=W_{gh}=\eta_{gh}\circ V_{\textnormal{inc}\circ\beta_{gh}}
		\end{align*}
		Thus, as $V_{\textnormal{inc}\circ\beta_{gh}}$ is invertible, $\eta_{gh}=\eta_g\circ \hat{\eta}_h\circ U^{-1}$. Hence
		$$\mathcal{F}(g)\circ \mathcal{F}(h)=(\alpha_{gh}, \textnormal{inc}\circ\beta_{gh}, \eta_g\circ \hat{\eta}_h\circ U^{-1})=(\alpha_{gh}, \textnormal{inc}\circ\beta_{gh}, \eta_{gh})=\mathcal{F}(gh)$$
		showing $\mathcal{F}$ preserves composition. Having checked the condition for functoriality, we conclude $\mathcal{F}$ is a covariant functor. As $\eta_g$ is a unitary for each $g\in G$, then $\mathcal{F}$ defines a unitary valued functor satisfying the two conditions.
		
		Notice, we have yet to use the fact that the correspondence is continuous and the $C^*$-dynamical systems are continuous. We use these facts now to show $\mathcal{F}$ is topological; that is, the group homomorphism
		$$G\to \textnormal{Aut}(A, B, (E, \phi))\quad\quad g\mapsto (\alpha_g, \textnormal{inc}\circ\beta_g, \eta_g)$$
		is continuous. To prove the map is continuous, it suffices to show convergent nets are mapped to convergent nets. Therefore, let $(g_\lambda)_{\lambda\in\Lambda}$ be a net in $G$ which converges to $g$. We need to show $((\alpha_g, \textnormal{inc}\circ\beta_{g_\lambda}, \eta_{g_\lambda}, ))_{\lambda\in\Lambda}$ converges in the topology on $\textnormal{Aut}((A, B, (E, \phi))$ to $(\alpha_g, \textnormal{inc}\circ\beta_g, \eta_g)$. From the topology on $\textnormal{Aut}(A, B, (E, \phi))$, this is equivalent to both $\alpha_{g_\lambda}\to \alpha_g$ and $\textnormal{inc}\circ\beta_{g_\lambda}\to \textnormal{inc}\circ\beta_g$ in the point-norm topology on $\textnormal{Hom}_{*\textnormal{-alg}}(A, A)$ and $\textnormal{Hom}_{*\textnormal{-alg}}(B, M(B))$, respectively, as well as $$W_{g_\lambda}=\eta_{g_\lambda}\circ V_{\textnormal{inc}\circ\beta_{g_\lambda}}\to \eta_g\circ V_{\textnormal{inc}\circ\beta_g}=W_g$$
		in the strong operator topology on $\mathcal{B}(E)$. As $\alpha$ and $\beta$ are continuous with respect to the point-norm topology on $\textnormal{Aut}_{*\textnormal{-alg}}(A)$ and $\textnormal{Aut}_{*\textnormal{-alg}}(B)$, then $\alpha_{g_\lambda}\to \alpha_g$ and $\textnormal{inc}\circ\beta_{g_\lambda}\to \textnormal{inc}\circ \beta_g$. Similarly, as $W$ is continuous with respect to the strong operator topology on $\mathcal{B}(E)$, then $\eta_{g_\lambda}\circ V_{\textnormal{inc}\circ\beta_{g_\lambda}}\to \eta_g\circ V_{\textnormal{inc}\circ\beta_g}$. Thus $(\alpha_g, \textnormal{inc}\circ\beta_{g_\lambda}, \eta_{g_\lambda}))_{\lambda\in\Lambda}$ converges to $(\alpha_g, \textnormal{inc}\circ\beta_g, \eta_g)$ showing $\mathcal{F}$ is topological.
		
		From how we constructed the correspondence from the functor data as well as how we constructed the functor from the correspondence data, it is clear these constructions are inverses of each other. 
		
	\end{proof}
	
	\begin{corollary}\

		\begin{enumerate}
			\item Let $(A, G, \alpha)$ and $(B, G, \beta)$ be continuous $C^*$-dynamical systems. Every continuous non-degenerate positive equivariant $C^*$-correspondence from $(A, G, \alpha)$ to $(B, G, \beta)$ determines and is determined by a unitary valued topological functor $\mathcal{F}:BG\to \textnormal{PosCor}_{*\textnormal{-alg}}$ such that
			\begin{enumerate}
				\item $\mathcal{F}(\textnormal{pt})=(A, B, (E, \phi))$ is a non-degenerate positive $C^*$-correspondence from $A$ to $B$.
				\item For all $g\in G$, $\mathcal{F}(g)=(\alpha_g, \textnormal{inc}\circ \beta_g, \eta_g)$.
			\end{enumerate}
			\item Let $(A, G, \alpha)$ and $(B, G, \beta)$ be $C^*$-dynamical systems. Every non-degenerate positive equivariant $C^*$-correspondence from $(A, G, \alpha)$ to $(B, G, \beta)$ determines and is determined by a unitary valued functor $\mathcal{F}:BG\to \textnormal{PosCor}_{*\textnormal{-alg}}$ such that
			\begin{enumerate}
				\item $\mathcal{F}(\textnormal{pt})=(A, B, (E, \phi))$ is a non-degenerate positive $C^*$-correspondence from $A$ to $B$.
				\item For all $g\in G$, $\mathcal{F}(g)=(\alpha_g, \textnormal{inc}\circ\beta_g, \eta_g)$.
			\end{enumerate}
		\end{enumerate}
		
	\end{corollary}

	\begin{corollary}\label{Cor: Important Cor}\

		\begin{enumerate}
			\item Let $(A, G, \alpha)$ and $(B, G, \beta)$ be continuous $C^*$-dynamical systems. Every continuous equivariant $C^*$-correspondence from $(A, G, \alpha)$ to $(B, G, \beta)$ determines and is determined by a unitary valued topological functor $\mathcal{F}:BG\to \textnormal{PosCor}_{*\textnormal{-alg}}$ such that
			\begin{enumerate}
				\item $\mathcal{F}(\textnormal{pt})=(A, B, (E, \phi))$ is a $C^*$-correspondence from $A$ to $B$
				\item For all $g\in G$, $\mathcal{F}(g)=(\alpha_g, \textnormal{inc}\circ\beta_g, \eta_g)$.
			\end{enumerate}
			\item Let $(A, G, \alpha)$ and $(B, G, \beta)$ be $C^*$-dynamical systems. Every equivariant $C^*$-correspondence from $(A, G, \alpha)$ to $(B, G, \beta)$ determines and is determined by a unitary valued functor $\mathcal{F}:BG\to \textnormal{PosCor}_{*\textnormal{-alg}}$ such that
			\begin{enumerate}
				\item $\mathcal{F}(\textnormal{pt})=(A, B, (E, \phi))$ is a $C^*$-correspondence from $A$ to $B$
				\item For all $g\in G$, $\mathcal{F}(g)=(\alpha_g, \textnormal{inc}\circ\beta_g, \eta_g)$.
			\end{enumerate}
		\end{enumerate}
		
	\end{corollary}
	
	\subsection{Equivariant KSGNS construction}
	
	Using the functorial representation of a equivariant correspondence along with the KSGNS endofucntor on $\textnormal{PosCor}_{*\textnormal{-iso}, *\textnormal{-alg}}$, we obtain the following equivariant version of the KSGNS construction.
	
	\begin{theorem}\label{Theorem: Dilation Result}\

		\noindent Let $(A, G, \alpha)$ and $(B, G, \beta)$ be continuous $C^*$-dynamical systems. If $((E, \phi), U)$ is a continuous strict positive equivariant $C^*$-correspondence from $(A, G, \alpha)$ to $(B, G, \beta)$, then there exists a quadruple $(F_\phi, \pi_\phi, V_\phi, \widetilde{U})$ where
		\begin{enumerate}
			\item $F_\phi$ is a Hilbert $B$-module.
			\item $\pi_\phi:A\to \mathcal{L}(F_\phi)$ is a non-degenerate $*$-algebra homomorphism.
			\item $V_\phi:E\to F_\phi$ is an adjointable $B$-linear map.
			\item $\widetilde{U}:G\to \mathcal{B}(F_\phi)$ is a strongly continuous representation for which $\widetilde{U}_g\in \mathcal{U}^{\beta_g}(F_\phi)$ for all $g\in G$.
		\end{enumerate}
		and the quadruple satisfies the following conditions:
		\begin{enumerate}
			\item $((F_\phi, \pi_\phi), \widetilde{U})$ is an equivariant $C^*$-correspondence from $(A, G, \alpha)$ to $(B, G, \beta)$.
			\item $\pi_\phi(A)V_\phi E$ is dense in $F_\phi$.
			\item For all $a\in A$, $\phi(a)=V_\phi^*\pi_\phi(a)V_\phi$.
			\item For all $g\in G$, $V_\phi\circ U_g=\tilde{U}_g\circ V_\phi$ 
		\end{enumerate}
		If $(F', \pi', V', U')$ is another quadruple satisfying conditions $(1)-(4)$, then there exists a $B$-linear unitary $W:F'\to F_\phi$ such that
		\begin{enumerate}
			\item $\pi_\phi(a)=W\pi'(a)W^*$ for all $a\in A$.
			\item $V_\phi=WV'$
			\item For all $g\in G$, $\widetilde{U}_g=WU'_gW^*$
		\end{enumerate}
		
	\end{theorem}
	
	\begin{proof}\
		
		\noindent Let us first prove uniqueness. Let $(F_\phi, \pi_\phi, V_\phi, \tilde{U})$ and $(F', \pi', V', U')$ be two quadruples satisfying conditions $(1)-(4)$. As $(F_\phi, \pi_\phi, V_\phi)$ and $(F', \pi', V')$ are two KSGNS dilations of $(E, \phi)$, then there exists a unitary $W:F'\to F_\phi$ satisfying $(1)$ and $(2)$. Note, $W$ is defined on the dense submodule $\pi'(A)V'E$ as
		$$W(\sum\limits_{k=1}^n \pi'(a_k)V'x_k)=\sum\limits_{k=1}^n \pi_\phi(a_k)V_\phi x_k$$
		To prove $(3)$, we can use continuity of $U_g, U_g',$ and $W$ to prove equality of $\tilde{U}_g=WU'_gW^*$ on the dense subspace $\pi_\phi(A)V_\phi E$. Using linearity, it suffices to prove equality on elements of the form $\pi_\phi(a)V_\phi x$ for $a\in A$ and $x\in E$. Given such elements, we use the definition of $W$ on the dense subspace as well as condition $(1)$ and $(4)$ of the quadruples to obtain
		\begin{align*}
			WU'_gW^*(\pi_\phi(a)V_\phi x)&=WU'_g(\pi'(a)V' x)=W(\pi'(\alpha_g(a))U'_gV'x)\\
			&=W(\pi'(\alpha_g(a))V'U_gx)=\pi_\phi(\alpha_g(a))V_\phi U_gx\\
			&=\pi_\phi(\alpha_g(a))\tilde{U}_gV_\phi x=\tilde{U}_g(\pi_\phi(a)V_\phi x)
		\end{align*}
		Hence $\tilde{U}_g=WU'_gW^*$.
		
		Let us now prove the existence of a quadruple. Let $\mathcal{F}:BG\to \textnormal{PosCor}_{*\textnormal{-alg}}$ be the corresponding unitary valued topological functor from the continuous strict positive equivariant $C^*$-corresopndence $((E, \phi), U)$. For each $g\in G$, we have $\mathcal{F}(g)=(\alpha_g, \textnormal{inc}\circ\beta_g, \eta_g)$ where $\eta_g=U_g\circ U_{\beta_g}$ for $U_{\beta_g}$ the $\beta^{-1}$-adjointable unitary from Lemma \ref{Lemma: beta^{-1} unitary}. As $\mathcal{F}$ factors through $\textnormal{PosCor}_{*\textnormal{-iso}, *\textnormal{-alg}}$, we can post-compose $\mathcal{F}$ by $\textnormal{KSGNS}$ to obtain the functor $\textnormal{KSGNS}\circ\mathcal{F}$.
		
		We claim $\textnormal{KSGNS}\circ\mathcal{F}$ is unitary valued. For each $g\in G$, we obtain an automorphism $(\alpha_g, \textnormal{inc}\circ\beta_g, \tilde{\eta}_g\circ V_g^{-1}, )$ of $(A, B, (F_\phi, \pi_\phi))$  where 
		$$V_g:A\otimes_{\phi\otimes 1_B}(E\otimes_{\textnormal{inc}\circ\beta_g} B)\to F_{\phi}\otimes_{\textnormal{inc}\circ\beta_g}B$$ is the unitary given in Theorem \ref{Theorem: KSNGS commutes with ITP}. As $\mathcal{F}$ is unitary valued, then $\eta_g$ is a unitary for all $g\in G$. From the construction of $\tilde{\eta}_g$ in Proposition \ref{Prop: Lifting 2}, we know $\tilde{\eta}_g$ is also a unitary. Thus $\tilde{\eta}_g\circ V_{g}^{-1}$ is unitary showing $\textnormal{KSGNS}\circ \mathcal{F}$ is unitary valued.
		
		We claim $\textnormal{KSGNS}\circ \mathcal{F}$ is a topological functor. As $\mathcal{F}$ is unitary valued, $\mathcal{F}$ maps morphisms $\textnormal{pt}\to \textnormal{pt}$ to morphisms from $(A, B, (E, \phi))$ to $(A, B, (E, \phi))$ for which the adjointable maps are uniformly norm bounded (in fact, they are bound by one). Thus, Proposition \ref{Prop: Continuity of KSGNS} implies the assignment on morphism given by $\textnormal{KSGNS}\circ\mathcal{F}$ is continuous; that is, $\textnormal{KSGNS}\circ \mathcal{F}$ is a topological functor.
		
		As $\textnormal{KSGNS}\circ \mathcal{F}$ is a unitary valued topological functor with
		$$(\textnormal{KSGNS}\circ\mathcal{F})(\textnormal{pt})=(A, B, (F_\phi, \pi_\phi))$$
		and for all $g\in G$,
		$$(\textnormal{KSGNS}\circ\mathcal{F})(g)=(\alpha_g, \textnormal{inc}\circ\beta_g, \tilde{\eta}_g\circ V_g^{-1})$$
		then Corollary \ref{Cor: Important Cor} implies $\textnormal{KSGNS}\circ\mathcal{F}$ induces a continuous equivariant $C^*$-correspondence $((F_\phi, \pi_\phi), \tilde{U})$ from $(A, G, \alpha)$ to $(B, G, \beta)$.
		
		We claim the quadruple $(F_\phi, \pi_\phi, V_\phi, \tilde{U})$ satisfies conditions $(1)-(4)$. From the end of the prior paragraph, we know condition $(1)$ is satisfied. As $(F_\phi, \pi_\phi, V_\phi)$ is the KSGNS dilation of $(E, \phi)$, then we know the quadruple satisfies conditions $(2)$ and $(3)$.  Therefore, all that remains to be seen is the quadruple satisfies condition $(4)$. Fix $g\in G$ and consider the following diagram:
		\begin{center}
			\begin{tikzcd}
				{E\otimes_{\textnormal{inc}\circ\beta_g}B} &&& {A\otimes_{\phi\otimes_{\textnormal{inc}\circ\beta_{g}} I}(E\otimes_{\textnormal{inc}\circ\beta_g}B)} && {F_\phi\otimes_{\textnormal{inc}\circ\beta_g}B} \\
				\\
				E &&&&& {F_\phi}
				\arrow["{V_{\phi\otimes_{\textnormal{inc}\circ\beta_g}I}}"{description}, from=1-1, to=1-4]
				\arrow["{U_{\beta_g}}"{description}, from=1-1, to=3-1]
				\arrow["{V_g}"{description}, from=1-4, to=1-6]
				\arrow["{U_{\beta_g}'}"{description}, from=1-6, to=3-6]
				\arrow["{V_\phi}"{description}, from=3-1, to=3-6]
			\end{tikzcd} 
		\end{center}
		We claim the diagram commutes. Using continuity and density, it suffices to show commutativity on the dense subspace $E\otimes_BB$.  Using linearity, it suffice to show commutativity on elements of the form $x\dot{\otimes}b\in E\otimes_BB$. Given such an element and an approximate unit $(a_\lambda)_{\lambda\in\Lambda}$ for $A$, we have
		\begin{align*}
			(V_\phi\circ U_{\beta_g})(x\dot{\otimes}b)&=V_\phi(x\beta_g^{-1}(b))=\lim\limits_{\lambda\to\infty}a_\lambda\dot{\otimes} x\beta_g^{-1}(b)\\
			&=\lim\limits_{\lambda\to\infty}U_{\beta_g}'((a_\lambda\dot{\otimes}x)\dot{\otimes}b)=\lim\limits_{\lambda\to\infty}(U_{\beta_g}'\circ V_g)(a_\lambda\dot{\otimes}(x\dot{\otimes}b))\\
			&=(U_{\beta_g}'\circ V_g\circ V_{\phi\otimes_{\textnormal{inc}\circ\beta_g} I})(x\dot{\otimes}b)
		\end{align*}
		Hence the diagram commutes.
		
		Now consider the following diagram:
		\begin{center}
			\begin{tikzcd}
				& {E\otimes_{\textnormal{inc}\circ\beta_g}B} &&& {A\otimes_{\phi\otimes_{\textnormal{inc}\circ\beta_g} I}(E\otimes_{\textnormal{inc}\circ\beta_g}B)} && {F_\phi\otimes_{\textnormal{inc}\circ\beta_g}B} \\
				\\
				E & E &&& {F_\phi} && {F_\phi}
				\arrow["{V_{\phi\otimes_{\textnormal{inc}\circ\beta_g} I}}"{description}, from=1-2, to=1-5]
				\arrow["\eta_g"{description}, from=1-2, to=3-2]
				\arrow["{\tilde{\eta}_g}"{description}, from=1-5, to=3-5]
				\arrow["{V_g^{-1}}"{description}, from=1-7, to=1-5]
				\arrow["{V_{\textnormal{inc}\circ\beta_g}}"{description}, from=3-1, to=1-2]
				\arrow["{U_g}"{description}, from=3-1, to=3-2]
				\arrow["{V_\phi}"{description}, bend right=20, from=3-1, to=3-7]
				\arrow["{V_{\phi}}"{description}, from=3-2, to=3-5]
				\arrow["{V_{\textnormal{inc}\circ\beta_g}'}"{description}, from=3-7, to=1-7]
				\arrow["{\widetilde{U}_g}"{description}, from=3-7, to=3-5]
			\end{tikzcd}
		\end{center}
		By definition, the triangle on the left and the square on the far right commute. Using property $(4)$ in Proposition \ref{Prop: Lifting 2}, the middle square commutes. Using commutativity of these three diagrams as well as commutativity of the diagram above, we conclude the bottom loop commute:
		\begin{align*}
			V_\phi\circ U_g&=V_\phi\circ \eta\circ V_{\textnormal{inc}\circ\beta_g}=\tilde{\eta}\circ V_{\phi\otimes_{\textnormal{inc}\circ\beta_g} I}\circ V_{\textnormal{inc}\circ\beta_g}\\
			&=\tilde{\eta}\circ V_g^{-1}\circ V_{\textnormal{inc}\circ\beta_g}'\circ V_\phi=\tilde{U}_g\circ V_\phi
		\end{align*}
		Therefore $V_\phi\circ U_g=\tilde{U}_g\circ V_\phi$ proving the claim.
		
	\end{proof}
	
	We observe continuity of the strict positive equivariant correspondence and continuity of the $C^*$-dynamical systems was used to know the resulting functor $\textnormal{KSGNS}\circ\mathcal{F}$ is a topological functor. Thus, the result and proof of the theorem holds when we have $C^*$-dynamical systems and strict positive equivariant correspondences. When we are in this relaxed setting, we obtain from Corollary \ref{Cor: Important Cor} that $\textnormal{KSGNS}\circ \mathcal{F}$ encodes an equivariant $C^*$-correspondence between the $C^*$-dynamical systems. This leads to the following version of Theorem \ref{Theorem: Dilation Result}.
	
	\begin{theorem}\

		\noindent Let $(A, G, \alpha)$ and $(B, G, \beta)$ be $C^*$-dynamical systems. If $((E, \phi), U)$ is a strict positive equivariant $C^*$-correspondence from $(A, G, \alpha)$ to $(B, G, \beta)$, then there exists a quadruple $(F_\phi, \pi_\phi, V_\phi, \widetilde{U})$ where
		\begin{enumerate}
			\item $F_\phi$ is a Hilbert $B$-module.
			\item $\pi_\phi:A\to \mathcal{L}(F_\phi)$ is a non-degenerate $*$-algebra homomorphism.
			\item $V_\phi:E\to F_\phi$ is an adjointable $B$-linear map.
			\item $\widetilde{U}:G\to \mathcal{B}(F_\phi)$ is a  representation for which $\widetilde{U}_g\in \mathcal{U}^{\beta_g}(F_\phi)$ for all $g\in G$.
		\end{enumerate}
		and the quadruple satisfies the following conditions:
		\begin{enumerate}
			\item $((F_\phi, \pi_\phi), \widetilde{U})$ is an equivariant $C^*$-correspondence from $(A, G, \alpha)$ to $(B, G, \beta)$.
			\item $\pi_\phi(A)V_\phi E$ is dense in $F_\phi$.
			\item For all $a\in A$, $\rho(a)=V_\phi^*\pi_\phi(a)V_\phi$.
			\item For all $g\in G$, $V_\phi\circ U_g=\tilde{U}_g\circ V_\phi$ 
		\end{enumerate}
		If $(F', \pi', V', U')$ is another quadruple satisfying conditions $(1)-(4)$, then there exists a $B$-linear unitary $W:F'\to F_\phi$ such that
		\begin{enumerate}
			\item $\pi_\phi(a)=W\pi'(a)W^*$ for all $a\in A$.
			\item $V_\phi=WV'$
			\item For all $g\in G$, $\widetilde{U}_g=WU'_gW^*$
		\end{enumerate}
		
	\end{theorem}
	
	\bibliographystyle{amsalpha} 
    \bibliography{bibliography}

@book{lance1995hilbert,
  title={Hilbert C*-Modules: A Toolkit for Operator Algebraists},
  author={Lance, E.C.},
  isbn={9780521479103},
  lccn={95154731},
  series={Lecture note series / London mathematical society},
  year={1995},
  publisher={Cambridge University Press}
}

@book{paulsen2002completely,
  title={Completely Bounded Maps and Operator Algebras},
  author={Paulsen, V.},
  isbn={9780521816694},
  lccn={02024624},
  series={Cambridge Studies in Advanced Mathematics},
  year={2002},
  publisher={Cambridge University Press}
}

@book{williams2007crossed,
  title={Crossed Products of $C^*$-Algebras},
  author={Williams, D.P.},
  isbn={9780821842423},
  lccn={2006047931},
  series={Mathematical surveys and monographs},
  year={2007},
  publisher={American Mathematical Society}
}

@book{blackadar2006operator,
  title={Operator Algebras: Theory of C*-Algebras and Von Neumann Algebras},
  author={Blackadar, B.},
  number={v. 13},
  isbn={9783540284864},
  lccn={2005934456},
  series={Encyclopaedia of Mathematical Sciences},
  year={2006},
  publisher={Springer}
}

@book{raeburn1998morita,
  title={Morita Equivalence and Continuous-Trace $C^*$-Algebras},
  author={Raeburn, I. and Williams, D.P.},
  isbn={9780821808603},
  lccn={98025838},
  series={Mathematical surveys and monographs},
  year={1998},
  publisher={American Mathematical Society}
}

@misc{eryuzlu2023exactsequencesenchiladacategory,
      title={Exact sequences in the Enchilada category}, 
      author={Menevse Eryuzlu and Steven Kaliszewski and John Quigg},
      year={2023},
      eprint={1909.00506},
      archivePrefix={arXiv},
      primaryClass={math.OA},
      url={https://arxiv.org/abs/1909.00506}, 
}

@misc{brix2024morphismscuntzpimsneralgebrascompletely,
      title={Morphisms of Cuntz-Pimsner algebras from completely positive maps}, 
      author={Kevin Aguyar Brix and Alexander Mundey and Adam Rennie},
      year={2024},
      eprint={2311.16600},
      archivePrefix={arXiv},
      primaryClass={math.OA},
      url={https://arxiv.org/abs/2311.16600}, 
}

@misc{echterhoff2005categoricalapproachimprimitivitytheorems,
      title={A Categorical Approach to Imprimitivity Theorems for C*-Dynamical Systems}, 
      author={Siegfried Echterhoff and S. Kaliszewski and John Quigg and Iain Raeburn},
      year={2005},
      eprint={math/0205322},
      archivePrefix={arXiv},
      primaryClass={math.OA},
      url={https://arxiv.org/abs/math/0205322}, 
}

@misc{echterhoff2000naturalityinducedrepresentations,
      title={Naturality and Induced Representations}, 
      author={Siegfried Echterhoff and S. Kaliszewski and John Quigg and Iain Raeburn},
      year={2000},
      eprint={math/0002039},
      archivePrefix={arXiv},
      primaryClass={math.OA},
      url={https://arxiv.org/abs/math/0002039}, 
}

@article{Buss_2012,
   title={A higher category approach to twisted actions on c* -algebras},
   volume={56},
   ISSN={1464-3839},
   url={http://dx.doi.org/10.1017/S0013091512000259},
   DOI={10.1017/s0013091512000259},
   number={2},
   journal={Proceedings of the Edinburgh Mathematical Society},
   publisher={Cambridge University Press (CUP)},
   author={Buss, Alcides and Meyer, Ralf and Zhu, Chenchang},
   year={2012},
   month=aug, pages={387–426} }

\end{document}